\documentclass[11pt]{amsart}
\usepackage{graphicx, amsmath, fullpage, amssymb, amsthm, amsfonts, color, algorithm,mathtools}
\usepackage{enumerate}
\mathtoolsset{showonlyrefs}
\newtheorem{theorem}{Theorem}[section]
\newtheorem{lemma}[theorem]{Lemma}

\theoremstyle{definition}
\newtheorem{definition}[theorem]{Definition}

\newtheorem{prop}[theorem]{Proposition}

\newtheorem{cor}[theorem]{Corollary}
\newtheorem{assumption}[theorem]{Assumption}

\usepackage[toc,page]{appendix} 

\theoremstyle{remark}
\newtheorem{remark}[theorem]{Remark}
\numberwithin{equation}{section}
\usepackage[colorlinks,
            linkcolor=red,
            anchorcolor=red,
            citecolor=blue
            ]{hyperref}
 
\newcommand{\R}{\mathbb{R}}
\newcommand{\C}{\mathbb{C}}
\newcommand{\N}{\mathcal{N}}
\newcommand{\NN}{\mathbb{N}}
\newcommand{\E}{\mathbb{E}}
\newcommand{\Id}{\operatorname{Id}}

\newcommand{\train}{\textnormal{train}}

\newcommand{\Tr}{\operatorname{Tr}}
\newcommand{\tr}{\operatorname{tr}}

\newcommand{\x}{\mathbf{x}}
\newcommand{\y}{\mathbf{y}}
\newcommand{\z}{\mathbf{z}}

\newcommand{\vmu}{\mbi{\mu}}
\newcommand{\vnu}{\mbi{\nu}}
\newcommand{\lameff}{\lambda_{\textnormal{eff}}}
\newcommand{\w}{\mbi{w}}

\renewcommand{\a}{\alpha}
\renewcommand{\aa}{\mbi{a}}
\renewcommand{\b}{\beta}
\newcommand{\eps}{\varepsilon}
\newcommand{\veps}{\boldsymbol \epsilon}

\renewcommand{\P}{\mathbb{P}}

\newcommand{\Var}{\operatorname{Var}}

\newcommand{\diag}{\operatorname{diag}}
\renewcommand{\vec}{\operatorname{vec}}

\newcommand{\limspec}{\operatorname{lim\,spec}}

\newcommand{\bbeta}{\pmb{\beta}}
\renewcommand{\Im}{\operatorname{Im}}
\renewcommand{\Re}{\operatorname{Re}}
\newcommand{\cE}{\mathcal{E}}

\def\mbi#1{\boldsymbol{#1}}
\def\v#1{\mbi{#1}} 

\begin{document}

\title{Deformed semicircle law and concentration of nonlinear random matrices for ultra-wide neural networks}

\author{Zhichao Wang}
\address{Department of Mathematics, University of California, San Diego, La Jolla, CA 92093}
\email{zhw036@ucsd.edu}


\author{Yizhe Zhu}
\address{Department of Mathematics, University of California, Irvine, Irvine, CA 92697}
\email{yizhe.zhu@uci.edu}

%
%

\begin{abstract}
In this paper, we investigate a  two-layer fully connected neural network of the form $
f(X)=\frac{1}{\sqrt{d_1}}\aa^\top \sigma\left(WX\right)$, where $X\in\R^{d_0\times n}$ is a deterministic data matrix, $W\in\R^{d_1\times d_0}$ and $\aa\in\R^{d_1}$ are random Gaussian weights, and $\sigma$ is a nonlinear activation function. We study the limiting spectral distributions of two empirical kernel matrices associated with $f(X)$: the empirical conjugate kernel (CK) and neural tangent kernel (NTK), beyond the linear-width regime ($d_1\asymp n$). We focus on the \textit{ultra-wide regime}, where the width $d_1$ of the first layer is much larger than the sample size $n$. Under appropriate assumptions on $X$ and $\sigma$, a deformed semicircle law emerges as $d_1/n\to\infty$ and $n\to\infty$. We first prove this limiting law for generalized sample covariance matrices with some dependency. To specify it for our neural network model, we provide a nonlinear Hanson-Wright inequality that is suitable for neural networks with random weights and Lipschitz activation functions. We also demonstrate non-asymptotic concentrations of the empirical CK and NTK around their limiting kernels in the spectral norm, along with lower bounds on their smallest eigenvalues. As an application, we show that random feature regression induced by the empirical kernel achieves the same asymptotic performance as its limiting kernel regression under the ultra-wide regime. This allows us to calculate the asymptotic training and test errors for random feature regression using the corresponding kernel regression. 
\end{abstract}
\date{\today}
\maketitle


\section{Introduction}\label{sec:intro}
Nowadays, deep neural networks have become one of the leading models in machine learning, and many theoretical results have been established to understand the training and generalization of neural networks. Among them, two kernel matrices are prominent in deep learning theory: \textit{Conjugate Kernel} (CK) \cite{neal1995bayesian,williams1997computing,rahimi2007random,cho2009kernel,daniely2016toward,poole2016exponential,schoenholz2017deep,lee2018deep,matthews2018gaussian} and \textit{Neural Tangent Kernel} (NTK) \cite{jacot2018neural,du2019gradienta,allen2019convergence}. The CK matrix defined in \eqref{eq:Y}, which has been exploited to study the generalization of random feature regression, is the Gram matrix of the output of the last hidden layer on the training dataset. While the NTK matrix,  defined in \eqref{eq:Hmatrixform},  is the Gram matrix of the Jacobian of the neural network with respect to training parameters, characterizing the performance of a wide neural network through gradient flows. Both are related to the kernel machine and help us explore the generalization and training process of the neural network. 

We are interested in the behaviors of CK and NTK matrices at random initialization. A recent line of work has proved that these two random kernel matrices will converge to their expectations when the width of the network becomes infinitely wide \cite{jacot2018neural,arora2019exact}. Although CK and NTK are usually referred to as these expected kernels in literature, we will always call CK and NTK the empirical kernel matrices in this paper, with a slight abuse of terminology. 

In this paper, we study the random CK and NTK matrices of a two-layer fully connected neural network with input data $X\in\R^{d_0\times n}$, given by $f: \mathbb R^{d_0\times n} \to \mathbb R^{n}$ such that
\begin{equation}\label{eq:1hlNN}
    f(X):=\frac{1}{\sqrt{d_1}}\aa^\top \sigma\left(WX\right),
\end{equation}where $W\in\R^{d_1\times d_0}$ is the weight matrix for the first layer, $\aa\in\R^{d_1}$ are the second layer weights, and $\sigma$ is a nonlinear activation function applied to the matrix $WX$ element-wisely. We assume that all entries of $\aa$ and $W$ are independently identically distributed by the standard Gaussian $\N(0,1)$. We will always view the  input data $X$ as a deterministic matrix (independent of the random weights in $\aa$ and $W$)  with certain assumptions. 

In terms of random matrix theory, we  study the difference between these two kernel matrices (CK and NTK) and their expectations with respect to random weights, showing both asymptotic and non-asymptotic behaviors of these differences as the width of the first hidden layer $d_1$ is growing faster than the number of samples $n$. As an extension of~\cite{fan2020spectra}, we prove that when $n/d_1\to 0$, the centered CK and NTK with appropriate normalization have the limiting eigenvalue distribution given by a deformed semicircle law, determined by the training data spectrum and the nonlinear activation function. To prove this global law, we further set up a limiting law theorem for centered sample covariance matrices with dependent structures and a nonlinear version of the Hanson-Wright inequality. These two results are very general, which makes them potentially applicable to different scenarios beyond our neural network model. For the non-asymptotic analysis, we establish concentration inequalities between the random kernel matrices and their expectations. As a byproduct, we provide lower bounds of the smallest eigenvalues of CK and NTK, which are essential for the global convergence of gradient-based optimization methods when training a wide neural network \cite{oymak2020toward,nguyen2020global,nguyen2021proof}. Because of the non-asymptotic results for kernel matrices, we can also describe how close the performances of the random feature regression and the limiting kernel regression are with a general dataset, which allows us to compute the limiting training error and generalization error for the random feature regression via its corresponding kernel regression in the ultra-wide regime.

\subsection{Nonlinear random matrix theory in neural networks}
Recently, the limiting spectra of CK and NTK at random initialization have received increasing attention from a random matrix theory perspective. Most of the papers focus on the \textit{linear-width regime} $d_1\propto n$, using both the moment method and Stieltjes transforms. Based on moment methods, \cite{pennington2017nonlinear} first computed the limiting law of the CK for two-layer neural networks with centered nonlinear activation functions, which is further described as a deformed Marchenko–Pastur law in \cite{peche2019note}. This result has been extended to sub-Gaussian weights and input data with real analytic activation functions by \cite{benigni2019eigenvalue}, even for multiple layers with some special activation functions.  Later, \cite{adlam2019random} generalized their results by adding a random bias vector in pre-activation and a more general input data matrix. Similar results for the two-layer model with a random bias vector and random input data were analyzed in \cite{piccolo2021analysis} by cumulant expansion. In parallel, by Stieltjes transform, \cite{louart2018random} investigated the CK of a one-hidden-layer network with general deterministic input data and Lipschitz activation functions via some deterministic equivalent. \cite{Liao2020ARM} further developed a deterministic equivalent for the Fourier feature map. With the help of the Gaussian equivalent technique and operator-valued free probability theory, the limit spectrum of NTK with one-hidden layer has been analyzed in \cite{adlam2020neural}. Then the limit spectra of CK and NTK of a multi-layer neural network with general deterministic input data have been fully characterized in \cite{fan2020spectra}, where the limiting spectrum of CK is given by the propagation of the Marchenko–Pastur map through the network, while the NTK is approximated by the linear combination of CK's of each hidden layer. \cite{fan2020spectra} illustrated that the \textit{pairwise approximate orthogonality} assumption on the input data is preserved in all hidden layers. Such a property is useful to approximate the expected CK and NTK. We refer to \cite{ge2021large} as a summary of the recent development in  nonlinear random matrix theory.

Most of the results in nonlinear random matrix theory focus on the case when $d_1$ is proportional to $n$ as $n\to\infty$. We build a random matrix result for both CK and NTK under the \textit{ultra-wide regime}, where $d_1/n\to \infty$ and $n\to\infty$. As an intrinsic interest of this regime, this exhibits the connection between wide (or overparameterized) neural networks and kernel learning induced by limiting kernels of CK and NTK. In this article, we will follow general assumptions on the input data and activation function in \cite{fan2020spectra} and study the limiting spectra of the centered and normalized CK matrix
\begin{equation}\label{eq:center_CK}
    \frac{1}{\sqrt{nd_1}}\left(Y^\top Y-\E[Y^\top Y]\right),
\end{equation}
where $Y:=\sigma(WX)$. Similar results for the NTK can be obtained as well. To complete the proofs, we  establish a nonlinear version of Hanson-Wright inequality, which has previously appeared in \cite{louart2018random,Liao2020ARM}. This nonlinear version is a generalization of the original Hanson-Wright inequality \cite{hanson1971bound,rudelson2013hanson,adamczak2015note}, and may have various applications in statistics, machine learning, and other areas. In addition, we also derive a deformed semicircle law for normalized sample covariance matrices without independence in columns. This result is of independent interest in random matrix theory as well.  

\subsection{General sample covariance matrices}
We observe that the  random matrix $Y\in\R^{d_1\times n}$ defined above has independent and identically-distributed rows. Hence, $Y^\top Y$ is a generalized sample covariance matrix. We first inspect a more general sample covariance matrix $Y$ whose rows are independent copies of some random vector $\y\in\R^n$. Assuming $n$ and $d_1$ both go to infinity but $n/d_1\to 0$, we aim to study the limiting empirical eigenvalue distribution of centered Wishart matrices in the form of \eqref{eq:center_CK} with certain conditions on $\y$. This regime is also related to the ultra-high dimensional setting in statistics \cite{qiu2021asymptotic}.

This regime has been studied for decades starting in \cite{bai1988convergence}, where $Y$ has i.i.d. entries and $\E[Y^\top Y]=d_1\Id$. In this setting, by the moment method, one can obtain the semicircle law. This normalized model also arises in quantum theory with respect to random induced states (see \cite{aubrun2012partial,aubrun2017alice,collins2018ppt}). The largest eigenvalue of such a normalized sample covariance matrix has been considered in \cite{chen2012convergence}. Subsequently, \cite{chen2015clt,li2016testing,yu2021central,qiu2021asymptotic} analyzed the fluctuations for the linear spectral statistics of this model and applied this result to hypothesis testing for the covariance matrix. A spiked model for sample covariance matrices in this regime was recently studied in \cite{feldman2021spiked}. This kind of semicircle law also appears in many other random matrix models. For instance, \cite{jiang2004limiting} showed this  limiting law for normalized sample correlation matrices. Also, the semicircle  law for centered sample covariance matrices has already been applied in machine learning: \cite{gamarnik2019stationary} controlled the generalization error of shallow neural networks with quadratic activation functions by the moments of this limiting semicircle law; \cite{granziol2020learning} derived a semicircle law of the fluctuation matrix between stochastic batch Hessian and the deterministic empirical Hessian of deep neural networks.

For general sample covariance, \cite{wang2014limiting} considered the form $Y=B X A^{1/2}$ with deterministic $A$ and $B$, where $X$ consists of i.i.d. entries with mean zero and variance one. The same result has been proved in \cite{bao2012strong} by generalized Stein's method. Unlike previous results, \cite{Xie2013LimitingSD} tackled the general case, only assuming $Y$ has independent rows with some deterministic covariance $\Phi_n$. Though this is similar to our model in Section \ref{sec:general_sample_covariance}, we will consider more general assumptions on each row of $Y$, which can be directly verified in our neural network models. 

\subsection{Infinite-width kernels and the smallest eigenvalues of empirical kernels} 
Besides the above asymptotic spectral fluctuation of \eqref{eq:center_CK}, we provide non-asymptotic concentrations of \eqref{eq:center_CK} in spectral norm and a corresponding result for the NTK. In the infinite-width networks, where $d_1\to\infty$ and $n$ are fixed, both CK and NTK will converge to their expected kernels. This has been investigated in \cite{daniely2016toward,schoenholz2017deep,lee2018deep,matthews2018gaussian} for the CK and \cite{jacot2018neural,du2019gradienta,allen2019convergence,arora2019exact,liang2020multiple} for the NTK. Such kernels are also called infinite-width kernels in literature. In this current work, we present the precise probability bounds for concentrations of CK and NTK around their infinite-width kernels, where the difference is of order $\sqrt{n/d_1}$. Our results permit more general activation functions and input data $X$ only with pairwise approximate orthogonality, albeit similar concentrations have been applied in \cite{avron2017random,song2019quadratic,adlam2020neural,montanari2020interpolation,hu2020surprising}.

A corollary of our concentration is the explicit lower bounds of the smallest eigenvalues of the CK and the NTK. Such extreme eigenvalues of the NTK have been utilized to prove the  global convergence of gradient descent algorithms of wide neural networks since the NTK governs the gradient flow in the training process, see, e.g., \cite{chizat2018lazy,du2019gradienta,arora2019fine,song2019quadratic,wu2019global,oymak2020toward,nguyen2020global,nguyen2021proof}. The smallest eigenvalue of NTK is also crucial for proving generalization bounds and memorization capacity in \cite{arora2019fine,montanari2020interpolation}. Analogous to Theorem 3.1 in \cite{montanari2020interpolation}, our lower bounds are given by the Hermite coefficients of the activation function $\sigma$. Besides, the lower bound of NTK for multi-layer ReLU networks is analyzed in \cite{nguyen2020tight}.

\subsection{Random feature regression and limiting kernel regression}
Another byproduct of our concentration results is to measure the difference of performance between random feature regression with respect to $\frac{1}{\sqrt{d_1}}Y$ and corresponding kernel regression when $d_1/n\to\infty$. Random feature regression can be viewed as the linear regression of the last hidden layer, and its performance has been studied in, for instance,  \cite{pennington2017nonlinear,louart2018random,mei2019generalization,Liao2020ARM,gerace2020generalisation,hu2020universality,lin2021causes,mei2021generalization,loureiro2021capturing} under the linear-width regime\footnote{This linear-width regime is also known as the high-dimensional regime, while our ultra-wide regime is also called a highly overparameterized regime in literature, see \cite{mei2019generalization}.}. In this regime, the CK matrix $\frac{1}{d_1}Y^\top Y$ is not concentrated around its expectation 
\begin{equation}\label{def:phi}
    \Phi:=\E_{\w}[\sigma(\w^\top X)^\top \sigma(\w^\top X)]
\end{equation}
under the spectral norm,
where $\w$ is the standard normal random vector in $\R^{d_0}$. But the limiting spectrum of CK is exploited to characterize the asymptotic performance and double descent phenomenon of random feature regression when $n,d_0,d_1\to\infty$ proportionally. Several works have also utilized this regime to depict the performance of the ultra-wide random network by letting $d_1/n\to\psi\in(0,\infty)$  first, getting the asymptotic performance and then taking $\psi\to\infty$ (see \cite{mei2019generalization,yang2021exact}). However, there is still a difference between this sequential limit and the ultra-wide regime. Before these results, random feature regression has already attracted significant attention in that it is a random approximation of the RKHS defined by population kernel function $K: \R^{d_0} \times \R^{d_0} \to \R$ such that
\begin{equation}\label{eq:def_K}
    K(\x,\z):=\mathbb E_{\w} [\sigma (\langle \w,\x \rangle ) \sigma (\langle \w,\z \rangle )],
\end{equation}
when width $d_1$ is sufficiently large \cite{rahimi2007random,bach2013sharp,rudi2016generalization,bach2017equivalence}. We point out that Theorem 9 of \cite{avron2017random} has the same order $\sqrt{n/d_1}$ of the approximation as ours, despite only for random Fourier features. 

In our work, the concentration between empirical kernel induced by $\frac{1}{d_1}Y^\top Y$ and the population kernel  matrix $K$ defined in \eqref{eq:def_K} for $X$ leads to the control of the differences of training/test errors between random feature regression and kernel regression, which were previously concerned by \cite{avron2017random,jacot2020implicit,montanari2020interpolation,mei2021generalization} in different cases. Specifically, \cite{jacot2020implicit} obtained the same kind of estimation but considered random features sampled from Gaussian Processes. Our results explicitly show how large width $d_1$ should be so that the random feature regression gets the same asymptotic performance as kernel regression \cite{mei2021generalization}. With these estimations, we can take the limiting test error of the kernel regression to predict the limiting test error of random feature regression as $n/d_1\to 0$ and $d_0,n\to \infty$. We refer \cite{liang2020just,liang2020multiple,liu2021kernel,mei2021generalization}, \cite[Section 4.3]{bartlett2021deep} and references therein for more details in high-dimensional kernel ridge/ridgeless regressions. We emphasize that the optimal prediction error of random feature regression in linear-width regime is actually achieved in the ultra-wide regime, which boils down to the limiting kernel regression, see \cite{mei2019generalization,mei2021generalization,yang2021exact,loureiro2021capturing}. This is one of the motivations for studying the ultra-wide regime and the limiting kernel regression.

In the end, we would like to mention the idea of spectral-norm approximation for the expected kernel $\Phi$, which helps us describe the asymptotic behavior of limiting kernel regression. For specific activation $\sigma$, kernel $\Phi$ has an explicit formula, see \cite{louart2018random,liao2018spectrum,Liao2020ARM}, whereas generally, it can be expanded in terms of the Hermite expansion of $\sigma$ \cite{pennington2017nonlinear,mei2019generalization,fan2020spectra}. Thanks to pairwise approximate orthogonality introduced in \cite[Definition 3.1]{fan2020spectra}, we can approximate $\Phi$ in the spectral norm for general deterministic data $X$. This pairwise approximate orthogonality defines how orthogonal is within different input vectors of $X$. With certain i.i.d. assumption on $X$, \cite{liang2020multiple} and \cite[Section 4.3]{bartlett2021deep}, where the scaling $d_0\propto n^\alpha$, for $\alpha\in(0,1]$, determined which degree of the polynomial kernel is sufficient to approximate $\Phi$. Instead, our theory leverages the approximate orthogonality among general datasets $X$ to obtain a similar approximation. Our analysis presumably indicates that the weaker orthogonality $X$ has, the higher degree of the polynomial kernel we need to approximate the kernel $\Phi$. 

\subsection{Preliminaries}

\subsubsection*{Notations}
We use $\tr(A)=\frac{1}{n}\sum_{i} A_{ii}$ as the normalized trace of a matrix $A\in\R^{n\times n}$ and $\Tr(A)=\sum_{i} A_{ii}$. Denote vectors by lowercase boldface. $\| A \|$ is the spectral norm for matrix $A$,  $\|A\|_F$ denotes the Frobenius norm, and $\|\x\|$ is the $\ell_2$-norm of any vector $\x$.  $A\odot B$ is the Hadamard product of two matrices, i.e., $(A\odot B)_{ij}=A_{ij}B_{ij}$. Let $\E_{\w}[\cdot]$ and $\Var_{\w}[\cdot]$ be the expectation and variance only with respect to random vector $\w$. Given any vector $\v v$, $\diag(\v v)$ is a diagonal matrix where the  main diagonal elements are given by $\v v$. $\lambda_{\min}(A)$ is the smallest eigenvalue of any Hermitian matrix $A$.

\bigskip 
Before stating our main results, we  describe our model with assumptions.
We first consider the output of the first hidden layer and empirical \textit{Conjugate Kernel} (CK):
\begin{equation}\label{eq:Y}
    Y:=\sigma( WX)\quad \text{and}\quad \frac{1}{d_1}Y^\top Y.
\end{equation}
Observe that the rows of matrix $Y$ are independent and identically distributed since only $W$ is random and $X$ is deterministic. Let the $i$-th row of $Y$ be $\y_i^\top$, for $1\le i\le d_1$. Then, we obtain a sample covariance matrix,
\begin{equation}\label{eq:sumr1}
    Y^\top Y=\sum_{i=1}^{d_1}\y_i\y_i^\top,
\end{equation}
which is the sum of $d_1$ independent rank-one random matrices in $\R^{n\times n}$. Let the second moment of any row $\y_i$ be \eqref{def:phi}. Later on, we will approximate $\Phi$ based on  the assumptions of input data $X$. 

Next, we define the empirical \textit{Neural Tangent Kernel} (NTK) for \eqref{eq:1hlNN}, denoted by $H\in\R^{n\times n}$. From Section~3.3 in \cite{fan2020spectra}, the $(i,j)$-th entry of $H$ can be explicitly written as 
\begin{align}\label{eq:Hentrywise}
 H_{ij}:=\frac{1}{d_1}\sum_{r=1}^{d_1}\left(\sigma(\w_r^\top\x_i)\sigma(\w_r^\top\x_j)+a_r^2\sigma'(\w_r^\top\x_i)\sigma'(\w_r^\top\x_j)\x_i^\top\x_j\right),\quad 1\le i,j\le n,   
\end{align}
where $\w_r$ is the $r$-th row of weight matrix $W$, $\x_i$ is the $i$-th column of matrix $X$, and $a_r$ is $r$-th entry of the output layer $\aa$. In the matrix form, $H$ can be written by 
\begin{align}\label{eq:Hmatrixform}
  H:=\frac{1}{d_1}\left(Y^\top Y+ (S^\top S)\odot(X^\top X)\right),  
\end{align}
where the $\alpha$-th column of $S$ is given by 
\begin{align}\label{eq:defcolumnS}
    \diag (\sigma'(W\x_\alpha))\aa, \quad \forall 1\leq \a\leq n.
\end{align}   

We introduce the following assumptions for the random weights, nonlinear activation function $\sigma$, and input data. These assumptions are basically carried on from \cite{fan2020spectra}.
\begin{assumption}\label{assump:W}
The entries of $W$ and $\aa$ are i.i.d.\ and distributed by $\N(0,1)$.
\end{assumption}

\begin{assumption}\label{assump:sigma}
Activation function $\sigma(x)$ is a Lipschitz function with the Lipschitz constant $\lambda_\sigma\in (0,\infty)$.  Assume that $\sigma$ is centered and normalized with respect to $\xi \sim \N(0,1)$ such that
\begin{align}
    \E[\sigma(\xi)]=0, \qquad &\E[\sigma^2(\xi)]=1.\label{eq:conditionsigma}
\end{align}
Define constants $a_\sigma$ and $b_\sigma \in \R$ by 
\begin{align}
    b_\sigma:=\E[\sigma'(\xi)], \qquad & a_\sigma:=\E[\sigma'(\xi)^2].\label{eq:conditionsigma1}
\end{align}
Furthermore, $\sigma$ satisfies \emph{either} of the following:
\begin{enumerate}
    \item $\sigma(x)$ is twice differentiable with $\sup_{x \in \R}
|\sigma''(x)| \leq \lambda_\sigma$, or
\item \label{assump:sigma_relu} $\sigma(x)$ is a piece-wise linear function defined by
\[\sigma(x)=
\begin{cases}
ax+b, &x>0,\\
cx+b, &x\le 0,
\end{cases}\]
for some constants $a,b,c\in\R$ such that \eqref{eq:conditionsigma} holds. 
\end{enumerate}
\end{assumption}

Analogously to \cite{hu2020surprising}, our Assumption \ref{assump:sigma} permits $\sigma$ to be the commonly used activation functions, including ReLU, Sigmoid, and Tanh, although we have to center and normalize the activation functions to guarantee \eqref{eq:conditionsigma}. Such normalized activation functions exclude some trivial spike in the limit spectra of CK and NTK  \cite{benigni2019eigenvalue,fan2020spectra}. The foregoing assumptions ensure our nonlinear Hanson-Wright inequality in the proof. As a future direction, going beyond Gaussian weights and Lipschitz activation functions may involve different types of concentration inequalities.  

Next, we present the conditions of the deterministic input data $X$ and the asymptotic regime for our main results.
Define the following $(\varepsilon, B)$-orthonormal property for our data matrix $X$. 
\begin{definition}
For given any $\eps, B>0$, matrix $X$ is {\bf $(\eps,B)$-orthonormal} if for any distinct columns $\x_{\alpha}, \x_{\beta}$ in $X$, we have
\[\big|\|\x_{\alpha}\|_2-1\big| \leq \eps,
\qquad \big|\|\x_{\beta}\|_2-1\big| \leq \eps, \qquad
\big|\x_{\alpha}^\top \x_{\beta}\big| \leq \eps,\]
	and also
	\[\sum_{\a=1}^n (\|\x_\a\|_2-1)^2 \leq B^2, \qquad \|X\| \leq B.\]
 
\end{definition}

\begin{assumption}\label{assump:asymptotics}
Let $n,d_0,d_1\to \infty$ such that
\begin{enumerate}[(a)]
\item $\gamma:=n/d_1 \to 0$;
\item $X$ is $\left(\varepsilon_n,B\right)$-orthonormal such that $n\varepsilon_n^4\to\ 0$ as $n\to \infty$; 
 
\item The empirical spectral distribution \ $\hat{\mu}_0$ of $X^\top X$
converges weakly to a fixed and non-degenerate probability distribution $\mu_0\not=\delta_0$ on $[0,\infty)$.
\end{enumerate}
\end{assumption}

In above (b), the $(\varepsilon_n,B)$-orthonormal property with $n\varepsilon^4_n=o(1)$ is a quantitative version of \textit{pairwise approximate orthogonality} for the column vectors of the data matrix $X\in \R^{d_0\times n}$. When $d_0\asymp n$, it holds, with high probability, for many random $X$ with independent columns $\x_{\alpha}$, including the anisotropic Gaussian vectors $\x_{\alpha}\sim \N(0,\Sigma)$ with $\tr(\Sigma)=1$ and $\|\Sigma\|\lesssim 1/n$,  vectors generated by Gaussian mixture models, and vectors satisfying the log-Sobolev inequality or convex Lipschitz concentration property. See \cite[Section 3.1]{fan2020spectra} for more details. Specifically, when $\x_{\alpha}$'s are independently sampled from the unit sphere $\mathbb{S}^{d_0-1}$,  $X$ is $\left(\varepsilon_n,B\right)$-orthonormal with high probability where $\varepsilon_n=O\Big(\sqrt{\frac{\log(n)}{n}}\Big)$ and $B=O(1)$ as $n\asymp d_0$. In this case, for any $\ell>2$, we have $n\varepsilon_n^\ell \to 0$. In our theory, we always treat $X$ as a deterministic matrix. However, our results also work for random input $X$ independent of weights $W$ and $\aa$ by conditioning on the high probability event that $X$ satisfies $(\varepsilon_n,B)$-orthonormal property. Unlike data vectors with independent entries, our assumption is promising to analyze real-world datasets \cite{loureiro2021capturing} and establish some $n$-dependent deterministic equivalents like \cite{Liao2020ARM}.

The following Hermite polynomials are crucial to the approximation of $\Phi$ in our analysis.
\begin{definition}[Normalized Hermite polynomials]\label{eq:hermitepolynomial}
The $r$-th normalized Hermite polynomial is given by 
\[ h_r(x)=\frac{1}{\sqrt {r!}} (-1)^r e^{x^2/2} \frac{d^r}{dx^r} e^{-x^2/2}.
\]
Here $\{h_r\}_{r=0}^{\infty}$ form an orthonormal basis of $L^2(\mathbb R, \Gamma)$, where $\Gamma$ denotes the standard Gaussian distribution. For $\sigma_1,\sigma_2\in L^2(\mathbb R, \Gamma)$,  the inner product is defined by 
\begin{align*}
    \langle \sigma_1,\sigma_2\rangle =\int_{-\infty}^{\infty} \sigma_1(x)\sigma_2(x) \frac{e^{-x^2/2}}{\sqrt{2\pi}}dx.
\end{align*}
Every function $\sigma \in L^2(\mathbb R, \Gamma)$ can be expanded as a Hermite polynomial expansion
\begin{align*}
    \sigma(x)=\sum_{r=0}^{\infty}\zeta_r(\sigma)h_r(x),
\end{align*}
where $\zeta_r(\sigma)$ is the $r$-th Hermite coefficient defined by 
\begin{align*}
    \zeta_r(\sigma):=\int_{-\infty}^{\infty} \sigma(x)h_r(x)\frac{e^{-x^2/2}}{\sqrt{2\pi}}dx.
\end{align*}
\end{definition}In the following statements and proofs,  we denote $\xi\sim \N(0,1)$. Then for any $k\in\NN$, we have
\begin{align}\label{eq:deflambdak}
    \zeta_k(\sigma)=\E[\sigma(\xi)h_k(\xi)].
\end{align}  
Specifically, $b_\sigma=\E[\sigma'(\xi)]=\E[\xi\cdot\sigma(\xi)]=\zeta_1(\sigma)$.  Let $f_k(x)=x^k$. We define the inner-product kernel random matrix $f_k(X^\top X)\in\R^{n\times n}$ by applying  $f_k$ entrywise to  $X^\top X$. Define a deterministic matrix 
 \begin{equation}\label{def:Phi0}
     \Phi_0:= \vmu\vmu^\top+\sum_{k=1}^3 \zeta_k(\sigma)^2f_k(X^\top X)+(1-\zeta_1(\sigma)^2-\zeta_2(\sigma)^2-\zeta_3(\sigma)^2)\Id,
 \end{equation}
 where the $\a$-th entry of $\vmu\in\R^n$ is $\sqrt{2}\zeta_2(\sigma)\cdot(\|\x_\a\|_2-1)$ for $1\le\a\le n$. We will employ $\Phi_0$ as an approximation of the population covariance $\Phi$ in \eqref{def:phi} in the spectral norm when $n\eps_n^4\to 0$.
 
 For any $n\times n$ Hermitian matrix $A_n$ with eigenvalues $\lambda_1,\dots, \lambda_n$, the empirical spectral distribution of $A$ is defined by 
 \begin{align*}
     \mu_{A_n}(x)=\frac{1}{n}\sum_{i=1}^n \delta_{\lambda_i}(x).
 \end{align*}
 We write $\limspec(A_n)=\mu$ if $\mu_{A_n}\to\mu$ weakly as $n\to\infty$.
 The main tool we use to study the limiting spectral distribution of a matrix sequence is  the Stieltjes transform defined as follows.
 \begin{definition}[Stieltjes transform]
 Let $\mu$ be a probability measure on $\R$. The Stieltjes transform of $\mu$ is a function $s(z)$ defined on the upper half plane $\mathbb C^+$ by
 \begin{align*}
     s(z)=\int_{\R} \frac{1}{x-z} d\mu(x).
 \end{align*}
 \end{definition}
 For any $n\times n$ Hermitian matrix $A_n$, the Stieltjes transform of the empirical spectral distribution of $A_n$ can be written as $\tr (A_n-z\Id)^{-1}$. We call $(A_n-z \Id)^{-1}$ the resolvent of $A_n$.
 
 \section{Main results}\label{sec:main}

\subsection{Spectra of the centered CK and NTK}
Our first result is a deformed semicircle law for the CK matrix. Denote by $\tilde{\mu}_0=(1-b_{\sigma})^2+b_{\sigma}^2 \mu_0$ the distribution of $(1-b_{\sigma}^2)+b_{\sigma}^2\lambda$ with $\lambda$ sampled from the distribution $\mu_0$. The limiting law of our centered and normalized CK matrix is depicted by $\mu_s\boxtimes \tilde{\mu}_0$, where $\mu_s$ is the standard semicircle law and the notation $\boxtimes$ is the \emph{free multiplicative convolution} in free harmonic analysis. For full descriptions of free independence and free multiplicative convolution, see \cite[Lecture 18]{nica2006lectures} and \cite[Section 5.3.3]{anderson2010introduction}. The free multiplicative convolution $\boxtimes$ was first introduced in \cite{voiculescu1987multiplication}, which later has many applications for products of asymptotic free random matrices. The main tool for computing free multiplicative convolution is the $S$-transform, invented by \cite{voiculescu1987multiplication}. $S$-transform was recently utilized to study the dynamical isometry of deep neural networks  \cite{pennington2017resurrecting,pennington2018emergence,xiao2018dynamical,hayase2021spectrum,collins2021asymptotic}. Some basic properties and intriguing examples for free multiplicative convolution with $\mu_s$ can also be found in \cite[Theorems 1.2, 1.3]{bai2010limiting}.

\begin{theorem}[Limiting spectral distribution for the conjugate kernel]\label{thm:law_DNN}
Suppose Assumptions \ref{assump:W}, \ref{assump:sigma} and \ref{assump:asymptotics} of the input matrix $X$ hold,  the empirical eigenvalue distribution of
\begin{equation}\label{eq:center_RF}
    \frac{1}{\sqrt{d_1n}}\left(Y^\top Y-\E[Y^\top Y]\right)
\end{equation}
 converges weakly to 
 \begin{align}\label{eq:expression_mu}
 \mu:=\mu_{s} \boxtimes \Big((1-b_\sigma^2)+b_\sigma^2 \cdot \mu_{0}\Big)=\mu_s\boxtimes \tilde{\mu}_0\end{align} almost surely as $n,d_0,d_1\to \infty$.
Furthermore, if $d_1\varepsilon_n^4\rightarrow 0$ as $n,d_0,d_1\to \infty$, then the empirical eigenvalue distribution of 
\begin{equation}\label{eq:center_RF2}
    \sqrt{\frac{d_1}{n}}\left(\frac{1}{d_1}Y^\top Y-\Phi_0\right)
\end{equation}
also  converges weakly to the probability measure $\mu $ almost surely, whose Stieltjes transform $m(z)$ is defined by 
\begin{equation}\label{eq:fixed_point1_CK}
    m(z)+\int_{\R} \frac{d \tilde \mu_0(x)}{z+\beta(z) x}=0
\end{equation}
 for each $z\in\C^+$, where  $\beta(z)\in\C^+$ is the unique solution to  
\begin{equation}\label{eq:fixed_point2_CK}
    \beta(z)+\int_{\R} \frac{xd \tilde \mu_0(x)}{z+\beta(z) x}=0. 
\end{equation}
\end{theorem}

Suppose that we additionally have $b_{\sigma}=0$, i.e. $\E[\sigma'(\xi)]=0$. In this case, our Theorem \ref{thm:law_DNN} shows that the limiting spectral distribution of \eqref{eq:center_CK} is the semicircle law, and from \eqref{eq:expression_mu}, the deterministic data matrix $X$ does not have an effect on the limiting spectrum. See Figure \ref{fig:law_semi} for a cosine-type $\sigma$ with $b_{\sigma}=0$. The only effect of the nonlinearity in $\mu$ is the coefficient $b_\sigma$ in the deformation $\tilde{\mu}_0$.
\begin{figure}[!ht]
\includegraphics[width=0.32\textwidth]{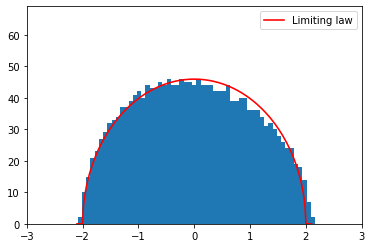}
\includegraphics[width=0.32\textwidth]{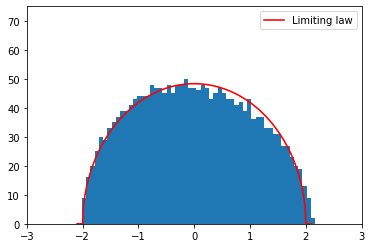}
\includegraphics[width=0.32\textwidth]{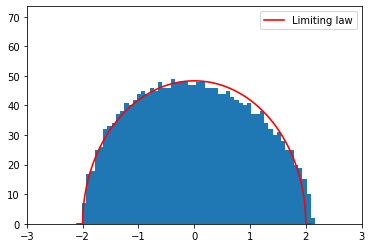}
\caption{\small Simulations for empirical eigenvalue distributions of \eqref{eq:center_RF2} and theoretical predication (red curves) of the limiting law $\mu$ where activation function $\sigma(x)\propto\cos (x)$ satisfies Assumption \ref{assump:sigma} with $b_{\sigma}=0$, and $X$ is a standard Gaussian random matrix. Dimension parameters are given by $n=1.9\times 10^3$, $d_0=2\times 10^3$ and $d_1=2\times 10^5$ (left); $n=2\times 10^3$, $d_0=1.9\times 10^3$ and $d_1=2\times 10^5$ (middle); $n=2\times 10^3$, $d_0=2\times 10^3$ and $d_1=2\times 10^5$ (right).}
\label{fig:law_semi}
\end{figure}

 Figure \ref{fig:law} is  a simulation of the limiting spectral distribution of CK with activation function $\sigma(x)$ given by $\arctan(x)$ after proper shifting and scaling. More simulations are provided in Appendix \ref{appendix:simulation} with different activation functions. The red curves are implemented by the self-consistent equations \eqref{eq:fixed_point1_CK} and \eqref{eq:fixed_point2_CK} in Theorem \ref{thm:law_DNN}. In Section~\ref{sec:general_sample_covariance}, we present general random matrix models with similar limiting eigenvalue distribution as $\mu$ whose  Stieltjes transform is also determined by \eqref{eq:fixed_point1_CK} and \eqref{eq:fixed_point2_CK}.
\begin{figure}[!ht]
\includegraphics[width=0.32\textwidth]{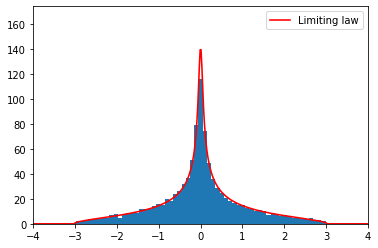}
\includegraphics[width=0.32\textwidth]{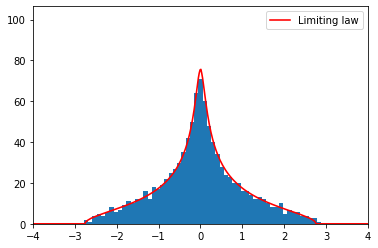}
\includegraphics[width=0.32\textwidth]{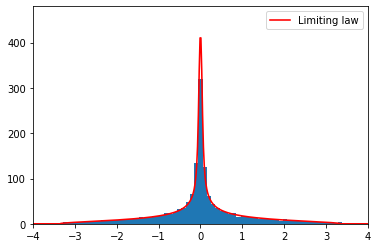}
\caption{\small Simulations for empirical eigenvalue distributions of \eqref{eq:center_RF2} and theoretical predication (red curves) of the limiting law $\mu$ where activation function $\sigma(x)\propto\arctan(x)$ satisfies Assumption \ref{assump:sigma} and $X$ is a standard Gaussian random matrix: $n=10^3$, $d_0=10^3$ and $d_1=10^5$ (left); $n=10^3$, $d_0=1.5\times 10^3$ and $d_1=10^5$ (middle); $n=1.5\times 10^3$, $d_0=10^3$ and $d_1=10^5$ (right).}
\label{fig:law}
\end{figure}

\bigskip 
Theorem \ref{thm:law_DNN} can be extended to the NTK model as well. Denote by
\begin{align}\label{eq:defPsi}
    \Psi:&=\frac{1}{d_1}\E[S^\top S]\odot(X^\top X)\in\R^{n\times n}.
\end{align}
As an approximation of $\Psi$ in the spectral norm, we define
\begin{align}
\Psi_0:&=\left(a_\sigma-\sum_{k=0}^2\eta_{k}^2(\sigma)\right)\Id+\sum_{k=0}^2\eta_{k}^2(\sigma) f_{k+1}( X^\top X),\label{eq:defPsi_0}
\end{align}
where $f_k$'s are defined in \eqref{def:Phi0},  $a_{\sigma}$ is defined in \eqref{eq:conditionsigma1}, and the $k$-th Hermite coefficient of $\sigma'$ is 
\begin{align}\label{eq:def_eta_k}
    \eta_k(\sigma):=\E[\sigma'(\xi)h_k(\xi)].
\end{align} 
Then, similar deformed semicircle law can be obtained for the empirical NTK matrix $H$. 
\begin{theorem}[Limiting spectral distribution of the NTK]\label{thm:law_NTK}
Under Assumptions \ref{assump:W}, \ref{assump:sigma} and \ref{assump:asymptotics} of the input matrix $X$, the empirical eigenvalue distribution of 
\begin{equation}\label{eq:NTK_1}
    \sqrt{\frac{d_1}{n}}\left(H-\E[H]\right)
\end{equation}
weakly converges to $\mu=\mu_{s} \boxtimes \Big((1-b_\sigma^2)+b_\sigma^2 \cdot \mu_{0}\Big)$ almost surely as $n,d_0,d_1\to \infty$ and $n/d_1\to 0$. 
Furthermore, suppose that $\varepsilon_n^4d_1\to 0$, then the empirical eigenvalue distribution of
\begin{equation}\label{eq:NTK_2}
    \sqrt{\frac{d_1}{n}}\left(H-\Phi_0-\Psi_0\right)
\end{equation}
weakly converges to $\mu$ almost surely, where $\Phi_0$ and $\Psi_0$ are defined in \eqref{def:Phi0} and \eqref{eq:defPsi_0}, respectively.
\end{theorem}

\subsection{Non-asymptotic estimations}\label{sec:non-asymp}

With  our nonlinear Hanson-Wright inequality (Corollary \ref{thm:quadratic_concentration}), we attain the following concentration bound on the CK matrix in the spectral norm.
\begin{theorem}\label{thm:spectralnorm} 
With Assumption \ref{assump:W}, assume $X$ satisfies
$ \sum_{i=1}^n (\|\x_i\|^2-1)^2\leq B^2$ for a constant $B\geq 0$,
 and $\sigma$ is $\lambda_{\sigma}$-Lipschitz with  $\E [\sigma(\xi)]=0$. Then with probability at least $1-4e^{-2n}$,
 \begin{align}\label{eq:d_1Y}
     \left\| \frac{1}{d_1}Y^{\top}Y-\Phi \right\|\leq C\left(\sqrt{\frac{n}{d_1}}+ \frac{n}{d_1}\right)\lambda_{\sigma}^2 \|X\|^2+32 B\lambda_{\sigma}^2 \|X\|\sqrt{\frac{n}{d_1}},
 \end{align}
 where $C>0$ is a universal constant.
\end{theorem}

\begin{remark}
Theorem \ref{thm:spectralnorm} ensures that the empirical spectral measure $\mu_n$ of the centered random matrix $\sqrt{\frac{d_1}{n}}\left(\frac{1}{d_1}Y^\top Y-\Phi\right)$ has a bounded support for all sufficiently large $n$. Together with the global law in Theorem~\ref{thm:law_DNN}, our concentration inequality \eqref{eq:d_1Y} is \emph{tight} up to a constant factor. Additionally, by the weak convergence of $\mu_n$ to $\mu$ proved in Theorem \ref{thm:law_DNN}, we can take a test function $x^2$ to obtain that
\[\int_{\R} x^2d\mu_n(x)\to \int_{\R} x^2d\mu(x),\quad  \text{i.e.,}\quad  \frac{\sqrt{d_1}}{n}\left\|\frac{1}{d_1}Y^\top Y-\Phi\right\|_F\rightarrow \left(\int_\R x^2d\mu(x)\right)^{\frac{1}{2}}\] 
almost surely, as $n,d_1\to \infty$ and $d_1/n\to\infty$.  Therefore, the fluctuation of $\frac{1}{d_1} Y^\top Y$ around $\Phi$ under the Frobenius norm is exactly of order  $n/\sqrt{d_1}$. 
\end{remark}

Based on the foregoing estimation, we have the following lower bound on  the smallest eigenvalue of the empirical conjugate kernel, denoted by $\lambda_{\min}\left(\frac{1}{d_1}Y^\top Y\right)$.
\begin{theorem}\label{thm:lowerKn}
Suppose Assumptions \ref{assump:W} and \ref{assump:sigma} hold and $\sigma$ is not a linear function, $X$ is $(\varepsilon_n, B)$-orthonormal. Then with  probability at least $1-4e^{-2n}$,
\begin{align*}
    &\lambda_{\min}\left(\frac{1}{d_1}Y^\top Y\right)
    \geq  1-\sum_{i=1}^3\zeta_i(\sigma)^2-C_B\eps_n^2\sqrt{n}-C\left(\sqrt{\frac{n}{d_1}}+\frac{n}{d_1}\right)\lambda_{\sigma}^2B^2,
\end{align*}
where $C_B$ is a constant depending on $B$.
In particular, if $\varepsilon_n^4 n=o(1), B=O(1), d_1=\omega(n)$, then with high probability,
\[\lambda_{\min}\left(\frac{1}{d_1}Y^\top Y\right)\geq 1-\sum_{i=1}^3\zeta_i(\sigma)^2-o(1).
\]
\end{theorem}
\begin{remark} A related result in
  \cite[Theorem 5]{oymak2020toward} assumed $\|\x_j\|=1$ for all $j\in [n]$, $\lambda_{\sigma}\leq B, |\sigma(0)|\leq B$, $d_1\geq C\log^2(n) \frac{n}{\lambda_{\min}(\Phi)}$ and obtained $
    \frac{1}{d_1} \lambda_{\min}(Y^\top Y)\geq  \lambda_{\min}(\Phi)-o(1). 
 $ We relax the assumption on the column vectors of $X$, and extend the range of $d_1$ down to $d_1=\Omega(n)$,   to guarantee that $ \frac{1}{d_1} \lambda_{\min}(Y^\top Y)$ is lower bounded by an absolute constant,  with an extra assumption that $\E[\sigma(\xi)]=0$. This assumption can always be satisfied by shifting the activation function with a proper constant. Our bound for $\lambda_{\min}(\Phi)$ is derived via Hermite  polynomial expansion, similar to \cite[Lemma 15]{oymak2020toward}. However, we apply an $\varepsilon$-net argument for concentration bound for $\frac{1}{d_1} Y^\top Y$ around $\Phi$, while a matrix Chernoff concentration bound with truncation was used in \cite[Theorem 5]{oymak2020toward}.
\end{remark}

\bigskip
Additionally, the concentration for the NTK matrix $H$ can be obtained in the next theorem. Recall that $H$ is defined by \eqref{eq:Hmatrixform} and the columns of $S$ are defined by \eqref{eq:defcolumnS} with Assumption~\ref{assump:W}.

\begin{theorem}\label{thm:NTK_concentration}
Suppose $d_1\geq \log n$, and $\sigma$ is $\lambda_{\sigma}$-Lipschitz. Then with  probability at least $1-n^{-7/3}$, 
\begin{align}\label{eq:Lupperboundd}
  \left  \|\frac{1}{d_1} (S^\top S-\E[S^\top S])\odot (X^\top X) \right\|\leq 10\lambda_{\sigma}^4 \|X\|^4\sqrt{\frac{\log n}{d_1}}.
\end{align}
Moreover, if the assumptions in Theorem \ref{thm:spectralnorm} hold, then with probability at least $1-n^{-7/3}-4e^{-2n}$,
\begin{align}\label{eq:YEY2}
    \| H-\E H\|\leq C\left(\sqrt{\frac{n}{d_1}}+ \frac{n}{d_1}\right)\lambda_{\sigma}^2 \|X\|^2+32 B\lambda_{\sigma}^2 \|X\|\sqrt{\frac{n}{d_1}}+10\lambda_{\sigma}^4 \|X\|^4\sqrt{\frac{\log n}{d_1}}.
\end{align}
\end{theorem}

\begin{remark}
  Compared to Proposition D.3 in \cite{hu2020surprising}, we assume $\aa$ is a Gaussian vector instead of a Rademacher random vector and attain a better bound. If $a_i\in \{+1,-1\}$, then one can apply matrix Bernstein inequality for the sum of bounded random matrices. In our case, the boundedness condition is not satisfied. Section S1.1 in \cite{adlam2020neural} applied matrix Bernstein inequality for the sum of bounded random matrices when $\aa$ is a Gaussian vector, but the boundedness condition does not hold in Equation (S7) of \cite{adlam2020neural}.
\end{remark}

Based on Theorem \ref{thm:NTK_concentration}, we get a lower bound for the smallest eigenvalue of the NTK.
\begin{theorem}\label{thm:NTK_Hlowerbound}
Under Assumptions~\ref{assump:W} and \ref{assump:sigma}, suppose that $X$ is $(\varepsilon_n, B)$-orthonormal, $\sigma$ is not a linear function, and $d_1\geq \log n$. Then with probability at least $1-n^{-7/3}$,
\begin{align*}
    \lambda_{\min}(H)\geq a_\sigma-\sum_{k=0}^2\eta_{k}^2(\sigma)-C_B\eps_n^4 n- 10\lambda_{\sigma}^4 B^4\sqrt{\frac{\log n}{d_1}},
\end{align*}
where $C_B$ is a constant depending only on $B$, and $\eta_k(\sigma)$ is defined in \eqref{eq:def_eta_k}.
In particular, if $\varepsilon_n^4n =o(1)$, $B=O(1)$, and $d_1=\omega(\log n)$, then with high probability,
\begin{align*}
    \lambda_{\min}(H)\geq \left( a_\sigma-\sum_{k=0}^2\eta_{k}^2(\sigma)\right)(1-o(1)).
\end{align*}
\end{theorem}
\begin{remark}
We relax the assumption  in \cite{nguyen2020tight} to $d_1=\omega(\log n)$ for the 2-layer case and our result is applicable beyond the ReLU activation function and to more general assumptions on $X$.  Our proof strategy is different from \cite{nguyen2020tight}. In \cite{nguyen2020tight}, the authors used the inequality $\lambda_{\min}((S^\top S)\odot(X^\top X))\geq \min_{i} \|S_i\|_2^2\cdot\lambda_{\min} (X^\top X)$ where $S_i$ is the $i$-th column of $S$. Then getting the  lower bound is reduced to show the concentration of the $2$-norm of the column vectors of $S$. Here we apply a matrix concentration inequality to $(S^\top S)\odot(X^\top X)$ and gain a  weaker assumption on $d_1$ to ensure the lower bound on $\lambda_{\min}(H)$.
\end{remark}

\begin{remark}\label{remark:linear}
    In Theorems \ref{thm:lowerKn} and \ref{thm:NTK_Hlowerbound}, we exclude the linear activation function. When $\sigma(x)=x$, it is easy to check both $\frac{1}{d_1}\lambda_{\min}(Y^\top Y)$ and $\lambda_{\min}(H)$ will trivially determined by $\lambda_{\min}(X^\top X)$, which can be vanishing. In this case, the lower bounds of the smallest eigenvalues of CK and NTK rely on the assumption of $\mu_0$ or the distribution of $X$. For instance, when the entries of $X$ are i.i.d. Gaussian random variables, $\lambda_{\min}(X^\top X)$ has been analyzed in \cite{silverstein1985smallest}.
\end{remark}

\subsection{Training and test errors for random feature regression}
We apply the results of the preceding sections to a two-layer neural network at random initialization defined in \eqref{eq:1hlNN}, to estimate the training errors and test errors with mean-square losses for random feature regression under the ultra-wide regime where $d_1/n\to\infty$ and $n\to\infty$. 
In this model, we take the random feature $\frac{1}{\sqrt{d_1}}\sigma(WX)$ and consider the regression with respect to $\v \theta\in\R^{d_1}$ based on
\[f_{\v \theta}(X):=\frac{1}{\sqrt{d_1}}\v \theta^\top \sigma\left(WX\right),\]
with training data $X\in\R^{d_0\times n}$ and training labels $\v y\in\R^{ n}$. Considering the ridge regression with ridge parameter $\lambda\ge 0$ and squared loss defined by
\begin{align}\label{def:lossfunction}
    L(\v \theta):= \|f_{\v \theta}(X)^\top-\v y\|^2+ \lambda\|\v \theta\|^2,
\end{align}
we can conclude that the minimization $ \hat{\v\theta}:=\arg\min_{\v\theta} L(\v\theta)$ has an explicit solution
\begin{equation}\label{eq:theta_hat}
    \hat{\v\theta}= \frac{1}{\sqrt{d_1}}Y\left(\frac{1}{d_1}Y^\top Y+\lambda\Id\right)^{-1}\v y,
\end{equation}where $Y=\sigma(WX)$ is defined in \eqref{eq:Y}. When $\sigma$ is nonlinear, by Theorem~\ref{thm:lowerKn}, it is feasible to take inverse in \eqref{eq:theta_hat} for any $\lambda\ge 0$. Hence, in the following results, we will focus on \emph{nonlinear} activation functions\footnote{As Remark~\ref{remark:linear} stated, when $\sigma(x)=x$, $\lambda_{\min}$ of CK may be possibly vanishing. To include the linear activation function, we can alternatively assume that the ridge parameter $\lambda$ is \emph{strictly} positive and focus on random feature \emph{ridge} regressions.}.
In general, the optimal predictor for this random feature with respect to \eqref{def:lossfunction} is
\begin{align}\label{eq:RFridge}
\hat{f}_{\lambda}^{(RF)}(\x):=\frac{1}{\sqrt{d_1}}\hat{\v\theta}^\top \sigma\left(W\x\right)=K_n(\x,X)(K_n(X,X)+\lambda\Id)^{-1}\v y,
\end{align}
where we define an empirical kernel $K_n(\cdot,\cdot):\R^{d_0} \times \R^{d_0} \to \R$ as
\begin{equation}\label{eq:K_n_def}
    K_n(\x,\z):=\frac{1}{d_1}\sigma(W\x)^\top\sigma(W\z)=\frac{1}{d_1}\sum_{i=1}^{d_1}\sigma (\langle \w_i,\x \rangle ) \sigma (\langle \w_i,\z\rangle ).
\end{equation}
The $n$-dimension row vector is given by 
\begin{equation}\label{eq:K_n_x_def}
K_n(\x,X)=[K_n(\x,\x_1),\ldots,K_n(\x,\x_n)],
\end{equation}
and the $(i,j)$ entry of $K_n(X,X)$ is defined by $K_n(\x_i,\x_j)$, for $1\le i,j\le n$.

Analogously, consider any kernel function $K(\cdot,\cdot): \R^{d_0} \times \R^{d_0} \to \R$. The optimal kernel predictor with a ridge parameter $\lambda\ge 0$ for the kernel ridge regression is given by (see \cite{rahimi2007random,avron2017random,liang2020just,jacot2020implicit,liu2021kernel,bartlett2021deep} for more details)
\begin{align}\label{eq:Kridge}
\hat{f}_{\lambda}^{(K)}(\x):=K(\x,X) (K(X,X)+\lambda\Id)^{-1} \v y,
\end{align}
where $K(X,X)$ is an $n\times n$ matrix such that its $(i,j)$ entry is  $K(\x_i,\x_j)$, and $K(\x,X)$ is a row vector in $\R^n$ similarly with \eqref{eq:K_n_x_def}.  We compare the characteristics of the two different predictors  $\hat{f}_{\lambda}^{(RF)}(\x)$ and $\hat{f}_{\lambda}^{(K)}(\x)$ when the kernel function $K$ is defined in \eqref{eq:def_K}. Denote the optimal predictors for random features and kernel $K$ on training data $X$ by
\begin{align*}
   \hat f_{\lambda}^{(RF)}(X)&=\left(\hat f_{\lambda}^{(RF)}(\x_1),\dots,\hat f_{\lambda}^{(RF)}(\x_n)\right)^\top,\\ 
    \hat f_{\lambda}^{(K)}(X)&=\left(\hat f_{\lambda}^{(K)}(\x_1),\dots, \hat f_{\lambda}^{(K)}(\x_n)\right)^\top,
\end{align*}
respectively. Notice that, in this case, $K(X,X)\equiv\Phi$ defined in \eqref{def:phi} and $K_n(X,X)$ is the random empirical CK matrix $\frac{1}{d_1}Y^\top Y$ defined in \eqref{eq:Y}.

We aim to compare the training and test errors for these two predictors  in ultra-wide random neural networks, respectively. Let \textit{training errors} of these two predictors be 
\begin{align}
 E_{\train}^{(K,\lambda)}:&=\frac{1}{n}\|\hat{f}_{\lambda}^{(K)}(X)-\v y\|_2^2=\frac{\lambda^2}{n}\|  (K(X,X)+\lambda\Id)^{-1}\v y\|^2, \label{eq:Etrain_Kn}\\
  E_{\train}^{(RF,\lambda)}:&=\frac{1}{n}\|\hat{f}_{\lambda}^{(RF)}(X)-\v y\|_2^2= \frac{\lambda^2}{n}\|  (K_n(X,X)+\lambda\Id)^{-1}\v y\|^2.\label{eq:Etrain_K}
\end{align}
In the following theorem, we show that, with high probability, the training error of the random feature regression model can be approximated by the corresponding kernel regression model with the same ridge parameter $\lambda\ge 0$ for ultra-wide neural networks.

\begin{theorem}[Training error approximation]\label{thm:train_diff}
Suppose  Assumptions \ref{assump:W}, \ref{assump:sigma} and \ref{assump:asymptotics} hold, and $\sigma$ is not a linear function. Then, for all large $n$, with  probability at least $1-4 e^{-2n}$,  
\begin{align}\label{eq:ETrain}
    \left| E_{\train}^{(RF,\lambda)}-E_{\train}^{(K,\lambda)}\right|\leq  \frac{C_1}{\sqrt{nd_1}}\left(\sqrt{\frac{n}{d_1}}+C_2\right) \|\v y\|^2,
\end{align}
where constants $C_1$ and $C_2$ only depend on $\lambda$, $B$ and $\sigma$.
\end{theorem}

Next, to investigate the test errors (or generalization errors), we introduce further assumptions on the data and the target function that we want to learn from training data. Denote the true regression function by $f^*: \mathbb R^{d_0}\to \R$. Then, the training labels are defined by 
\begin{equation}\label{def:vy}
    \v y=f^*(X)+\boldsymbol \epsilon\quad\text{and}\quad f^*(X)=(f^*(\x_1),\ldots,f^*(\x_n))^\top,
\end{equation}
where $\veps\in\R^n$ is the training label noise. For simplicity, we further impose the following assumptions, analogously to \cite{lin2021causes}.

\begin{assumption}\label{assump:target}
Assume that the target function is a linear function $f^*(\x)=\langle\bbeta^*,\x\rangle$, where random vector satisfies $\bbeta^*\sim\N(0,\sigma^2_{\bbeta}\Id)$.
Then, in this case, the  training label vector is given by $\v y=X^\top\bbeta^*+\veps$ where $\boldsymbol \epsilon\sim \N (\mathbf{0},\sigma_{\veps}^2\Id)$ independent with $\bbeta^*\in\R^{d_0}$.
\end{assumption}
\begin{assumption}\label{assump:testdata}
Suppose that training dataset $X=[\x_1,\ldots,\x_n]\in\R^{d_0\times n}$ satisfies $(\eps_n,B)$-orthonormal condition with $n\eps_n^4=o(1)$, and a test data $\x\in\R^{d_0}$ is independent with $X$ and $\v y$ such that $\tilde X:=[\x_1,\ldots,\x_n,\x]\in\R^{d_0\times (n+1)} $ is also $(\eps_n,B)$-orthonormal. For convenience, we further assume the population covariance of the test data is $\E_\x[\x\x^\top ]=\frac{1}{d_0}\Id$.
\end{assumption}
\begin{remark}
Our Assumption \ref{assump:testdata} of  the test data $\x$ ensures the same statistical behavior as training data in $X$, but we do not have any explicit assumption of the distribution of $\x$. It is promising to adopt such assumptions to handle statistical models with real-world data \cite{liao2018spectrum,Liao2020ARM}. Besides, it is possible to extend our analysis to general population covariance for $\E_\x[\x\x^\top ]$.
\end{remark}

For any predictor $\hat f$, define the \textit{test error} (generalization error) by 
\begin{align}\label{eq:def_testerror}
    \mathcal L(\hat f):=\E_{\x}[|\hat f(\x)-f^*(\x)|^2].
\end{align} 
We first present the following approximation of the test error of a random feature predictor via its corresponding kernel predictor.
\begin{theorem}[Test error approximation]\label{thm:test_diff}
Suppose that Assumptions \ref{assump:W}, \ref{assump:sigma}, \ref{assump:target} and \ref{assump:testdata} hold, and $\sigma$ is not a linear function. 
Then, for any $\varepsilon\in(0,1/2)$, the difference of test errors satisfies
\begin{equation}\label{eq:test_diff}
     \left| \mathcal L(\hat{f}_{\lambda}^{(RF)}(\x))-\mathcal L (\hat{f}_{\lambda}^{(K)}(\x))\right| =o\left(\left(n/d_1\right)^{\frac{1}{2}-\varepsilon}\right),
\end{equation}
with probability $1-o(1)$, when $n/d_1\to 0 $ and $n\to \infty$. 
\end{theorem}

Theorems \ref{thm:train_diff} and \ref{thm:test_diff} verify that the random feature regression achieves
the same asymptotic errors as the kernel regression, as long as $n/d_1\to\infty$. This is closely related to \cite[Theorem 1]{mei2021generalization} with different settings. Based on that, we can compute the asymptotic training and test errors for the random feature model by calculating the corresponding quantities for the kernel regression in the ultra-wide regime where $n/d_1\to 0$. 
\begin{theorem}[Asymptotic training and test errors]\label{thm:limit_error} 
Suppose  Assumptions \ref{assump:W} and \ref{assump:sigma} hold,  and $\sigma$ is not a linear function. 
Suppose the target function $f^*$, training data $X$ and test data $\x\in\R^{d_0}$ satisfy Assumptions \ref{assump:target} and \ref{assump:testdata}. For any $\lambda\ge 0$, let the effective ridge parameter be
\begin{equation}\label{eq:lameff}
    \lameff(\lambda,\sigma):=\frac{1+\lambda-b_\sigma^2}{b_\sigma^2}.
\end{equation}
If the training data has some limiting eigenvalue distribution $\mu_0=\limspec X^\top X$ as $n\to \infty$ and $n/d_0\to \gamma\in(0,\infty)$, then when $n/d_1\to 0$ and $n\to\infty$, the training error satisfies
\begin{equation}\label{eq:train_limit}
    E_{\train}^{(RF,\lambda)}\xrightarrow[]{\P }\frac{\sigma^2_{\bbeta}\lambda^2}{\gamma b_\sigma^4}\mathcal{V}_K(\lameff(\lambda,\sigma))+\frac{\sigma^2_{\veps}\lambda^2}{\gamma (1+\lambda-b_\sigma^2)^2}\left(\mathcal{B}_K(\lameff(\lambda,\sigma))-1+\gamma\right),
\end{equation}
and the test error satisfies
\begin{equation}\label{eq:test_limit}
     \mathcal L(\hat{f}_{\lambda}^{(RF)}(\x)) \xrightarrow[]{\P } \sigma^2_{\bbeta}\,\mathcal{B}_{K}(\lameff(\lambda,\sigma))+\sigma_{\veps}^2\,\mathcal{V}_K(\lameff(\lambda,\sigma)),
\end{equation}
where the bias and variance functions are defined by
\begin{align}
    \mathcal{B}_{K}(\nu):=~&(1-\gamma)+\gamma\nu^2 \int_{\R} \frac{1}{(x+\nu)^2}d\mu_0(x),\\
    \mathcal{V}_{K}(\nu):=~&\gamma \int_{\R} \frac{x}{(x+\nu)^2}d\mu_0(x).
\end{align}
\end{theorem}

We emphasize that in the proof of Theorem \ref{thm:limit_error}, we also get $n$-dependent deterministic equivalents for training/test errors of the kernel regression to approximate the performance of random feature regression. This is akin to \cite[Theorem 3]{Liao2020ARM} and \cite[Theorem 4.13]{bartlett2021deep}, but in different regimes. In the following Figure \ref{fig:krr}, we present implementations of test errors for random feature regressions on standard Gaussian random data and their limits \eqref{eq:test_limit}. For simplicity, we fix $n,d_0$, only let $d_1\to\infty$, and use empirical spectral distribution of $X^\top X$ to approximate $\mu_0$ in $\mathcal{B}_{K}(\lameff(\lambda,\sigma))$ and $\mathcal{V}_{K}(\lameff(\lambda,\sigma))$, which is actually the $n$-dependent deterministic equivalent. However, for Gaussian random matrix $X$, $\mu_0$ is actually a Marchenko-Pastur law with ratio $\gamma$, so  $\mathcal{B}_{K}(\lameff(\lambda,\sigma))$ and $\mathcal{V}_{K}(\lameff(\lambda,\sigma))$ can be computed explicitly according to \cite[Definition 1]{lin2021causes}.
\begin{figure}[!ht]  
\centering
\begin{minipage}[t]{0.45\linewidth}
\centering
{\includegraphics[width=0.99\textwidth]{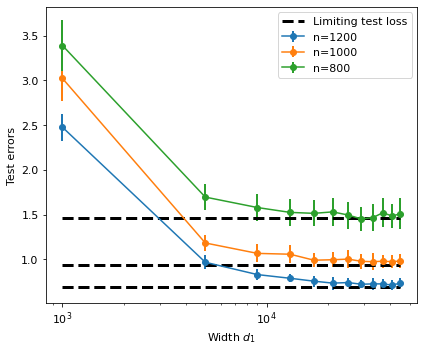}}  
\end{minipage}
\begin{minipage}[t]{0.45\linewidth}
\centering 
{\includegraphics[width=0.99\textwidth]{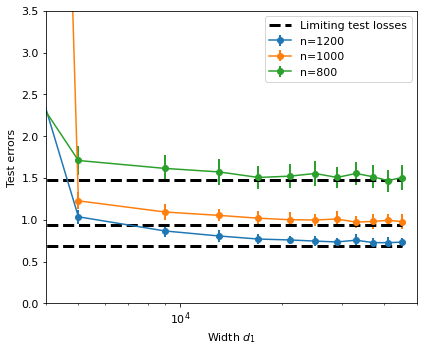}}
\end{minipage}  
 \caption{\small Simulations for the test errors of random feature regressions with centered Gaussian random matrix as input $X$ and regularization parameter $\lambda=10^{-3}$ (left) and $\lambda=10^{-6}$ (right). Here, the activation function $\sigma$ is a re-scaled Sigmoid function, $\sigma_{\veps}=1$ and $\sigma_{\bbeta}=2$. We fix $d_0=500$, varying values of sample sizes $n$ and widths $d_1$. Test errors in solid lines with error bars are computed using an independent test set of size 5000. We average our results over 50 repetitions. Limiting test errors in black dash lines are  computed by \eqref{eq:test_limit}, and we take $\mu_0$ to be the corresponding Marchenko–Pastur distributions.}   
\label{fig:krr}  
\end{figure}

\begin{remark}[Implicit regularization]
 For nonlinear $\sigma$, the effective ridge parameter \eqref{eq:lameff} can be viewed as an inflated ridge parameter since $b_\sigma^2\in [0,1)$ and $\lameff>\lambda\ge 0$. This $\lameff$ leads to \emph{implicit regularization} for our random feature and kernel ridge regressions  even for the ridgeless regression with $\lambda=0$ \cite{liang2020just,montanari2020interpolation,jacot2020implicit,bartlett2021deep}. This effective ridge parameter $\lameff$ also shows the effect of the nonlinearity in the random feature and kernel regressions induced by ultra-wide neural networks.
\end{remark}

\begin{remark}
For convenience, we only consider the linear target function $f^*$, but in general, the above theorems can also be obtained for nonlinear target functions, for instance, $f^*$ is a nonlinear single-index model. Under $(\eps_n,B)$-orthonormal assumption with $n\eps_n^4\to 0$, our expected  kernel $K(X,X)\equiv\Phi$ is approximated in terms of
\begin{equation}\label{eq:K_equi}
    \limspec K(X,X)=\limspec \left(b_\sigma^2 X^\top X+(1-b^2_\sigma)\Id\right),
\end{equation}
whence, this  kernel regression can only learn linear functions. So if $f^*$ is nonlinear, the limiting test error should be decomposed into the linear part as \eqref{eq:test_limit} and the nonlinear component as a noise \cite[Theorem 4.13]{bartlett2021deep}. For more conclusions of this kernel machine, we refer to \cite{liang2020just,liang2020multiple,liu2021kernel,mei2021generalization}.
\end{remark}
\begin{remark}[Neural tangent regression]
In parallel to the above results, we can obtain a similar analysis of the limiting training and test errors for random feature regression  in \eqref{eq:RFridge} with empirical NTK given by either $K_n(X, X)=\frac{1}{d_1}(S^\top S)\odot(X^\top X)$ or $K_n(X, X)=H$. This random feature regression also refers to \emph{neural tangent regression} \cite{montanari2020interpolation}. With the help of our concentration results in Theorem \ref{thm:NTK_concentration} and the lower bound of the smallest eigenvalues in Theorem \ref{thm:NTK_Hlowerbound}, we can directly extend the above Theorems~\ref{thm:train_diff}, \ref{thm:test_diff} and \ref{thm:limit_error} to this neural tangent regression. We omit the proofs in these cases and only state the results as follows.

If $K_n(X, X)=\frac{1}{d_1}(S^\top S)\odot(X^\top X)$ with expected kernel $K(X, X)=\Psi$ defined by \eqref{eq:defPsi}, the limiting training and test errors of this neural tangent regression can be approximated by the kernel regression with respect to $\Psi$, as long as $d_1=\omega(\log n)$. Analogously to \eqref{eq:K_equi}, we have an additional approximation
\begin{equation}\label{eq:K_equi_psi}
    \limspec \Psi=\limspec \left(b_\sigma^2 X^\top X+(a_\sigma-b^2_\sigma)\Id\right).
\end{equation}
Under the same assumptions of Theorem~\ref{thm:limit_error} and replacing $n/d_1\to 0$ with  $d_1=\omega(\log n)$, we can conclude that the test error of this neural tangent regression has the same limit as \eqref{eq:test_limit} but changing the effective ridge parameter \eqref{eq:lameff} into $\lameff(\lambda,\sigma)=\frac{a_\sigma+\lambda-b_\sigma^2}{b_\sigma^2}$. This result is akin to \cite[Corollary 3.2]{montanari2020interpolation} but permits more general assumptions on $X$. The limiting training error of this neural tangent regression can be obtained by slightly modifying the coefficient in \eqref{eq:train_limit}.

Similarly, if $K_n(X, X)=H$ defined by \eqref{eq:Hmatrixform} possesses an expected kernel $K(X,X)=\Phi+\Psi$, this neural tangent regression in \eqref{eq:RFridge} is close to kernel regression \eqref{eq:Kridge} with kernel
\[K(\x,\z)=\mathbb E_{\w} [\sigma (\w^\top\x) \sigma (\w^\top\x)]+\mathbb E_{\w} [\sigma' (\w^\top\x ) \sigma'(\w^\top\x )]\x^\top\z,\]
under the ultra-wide regime, $n/d_1 \to 0$. Combining \eqref{eq:K_equi} and \eqref{eq:K_equi_psi}, Theorem~\ref{thm:limit_error} can directly be extended to this neural tangent regression but replacing \eqref{eq:lameff} with $\lameff(\lambda,\sigma)=\frac{a_\sigma+1+\lambda-2b_\sigma^2}{2b_\sigma^2}$.
Section 6.1 of \cite{adlam2020neural} also calculated this limiting test error when data $X$ is isotropic Gaussian.
\end{remark} 

\subsubsection*{Organization of the paper} The remaining parts of the paper are structured as follows.
In Section \ref{sec:nonlinearHW}, we first provide a nonlinear Hanson-Wright inequality as a concentration tool for our spectral analysis. Section \ref{sec:general_sample_covariance} gives a general theorem for the limiting spectral distributions of generalized centered sample covariance matrices. We prove the limiting spectral distributions  for the empirical CK and NTK matrices (Theorem \ref{thm:law_DNN} and Theorem \ref{thm:law_NTK}) in Section \ref{sec:1hlNN}. Non-asymptotic estimates  in subsection~\ref{sec:non-asymp} are proved in Section \ref{sec:nonasymptotic}. In Section \ref{sec:generalizationerror}, we justify the asymptotic results of the training and test errors for the random feature model (Theorem \ref{thm:train_diff} and Theorem \ref{thm:test_diff}). 
Auxiliary lemmas and additional simulations are included in Appendices.

\section{A non-linear Hanson-Wright inequality}\label{sec:nonlinearHW}
We give an improved version of Lemma 1 in \cite{louart2018random} with a simple proof based on a Hanson-Wright inequality for random vectors with dependence \cite{adamczak2015note}. This serves as the concentration tool for us to prove the deformed semicircle law in Section~\ref{sec:1hlNN} and provide bounds on extreme eigenvalues in Section~\ref{sec:nonasymptotic}. We first define some concentration properties for random vectors.

\begin{definition}[Concentration property]
Let $X$ be a random vector in $\mathbb R^n$. We say $X$ has the \textit{$K$-concentration property} with constant $K$ if for any $1$-Lipschitz function $f: \mathbb R^n \to \R$, we have $\mathbb E|f(X)|<\infty$ and for any $t>0$,
\begin{align}\label{eq:Kconvex}
    \P (|f(X)-\E f(X)|\geq t)\leq 2\exp(-t^2/K^2).
\end{align}
\end{definition}
There are many distributions of random vectors satisfying $K$-concentration property, including uniform random vectors on the sphere, unit ball, hamming  or continuous cube, uniform random permutation, etc. See \cite[Chapter 5]{vershynin2018high} for more details.

\begin{definition}[Convex concentration property]
Let $X$ be a random vector in $\mathbb R^n$. We say $X$ has the \textit{$K$-convex concentration property} with the constant $K$ if for any $1$-Lipschitz convex function $f: \mathbb R^n \to \R$, we have $\mathbb E|f(X)|<\infty$ and for any $t>0$,
\begin{align*}
    \P (|f(X)-\E f(X)|\geq t)\leq 2\exp(-t^2/K^2).
\end{align*}
\end{definition}

We will apply the following result from \cite{adamczak2015note}  to the nonlinear setting.

\begin{lemma}[Theorem 2.5 in \cite{adamczak2015note}]\label{eq:HWinequality}
Let $X$ be a mean zero random vector in $\R^n$. If $X$ has the $K$-convex concentration property, then for any $n\times n$ matrix $A$ and any $t>0$,
\begin{align*}
    \P ( |X^\top AX -\E (X^\top AX)| \geq t) \leq 2\exp \left(-\frac{1}{C} \min \left\{ \frac{t^2}{2K^4\|A\|_F^2}, \frac{t}{K^2\|A\|}\right\}\right) 
\end{align*}
for some universal constant $C>1$.
\end{lemma}

\begin{theorem}
Let $\w\in \mathbb R^{d_0}$ be a  random vector with $K$-concentration property, $X=(\x_1,\dots, \x_n)\in \R^{d_0\times n}$ be a deterministic matrix. Define  $\y=\sigma(\w^{\top} X)^{\top}$, where $\sigma$ is $\lambda_{\sigma}$-Lipschitz, and $\Phi=\mathbb E \y \y ^\top$.
Let $A$ be an $n\times n$ deterministic matrix. 
\begin{enumerate}
\item  If $\mathbb E[\y]=0$, for any $t>0$,
\begin{align}\label{eq:zeromeaninequality1}
   \P \left(|\y^{\top} A \y-\Tr A\Phi|\geq t\right)
   \leq  2\exp \left(-\frac{1}{C} \min \left\{ \frac{t^2}{2K^4\lambda_{\sigma}^4 \|X\|^4\|A\|_F^2}, \frac{t}{K^2\lambda_{\sigma}^2 \|X\|^2\|A\|}\right\}\right),
\end{align}
where $C>0$ is an absolute constant.
    \item 
If $\mathbb E[\y]\not=0$, for any $t>0$, 
\begin{align*} 
    \mathbb P \left(|\y^{\top} A \y-\Tr A\Phi| >t \right)
    \leq ~&2\exp \left(-\frac{1}{C} \min \left\{ \frac{t^2}{4K^4 \lambda_{\sigma}^4 \|X\|^4\|A\|_F^2}, \frac{t}{K^2\lambda_{\sigma}^2 \|X\|^2\|A\|}\right\}\right) \\
    &+2\exp \left( -\frac{t^2}{16K^2\lambda_{\sigma}^2\|X\|^2\|A\|^2 \|\E \y\|^2  }\right). \notag
\end{align*}
for some constant $C>0$.
\end{enumerate}
\end{theorem}

\begin{proof}
Let $f$ be any $1$-Lipschitz convex function.
Since $\y=\sigma (\w^\top X)^\top $, $f(\y)=f(\sigma(\w^\top X)^\top) $ is a $\lambda_{\sigma} \|X\|$-Lipschitz function of $\w$. Then by the Lipschitz concentration property of $\w$ in \eqref{eq:Kconvex}, we obtain
\begin{align*}
    \P (|f(\y) -\E f(\y)|\geq t)\leq 2\exp\left(-\frac{t^2}{K^2\lambda_{\sigma}^2 \|X\|^2} \right).
\end{align*}
Therefore, $\y$ satisfies the $K\lambda_{\sigma}\|X\|$-convex concentration property. Define $\tilde{f}(\x)=f(\x-\mathbb E\y)$, then $\tilde{f}$ is also a convex $1$-Lipschitz function and $\tilde{f}(\y)=f(\y-\E\y)$. Hence $\tilde{\y}:=\y-\E \y$ also satisfies the $K\lambda_{\sigma}\|X\|$-convex concentration property. Applying Theorem \ref{eq:HWinequality} to $\tilde{\y}$, we have for any $t>0$,
\begin{align}\label{eq:bound1}
    \P ( | {\tilde{\y}}^\top A \tilde\y -\E (\tilde{\y}^\top A\tilde{\y})| \geq t) \leq 2\exp \left(-\frac{1}{C} \min \left\{ \frac{t^2}{ 2K^4\lambda_{\sigma}^4 \|X\|^4\|A\|_F^2}, \frac{t}{K^2\lambda_{\sigma}^2  \|X\|^2\|A\|}\right\}\right).
\end{align}
Since $\E\tilde {\y}=0$, the inequality above implies \eqref{eq:zeromeaninequality1}.   Note that 
\begin{align*}
    {\tilde{\y}}^\top A \tilde\y -\E (\tilde{\y}^\top A\tilde{\y})=(\y^\top A \y -\Tr A\Phi)- \tilde{\y}^{\top}A\E \y -\E \y ^{\top} A\tilde{\y},
\end{align*}
Hence,
\begin{align}\label{eq:sum2}
  \y^\top A \y -\Tr A\Phi&= ({\tilde{\y}}^\top A \tilde\y -\E (\tilde{\y}^\top A\tilde{\y})) + (\y-\E  \y)^{\top}(A+A^{\top})\E \y \notag\\
  &=({\tilde{\y}}^\top A \tilde\y -\E (\tilde{\y}^\top A\tilde{\y}))+ (\y^{\top}(A+A^{\top})\E \y-\E\y^{\top}(A+A^{\top})\E \y).
\end{align}
Since $\y^{\top}(A+A^{\top})\E \y$ is a  $(2\|A\| \|\E \y\| \|X\| \lambda_{\sigma})$-Lipschitz function of $\w$, by the Lipschitz concentration property of $\w$, we have 
\begin{align}\label{eq:bound2}
    \P ( |(\y-\E  \y)^{\top}(A+A^{\top})\E \y|\geq t  )\leq  2\exp \left( -\frac{t^2}{4K^2(\|A\| \|\E \y\| \|X\| \lambda_{\sigma})^2}\right).
\end{align}
Then combining \eqref{eq:bound1}, \eqref{eq:sum2}, and \eqref{eq:bound2}, we have 
\begin{align}
    \P (|\y^\top A \y -\Tr A\Phi|\geq t)
    &\leq  \P ( | {\tilde{\y}}^\top A \tilde\y -\E (\tilde{\y}^\top A\tilde{\y})| \geq t/2)+  \P ( |(\y-\E  \y)^{\top}(A+A^{\top})\E \y|\geq t/2) \notag \\
    &\leq 2\exp \left(-\frac{1}{2C} \min \left\{ \frac{t^2}{4K^4 \lambda_{\sigma}^4 \|X\|^4\|A\|_F^2}, \frac{t}{K^2\lambda_{\sigma}^2 \|X\|^2\|A\|}\right\}\right) \notag\\
    &\quad + 2\exp \left( -\frac{t^2}{16K^2\lambda_{\sigma}^2\|X\|^2\|A\|^2 \|\E \y\|^2  }\right). \notag
\end{align}
This finishes the proof.
\end{proof}

Since the Gaussian random vector $\w\sim \N(0,I_{d_0})$ satisfies the $K$-concentration inequality with $K=\sqrt{2}$ (see for example \cite{boucheron2013concentration}), we have the following corollary.

\begin{cor}\label{thm:quadratic_concentration}
Let $\w\sim \N(0, I_{d_0})$, $X=(\x_1,\dots, \x_n)\in \R^{d_0\times n}$ be a deterministic matrix. Define  $\y=\sigma(\w^{\top} X)^{\top}$, where $\sigma$ is $\lambda_{\sigma}$-Lipschitz, and $\Phi=\mathbb E \y \y ^\top$.
Let $A$ be an $n\times n$ deterministic matrix. 
\begin{enumerate}
\item  If $\mathbb E[\y]=0$, for any $t>0$,
\begin{align} \label{eq:zero-meanbound}
   ~~~\P \left(|\y^{\top} A \y-\Tr A\Phi|\geq t\right)\leq  2\exp \left(-\frac{1}{C} \min \left\{ \frac{t^2}{4\lambda_{\sigma}^4 \|X\|^4\|A\|_F^2}, \frac{t}{\lambda_{\sigma}^2 \|X\|^2\|A\|}\right\}\right)
\end{align}
for some absolute constant $C>0$.
    \item 
If $\mathbb E[\y]\not=0$, for any $t>0$, 
\begin{align}
    &\mathbb P \left(|\y^{\top} A \y-\Tr A\Phi| >t \right) \\
    \leq ~&2\exp \left(-\frac{1}{C} \min \left\{ \frac{t^2}{8 \lambda_{\sigma}^4 \|X\|^4\|A\|_F^2}, \frac{t}{\lambda_{\sigma}^2 \|X\|^2\|A\|}\right\}\right) +2\exp \left( -\frac{t^2}{32\lambda_{\sigma}^2\|X\|^2\|A\|^2 \|\E \y\|^2  }\right) \notag\\
\leq  ~&2\exp \left(-\frac{1}{C} \min \left\{ \frac{t^2}{8 \lambda_{\sigma}^4 \|X\|^4\|A\|_F^2}, \frac{t}{\lambda_{\sigma}^2 \|X\|^2\|A\|}\right\}\right) + 2\exp \left( -\frac{t^2}{32\lambda_{\sigma}^2\|X\|^2\|A\|^2 t_0} \right), \label{eq:inparticular}
\end{align}
where \begin{align}\label{eq:def_t_0}
    t_0:=2\lambda_{\sigma}^2 \sum_{i=1}^n (\|\x_i\|-1)^2+ 2n (\mathbb E \sigma(\xi))^2 , \quad \xi \sim \N(0,1).
\end{align}
\end{enumerate}
\end{cor}

\begin{remark}
Compared to \cite[Lemma 1]{louart2018random}, we identify the dependence on $\|A\|_F$ and $\E \y$ in the probability estimate. By using the inequality $\|A\|_F\leq \sqrt {n} \|A\|$, we obtain a similar inequality to the one in \cite{louart2018random} with a better dependence on $n$. Moreover, our bound in $t_0$  is independent of $d_0$, while the corresponding term $t_0$ in \cite[Lemma 1]{louart2018random} depends on $\|X\|$ and $d_0$.
In particular, when $\E\sigma(\xi)=0$ and $X$ is $(\varepsilon_n, B)$-orthonormal, $t_0$ is of order 1.  Hence, \eqref{eq:inparticular} with the special choice of $t_0$ is the key ingredient in the proof of Theorem \ref{thm:spectralnorm} to get a concentration of the spectral norm for CK. 
\end{remark}

\begin{proof}[Proof of Corollary \ref{thm:quadratic_concentration}]  We only need to prove \eqref{eq:inparticular}, since other statements follow immediately by taking $K=\sqrt{2}$.
Let $\x_i$ be the $i$-th column of $X$. Then
\begin{align*} 
      \|\E \y\|^2&= \| \E \sigma(\w^\top X)\|^2=\sum_{i=1}^n [\E \sigma(\w^{\top}\x_i)]^2. 
\end{align*}
Let $\xi\sim \N(0,1)$. We have
\begin{align}
    |\E \sigma(\w^{\top}\x_i)|&=|\E \sigma(\xi \|\x_i\|)|\leq \E| (\sigma(\xi \|\x_i\|)-\sigma(\xi))| +|\E \sigma(\xi)| \notag\\
   & \leq \lambda_{\sigma}\E|\xi (\|\x_i\|-1)| + |\E \sigma(\xi)|
   \leq \lambda_{\sigma} |\|\x_i\|-1|+|\E \sigma(\xi)|.\label{eq:Esigma}
\end{align}
Therefore
\begin{align}\label{eq:Eybound}
   \|\E \y\|^2&\leq \sum_{i=1}^n (\lambda_{\sigma} (\|\x_i\|-1)+|\E \sigma(\xi)|)^2\leq \sum_{i=1}^n 2\lambda_{\sigma}^2(\|\x_i\|-1 )^2+2(\E\sigma(\xi))^2  \\
   &=2\lambda_{\sigma}^2\sum_{i=1}^n (\|\x_i\|-1 )^2+2n(\E\sigma(\xi))^2=t_0,  \notag
\end{align}
and \eqref{eq:inparticular} holds.
 
\end{proof}

We include the following corollary about the variance of $\y^\top A\y$, which will be used in Section \ref{sec:1hlNN} to study the spectrum of the CK and NTK.

\begin{cor}\label{cor:L_2_converge}
Under the same assumptions of Corollary \ref{thm:quadratic_concentration}, we further assume that $t_0\leq C_1 n$,  and $\|A\|,\|X\|\leq C_2$. Then  as $n\to\infty$,
\[\frac{1}{n^2}\mathbb E\left[
\left|\y^{\top} A \y-\Tr A\Phi\right|^2\right]\to  0.\]
\end{cor}
\begin{proof}
Notice that $\|A\|_F\le \sqrt{n}\|A\|$. Thanks to Theorem \ref{thm:quadratic_concentration} (2), we have that for any $t>0$,
 \begin{equation}
    \P\left(\frac{1}{n}\left|\y^{\top} A \y-\Tr A\Phi\right|>t\right)\le 4\exp\left(-Cn\min\{t^2,t\}\right),
 \end{equation}where constant $C>0$ only relies  on  $C_1, C_2$, $\lambda_\sigma$, and $K$. Therefore, we can compute the variance in the following way:
 \begin{align*}
     \E\left[\frac{1}{n^2}\left|\y^{\top} A \y-\Tr A\Phi\right|^2\right]=~&\int_{0}^\infty \P\left(\frac{1}{n^2}\left|\y^{\top} A \y-\Tr A\Phi\right|^2>s\right)ds\\
     \le~ & 4\int_{0}^\infty \exp\left(-Cn\min\{s,\sqrt{s}\}\right)ds\\
     =~ & 4\int_{0}^1 \exp\left(-Cn\sqrt{s}\right)ds+4\int_{1}^{+\infty} \exp\left(-Cns\right)ds\rightarrow 0,
 \end{align*}as $n\rightarrow \infty$. Here, we use the dominant convergence theorem for the first integral in the last line.
\end{proof}

\section{Limiting law for general centered sample covariance matrices}\label{sec:general_sample_covariance}
Independent of the subsequent sections, this section focuses on the generalized sample covariance matrix where the dimension of the feature is much smaller than the sample size. We will later interpret such sample covariance matrix specifically for our neural network applications. Under certain weak assumptions, we prove the limiting eigenvalue distribution of the normalized sample covariance matrix satisfies two self-consistent equations, which are subsumed into a deformed semicircle law. Our findings in this section demonstrate some degree of universality, indicating that they hold across various random matrix models and may have implications for other related fields.

\begin{theorem}\label{thm:iff_conditions}
Suppose $\y_1,\ldots,\y_d\in\R^n$ are independent random vectors with the same distribution of a random vector  $\y\in\R^n$. Assume that $\E[\y]=\textbf{0}$, $\E[\y\y^\top]=\Phi_n\in\R^{n\times n}$, where $\Phi_n$ is a  deterministic matrix whose limiting eigenvalue distribution is $\mu_{\Phi}\not=\delta_0$. Assume $\|\Phi_n\|\le C$ for some constant $C$. Define $A_n:=\sqrt{\frac{d}{n}}\left(\frac{1}{d}\sum_{i=1}^d \y_i\y_i^\top-\Phi_n\right)$ and $R(z):=(A_n-z\Id)^{-1}$. For any $z\in\C^+$ and any deterministic matrices $D_n$ with $\|D_n\|\leq C$, suppose that as $n,d\rightarrow\infty$ and $n/d\rightarrow 0$,
\begin{equation}\label{condition_1}
    \tr R(z)D_n-\E \left[ \tr R(z)D_n\right]\overset{\textnormal{a.s.}}{\longrightarrow}0,
\end{equation}
and 
\begin{equation}\label{condition_2}
    \frac{1}{n^2} \E\left[\left|\y^\top D_n\y-\Tr D_n \Phi_n\right|^2\right]\to 0.
\end{equation}
 Then the empirical eigenvalue  distribution of matrix $A_n$ weakly converges to $\mu$ almost surely, whose Stieltjes transform $m(z)$ is defined by 
\begin{equation}\label{eq:fixed_point1}
    m(z)+\int \frac{d \mu_\Phi(x)}{z+\beta(z) x}=0
\end{equation}
 for each $z\in\C^+$, where $\beta(z)\in\C^+$ is the unique solution to  
\begin{equation}\label{eq:fixed_point2}
    \beta(z)+\int \frac{xd \mu_\Phi(x)}{z+\beta(z) x}=0. 
\end{equation}
In particular, $\mu=\mu_s\boxtimes \mu_\Phi$.
\end{theorem}

\begin{remark}
In \cite{Xie2013LimitingSD},  it was assumed that
$\frac{d}{n^3}\mathbb E\left|\y^\top D_n\y-\Tr D_n \Phi_n\right|^2\to 0,$
where $n^3/d\rightarrow \infty$ and $n/d\to 0$ as $n\to \infty$. By martingale difference, this condition implies \eqref{condition_1}. However, we are not able to verify a certain step in the proof of \cite{Xie2013LimitingSD}. Hence, we will not directly adopt the result of \cite{Xie2013LimitingSD} but consider a more general situation without assuming $n^3/d\rightarrow \infty$. The weakest conditions we found are conditions \eqref{condition_1} and \eqref{condition_2}, which can be verified in our nonlinear random model.  
\end{remark}

The self-consistent equations we derived are consistent with the results in \cite{bao2012strong,Xie2013LimitingSD}, where they studied the empirical spectral distribution of separable sample covariance matrices in the regime $n/d\to 0$ under different assumptions. When $n\rightarrow\infty$ and $n/d\rightarrow 0$, our goal is to prove that the Stieltjes transform $m_n(z)$ of the empirical eigenvalue distribution of $A_n$ and $\beta_n(z):=\tr[R(z)\Phi_n]$ point-wisely converges to $m(z)$ and $\beta(z)$, respectively.

For the rest of this section, we first prove a series of lemmas to get $n$-dependent deterministic equivalents related to \eqref{eq:fixed_point1} and \eqref{eq:fixed_point2} and then deduce the proof of Theorem \ref{thm:iff_conditions} at the end of this section. Recall $A_n=\sqrt{\frac{d}{n}}\left(\frac{1}{d}\sum_{i=1}^d \y_i\y_i^\top-\Phi_n\right)$, $R(z)=(A_n-z\Id)^{-1}$, and $\y$ is a random vector independent of $A_n$ with the same distribution of $\y_i$.
\begin{lemma}\label{lemma:fixed_point_error}
Under the assumptions of Theorem \ref{thm:iff_conditions}, for any $z\in\C^+$, as $d,n\rightarrow\infty$,
\begin{equation}\label{eq:fixed_point_error}
    \tr D+z\E[\tr R(z)D]+\E\left[\frac{\frac{1}{n}\y^\top DR(z)\y\cdot\frac{1}{n}\y^\top R(z)\y}{1+\sqrt{\frac{n}{d}}\frac{1}{n}\y^\top R(z)\y}\right]=o(1),
\end{equation} where $D\in\R^{n\times n}$ is any deterministic matrix such that $\|D\|\le C$, for some constant $C$.
\end{lemma}
\begin{proof}
Let $z=u+iv\in\C^+$ where $u\in\R$ and $v>0$. Let
\[\Hat{R}:=\left(\frac{1}{\sqrt{dn}}\sum_{j=1}^{d+1}\y_j\y_j^\top-\sqrt{\frac{d}{n}}\Phi_n -z\Id\right)^{-1},\] where $\y_j$'s are independent copies of $\y$ defined in Theorem \ref{thm:iff_conditions}. Notice that, for any deterministic matrix $D\in \R^{n\times n}$,
\begin{align}D&=\Hat{R}\left(\frac{1}{\sqrt{dn}}\sum_{j=1}^{d+1}\y_j\y_j^\top-\sqrt{\frac{d}{n}}\Phi_n -z\Id\right)D\\
&= \frac{1}{\sqrt{dn}}\Hat{R}\left(\sum_{i=1}^{d+1}\y_i\y_i^\top\right)D-\sqrt{\frac{d}{n}}\Hat{R}\Phi_n D-z\Hat{R}D.
\end{align}

Without loss of generality, we assume $\|D\|\le 1$. Taking normalized trace, we have
\begin{equation}\label{eq:tr_A}
     \tr D+z\tr [\Hat{R}D]= \frac{1}{\sqrt{dn}}\frac{1}{n}\sum_{i=1}^{d+1}\y_i^\top D\Hat{R}\y_i-\sqrt{\frac{d}{n}}\tr [\Hat{R}\Phi_n D].
\end{equation} 
For each $1\le i\le d+1$, Sherman–Morrison formula (Lemma \ref{lem:SMformula}) implies
\begin{equation}\label{eq:S-M formula}
    \Hat{R} = R^{(i)}-\frac{R^{(i)}\y_i\y_i^\top R^{(i)}}{\sqrt{dn}+\y_i^\top R^{(i)}\y_i},
\end{equation}
where the leave-one-out resolvent $R^{(i)}$ is defined as  \[R^{(i)}:=\left(\frac{1}{\sqrt{dn}}\sum_{ 1\leq j\leq d+1, j\ne i}\y_j\y_j^\top-\sqrt{\frac{d}{n}}\Phi_n -z\Id\right)^{-1}.\] 
Hence, by \eqref{eq:S-M formula}, we obtain
\begin{equation}\label{eq:quadratic_form}
    \frac{1}{\sqrt{dn}}\frac{1}{n}\sum_{i=1}^{d+1}\y_i^\top D\Hat{R}\y_i = \frac{1}{n}\sum_{i=1}^{d+1} \frac{\y_i^\top DR^{(i)}\y_i}{\sqrt{dn}+\y_i^\top R^{(i)}\y_i}.
\end{equation}
Combining equations \eqref{eq:tr_A} and \eqref{eq:quadratic_form}, and applying expectation at both sides implies
\begin{align}
\tr D+z\E[\tr\Hat{R}D]=~&\frac{1}{n}\sum_{i=1}^{d+1}\E\left[\frac{\y_i^\top DR^{(i)}\y_i}{\sqrt{dn}+\y_i^\top R^{(i)}\y_i}\right]-\sqrt{\frac{d}{n}}\E\tr \Hat{R}\Phi_n D \notag\\
=~&\frac{d+1}{n}\E\left[\frac{\y^\top DR(z)\y}{\sqrt{dn}+\y^\top R(z)\y}\right]-\sqrt{\frac{d}{n}}\E\tr \Hat{R}\Phi_n D,\label{eq:trE}
\end{align}
because of the assumption that all $\y_i$'s have the same distribution as vector $\y$ for all $i\in [d+1]$.  With \eqref{eq:trE}, to prove \eqref{eq:fixed_point_error}, we will first show that when $n,d\rightarrow\infty$,
\begin{align}
     \sqrt{\frac{d}{n}}\left(\E[\tr\Hat{R}\Phi_n D]-\E[\tr R(z)\Phi_n D]\right)=o(1), \label{eq:Condition1}\\
    \E[\tr\Hat{R}D]-\E[\tr R(z)D]=o(1),\label{eq:Condition2}\\
     \frac{1}{n}\E\left[\frac{\y^\top DR(z)\y}{\sqrt{dn}+\y^\top R(z)\y}\right]=o(1).\label{eq:Condition3}
\end{align}
 Recall that 
\[\Hat{R}-R(z)=\frac{1}{\sqrt{dn}}R(z)\left(\y_{d+1}\y_{d+1}^\top\right)\Hat{R},\]
and spectral norms $\|\Hat{R}\|,\|R(z)\|\le 1/v$ due to Proposition C.2 in \cite{fan2020spectra}. Notice that $\|\Phi_n\|\leq C$. Hence, we can deduce that
\begin{align*}
\sqrt{\frac{d}{n}}\left|\E[\tr\Hat{R}\Phi_n D]-\E[\tr R(z)\Phi_n D]\right|\le ~& \frac{1}{n}\E[|\tr  R(z)\y_{d+1}\y_{d+1}^\top\Hat{R}\Phi_n D|]\\
\le~ & \frac{1}{n^2}\E[\|\Hat{R}\Phi_nDR(z)\|\cdot \|\y_{d+1} \|^2]\\
= ~& \frac{C}{v^2n^2}\E[\Tr \y_{d+1}\y_{d+1}^\top ]=\frac{C\Tr\Phi_n}{v^2n^2}\le \frac{C^2}{v^2n}\rightarrow 0,
\end{align*}as $n\rightarrow\infty.$ The same argument can be applied to the error of $\E[\tr\Hat{R}D]-\E[\tr R(z)D]$. Therefore \eqref{eq:Condition1} and \eqref{eq:Condition2} hold.
For \eqref{eq:Condition3}, we denote $\tilde{\y}:=\y/(nd)^{1/4}$ and observe that
\begin{equation}\label{eq:second_term}
    \frac{1}{n}\E\left[\frac{\y^\top DR(z)\y}{\sqrt{dn}+\y^\top R(z)\y}\right]= \frac{1}{n}\E\left[\frac{\tilde{\y}^\top DR(z)\tilde{\y}}{1+\tilde{\y}^\top R(z)\tilde{\y}}\right].
\end{equation}
Let $R(z)=\sum_{i=1}^n \frac{1}{\lambda_i-z} \mathbf u_i \mathbf u_i^\top$ be the eigen-decomposition of $R(z)$. Then
\begin{align}
  \tilde{\y}^\top R(z)\tilde{\y}/\|\tilde{\y}\|^2 =\sum_{i=1}^n \frac{1}{\lambda_i-z} \frac{(\langle \mathbf u_i, \tilde{\y}\rangle)^2 }{\|\tilde {\y}\|^2} :=\int \frac{1}{x-z} d\mu_{\tilde \y}
\end{align}
is the Stieltjes transform of a discrete measure $\mu_{\tilde \y}=\sum_{i=1}^n \frac{(\langle \mathbf u_i, \tilde{\y}\rangle)^2 }{\|\tilde {\y}\|^2} \delta_{\lambda_i} $. Then, we can control the real part of $\tilde{\y}^\top R(z)\tilde{\y}$ by Lemma \ref{lem:BS10B11}:
\begin{equation}\label{eq:Real_part}
    \left|\Re (\tilde{\y}^\top R(z)\tilde{\y})\right|\le v^{-1/2}\|\tilde{\y}\|\left(\Im (\tilde{\y}^\top R(z)\tilde{\y})\right)^{1/2}.
\end{equation}
We now separately consider two cases in the following: 
\begin{itemize}
\item  If the right-hand side of the above inequality \eqref{eq:Real_part} is at most $1/2$, then 
\[\left|1+ \tilde{\y}^\top R(z)\tilde{\y}\right|\ge \left|1+ \Re (\tilde{\y}^\top R(z)\tilde{\y})\right|\ge \frac{1}{2},\]
which results in 
\begin{equation}\label{eq:first_case}
    \left|\frac{\tilde{\y}^\top DR(z)\tilde{\y}}{1+\tilde{\y}^\top R(z)\tilde{\y}}\right|\le \frac{C}{\sqrt{dn}}\|\y\|^2.
\end{equation}
\item When $v^{-1/2}\|\tilde{\y}\|\left(\Im (\tilde{\y}^\top R(z)\tilde{\y})\right)^{1/2}>1/2$, we know that
\begin{align}
    \left|\frac{\tilde{\y}^\top DR(z)\tilde{\y}}{1+\tilde{\y}^\top R(z)\tilde{\y}}\right|\le~&  \frac{\|\tilde{\y}^\top D\|\|R(z)\tilde{\y}\|}{|\Im (1+\tilde{\y}^\top R(z)\tilde{\y})|}
    =  \frac{\|\tilde{\y}^\top D\|\|R(z)\tilde{\y}\|}{\tilde{\y}^\top \Im (R(z))\tilde{\y}}\nonumber\\
    \le ~&  \frac{\|\tilde{\y}^\top D\|}{\left(v\tilde{\y}^\top \Im (R(z))\tilde{\y}\right)^{1/2}}\le \frac{2\|\tilde{\y}^\top D\|\|\tilde{\y}\|}{v}\le \frac{C\|\y\|^2}{v\sqrt{nd}},\label{eq:second_case}
\end{align}where we exploit the fact that (see also Equation (A.1.11) in \cite{bai2010spectral})
\[\|R(z)\tilde{\y}\|=(\tilde{\y}^\top R(\Bar{z})R(z)\tilde{\y})^{1/2}=\left(\frac{1}{v}\tilde{\y}^\top \Im (R(z))\tilde{\y}\right)^{1/2}.\]
\end{itemize}
 
Finally, combining \eqref{eq:first_case} and \eqref{eq:second_case} in the above two cases, we can conclude the asymptotic result \eqref{eq:Condition3} because $\E\|\y\|^2=\Tr \Phi_n\le Cn$ in terms of the assumptions of Theorem~\ref{thm:iff_conditions}.

Then with \eqref{eq:Condition1}, \eqref{eq:Condition2}, and \eqref{eq:Condition3}, we get 
\begin{equation}\label{eq:fixed_point_error_1}
     \tr D+z\E[\tr R(z)D]=\E\left[\frac{\sqrt{\frac{d}{n}}\frac{1}{n}\y^\top DR(z)\y}{1+\frac{1}{\sqrt{dn}}\y^\top R(z)\y}-\sqrt{\frac{d}{n}}\tr R(z)\Phi_n D\right]+o(1),
\end{equation}
as $n\rightarrow\infty$. We utilize the notion $\E_\y$ to clarify the expectation only with respect to random vector $\y$, conditioning on other independent random variables. So the conditional expectation is $\E_{\y}\left[\frac{1}{n}\y^\top DR(z)\y\right] = \tr D R(z)\Phi_n$ and
\[\E\left[\frac{1}{n}\y^\top DR(z)\y\right]=\E\left[\E_\y\left[\frac{1}{n}\y^\top DR(z)\y\right]\right]=\E[\tr R(z)\Phi_nD].\] Therefore, based on \eqref{eq:fixed_point_error_1}, the conclusion \eqref{eq:fixed_point_error} holds.
\end{proof}

In the next lemma, we apply the quadratic concentration condition \eqref{condition_2} to simplify \eqref{eq:fixed_point_error}. 
\begin{lemma}\label{lemma:replacing}
 Under the assumptions of Theorem \ref{thm:iff_conditions}, condition \eqref{condition_2} of Theorem \ref{thm:iff_conditions} implies that
\begin{equation}\label{eq:condition_for_fraction}
    \E\left[\frac{\frac{1}{n}\y^\top DR(z)\y\cdot\frac{1}{n}\y^\top R(z)\y}{1+\sqrt{\frac{n}{d}}\frac{1}{n}\y^\top R(z)\y}\right]=\E\left[\frac{\tr DR(z)\Phi_n\tr R(z)\Phi_n}{1+\sqrt{\frac{n}{d}}\tr R(z)\Phi_n}\right]+o(1),
\end{equation}for each $z\in\C^+$ and any deterministic matrix $D$ with $\|D\|\leq C$.
\end{lemma}
\begin{proof}
Let us denote
 \[\delta_n:=\frac{\frac{1}{n}\y^\top DR(z)\y\cdot\frac{1}{n}\y^\top R(z)\y}{1+\sqrt{\frac{n}{d}}\frac{1}{n}\y^\top R(z)\y}-\frac{\tr DR(z)\Phi_n\tr R(z)\Phi_n}{1+\sqrt{\frac{n}{d}}\tr R(z)\Phi_n},\]
 \[Q_1:=\frac{1}{n}\y^\top DR(z)\y, \quad  Q_2:=\frac{1}{n}\y^\top R(z)\y,\]
 $\Bar{Q}_1:=\E_{\y}[Q_1]=\tr  DR(z)\Phi_n$, and $\Bar{Q}_2:=\E_{\y}[Q_1]=\tr  R(z)\Phi_n$. In other words, $\delta_n$ can be expressed by
 \begin{align*}
     \delta_n =~ & \frac{Q_1Q_2}{1+\sqrt{\frac{n}{d}}Q_2}-\frac{\Bar{Q}_1\Bar{Q}_2}{1+\sqrt{\frac{n}{d}}\Bar{Q}_2}\\
     =~&\frac{Q_1\left(Q_2+\sqrt{\frac{d}{n}}\right)}{1+\sqrt{\frac{n}{d}}Q_2}-\frac{\sqrt{\frac{d}{n}}Q_1}{1+\sqrt{\frac{n}{d}}Q_2}-\frac{\Bar{Q}_1\left(\Bar{Q}_2+\sqrt{\frac{d}{n}}\right)}{1+\sqrt{\frac{n}{d}}\Bar{Q}_2}+\frac{\sqrt{\frac{d}{n}}\Bar{Q}_1}{1+\sqrt{\frac{n}{d}}\Bar{Q}_2}\\
     =~&  \sqrt{\frac{d}{n}}(Q_1-\Bar{Q}_1)+\frac{\sqrt{\frac{d}{n}}(\Bar{Q}_1-Q_1)}{1+\sqrt{\frac{n}{d}}\Bar{Q}_2}+\frac{\sqrt{\frac{n}{d}}Q_1\sqrt{\frac{d}{n}}(\Bar{Q}_2-Q_2)}{\left(1+\sqrt{\frac{n}{d}}\Bar{Q}_2\right)\left(1+\sqrt{\frac{n}{d}}Q_2\right)}.
 \end{align*} 
 Observe that $\E[\Bar{Q}_i]=\E[Q_i]$ for $i=1,2$. Thus, $\delta_n$ has the same expectation as the last term
 \[\Delta_n:=\frac{Q_1(\Bar{Q}_2-Q_2)}{\left(1+\sqrt{\frac{n}{d}}\Bar{Q}_2\right)\left(1+\sqrt{\frac{n}{d}}Q_2\right)},\]
 since we can first take the expectation for $\y$ conditioning on the resolvent $R(z)$ and then take the expectation for $R(z)$. Besides, notice that $|\Bar{Q}_1|,|\Bar{Q}_2|\le \frac{C}{v}$ uniformly. Hence, $\sqrt{\frac{n}{d}}\Bar{Q}_2$ converges to zero uniformly and there exists some constant $C>0$ such that 
 \begin{equation}\label{eq:bound_fraction1}
     \left|\frac{1}{1+\sqrt{\frac{n}{d}}\Bar{Q}_2}\right|\le C,
 \end{equation}for all large $d$ and $n$. In addition, observe that 
 \[\frac{\sqrt{\frac{n}{d}}Q_1}{1+\sqrt{\frac{n}{d}}Q_2}=\frac{\tilde{\y}^\top DR(z)\tilde{\y}}{1+\tilde{\y}^\top R(z)\tilde{\y}},\]
where $\tilde{\y}$ is defined in the proof of Lemma \ref{lemma:fixed_point_error}. In terms of \eqref{eq:first_case} and \eqref{eq:second_case}, we verify that
 \begin{equation}\label{eq:bound_fraction2}
     \left|\frac{Q_1}{1+\sqrt{\frac{n}{d}}Q_2}\right|\le \frac{C\|\y\|^2}{n},
 \end{equation}where $C>0$ is some constant depending on $v$. Next, recall that condition \eqref{condition_2} exposes that 
 \begin{equation}\label{eq:bound_fraction3}
     \E (Q_2-\Bar{Q}_2)^2\to 0\quad\text{and}\quad \E(\|\y\|^2/n-\tr\Phi_n)^2\to 0
 \end{equation}
 as $n\rightarrow\infty$. The first convergence is derived by viewing $D_n=R(z)$ and taking expectation conditional on $R(z)$. To sum up, we can bound $|\Delta_n|$ based on \eqref{eq:bound_fraction1} and \eqref{eq:bound_fraction2} in the subsequent way:
 \begin{align*}
     |\Delta_n|\le~ & \frac{C\|\y\|^2}{n}|\Bar{Q}_2-Q_2|
     \le   C\left|\|\y\|^2/n-\tr\Phi_n\right|\cdot|\Bar{Q}_2-Q_2|+C\left|\tr\Phi_n\right|\cdot|\Bar{Q}_2-Q_2|.
 \end{align*} 
 Here, $|\tr\Phi_n|\le \|\Phi_n\|$ and $\|\Phi_n\|$ is uniformly bounded by some constant. Then, by Hölder's inequality, \eqref{eq:bound_fraction3} implies that $\E[|\Delta_n|]\rightarrow 0$, as $n$ approaching to infinity. This concludes $\E[\delta_n]=\E[\Delta_n]$ converges to zero.
 
\end{proof}

\begin{lemma}\label{lemma:fixed_point_eq}
Under assumptions of Theorem \ref{thm:iff_conditions}, we can conclude that
 \[ \lim_{n,d\rightarrow\infty}\left(\tr D+z\E[\tr R(z)D]+\E\left[ \tr DR(z)\Phi_n\right]\E\left[\tr R(z)\Phi_n \right]\right)=0\]
holds for each $z\in\C^+$ and deterministic matrix $D$ with uniformly bounded spectral norm.
\end{lemma}
\begin{proof}
Based on Lemma \ref{lemma:fixed_point_error} and Lemma \ref{lemma:replacing}, \eqref{eq:condition_for_fraction} and \eqref{eq:fixed_point_error} yield
\[ \tr D+z\E[\tr R(z)D]+\E\left[\frac{\tr DR(z)\Phi_n\tr R(z)\Phi_n}{1+\sqrt{\frac{n}{d}}\tr R(z)\Phi_n}\right]=o(1).\]
As $|\tr R(z)D|$ and $| \tr R(z)D\Phi_n|$ are bounded by some constants uniformly and almost surely, for sufficiently large $d$ and $n$, $|\sqrt{\frac{n}{d}}\tr R(z)\Phi_n|<1/2$ and 
\begin{align*}
    &\left|\E\left[\frac{\tr DR(z)\Phi_n\tr R(z)\Phi_n}{1+\sqrt{\frac{n}{d}}\tr R(z)\Phi_n}\right]-\E\left[ \tr DR(z)\Phi_n\tr R(z)\Phi_n \right]\right|\\
    \le ~&\E\left[|\tr R(z)D|\cdot |\tr R(z)D\Phi_n|\cdot\left|\frac{\sqrt{\frac{n}{d}}\tr R(z)\Phi_n}{1+\sqrt{\frac{n}{d}}\tr R(z)\Phi_n}\right|\right]\le 2C\sqrt{\frac{n}{d}}\rightarrow 0,
\end{align*}as $n/d\rightarrow 0$. Hence, 
\begin{equation}\label{eq:crossing_term}
    \tr D+z\E[\tr R(z)D]+\E\left[ \tr DR(z)\Phi_n \tr R(z)\Phi_n \right]=o(1).
\end{equation}
Considering $D_n=\Phi_n$ in \eqref{condition_1}, we can get almost sure convergence for $\tr DR(z)\Phi_n\cdot (\tr R(z)\Phi_n -\E\left[\tr R(z)\Phi_n \right])$ to zero. Thus by dominated convergence theorem, $$\lim_{n\rightarrow\infty}\E\left[\tr DR(z)\Phi_n\cdot (\tr R(z)\Phi_n -\E\left[\tr R(z)\Phi_n \right])\right]\rightarrow 0.$$
So we can replace the third term at the right-hand side of \eqref{eq:crossing_term} with $$\E\left[ \tr DR(z)\Phi_n\right]\E\left[\tr R(z)\Phi_n \right]$$ to obtain the conclusion.
\end{proof}

\begin{proof}[Proof of Theorem \ref{thm:iff_conditions}]
Fix any $z\in \C^+$. Denote the  Stieltjes transform of empirical spectrum of $A_n$ and its expectation by $m_n(z):=\tr R(z)$ and $\Bar{m}_n(z):=\E[m_n(z)]$ respectively. Let $\beta_n(z):=\tr R(z)\Phi_n$ and $\Bar{\beta}_n(z):=\E[\beta_n(z)]$. Notice that $m_n(z), \Bar{m}_n(z), \beta_n$ and $\Bar{\beta}_n(z)$ are all in $\C^+$ and uniformly and almost surely bounded by some constant. By choosing $D=\Id$ in Lemma \ref{lemma:fixed_point_eq}, we conclude 
 \begin{equation}\label{eq:fixed_point_appro_10}
     \lim_{n,\,d\rightarrow\infty}\left(1+z\bar{m}_n(z)+\Bar{\beta}_n(z)^2\right)=0.
 \end{equation}
 
 Likewise, in Lemma \ref{lemma:fixed_point_eq}, we consider $D=\left(\Bar{\beta}_n(z)\Phi_n+z\Id\right)^{-1}\Phi_n$. Let $$U=\left(\Bar{\beta}_n(z)\Phi_n+z\Id\right)^{-1}.$$ Because $\|\Phi_n\|$ is uniformly bounded, $\|D\|\le C\|U\|$. In terms of Lemma \ref{lem:propC2}, we only need to provide a lower bound for the imaginary part of $U$. Observe that $\Im U=\Im\Bar{\beta}_n(z)\Phi_n+v\Id	\succeq v\Id $ since $\lambda_{\min}(\Phi_n)\ge 0$ and $\Im\Bar{\beta}_n(z)>0$. Thus, $\|D\|\le Cv^{-1}$ for all $n$. Meanwhile, we have the equation $\Bar{\beta}_n(z)\Phi_n D=\Phi_n-zD$ and hence, 
 \[\Bar{\beta}_n(z)\E[\tr R(z)\Phi_n D]=\E[\tr R(z)\Phi_n D]\E[\tr R(z)\Phi_n]=\Bar{\beta}_n(z)-z\E[\tr R(z)D].\]
So applying Lemma \ref{lemma:fixed_point_eq} again, we have another limiting equation $\tr D+\Bar{\beta}_n(z)\rightarrow 0$. In other words,
  \begin{equation}\label{eq:fixed_point_appro_20}
     \lim_{n,\,d\rightarrow\infty}\left(\tr\left(\Bar{\beta}_n(z)\Phi_n+z\Id\right)^{-1}\Phi_n+\bar{\beta}_n(z)\right)=0.
 \end{equation}Thanks to the identity
 \begin{align*}
     \bar{\beta}_n(z)\tr\left(\Bar{\beta}_n(z)\Phi_n+z\Id\right)^{-1}\Phi_n-1 =-z\tr\left(\Bar{\beta}_n(z)\Phi_n+z\Id\right)^{-1},
 \end{align*}
 we can modify \eqref{eq:fixed_point_appro_10} and \eqref{eq:fixed_point_appro_20} to get 
  \begin{equation}\label{eq:fixed_point_appro_1}
     \lim_{n,\,d\rightarrow\infty}\left(\bar{m}_n(z)+\tr\left(\Bar{\beta}_n(z)\Phi_n+z\Id\right)^{-1}\right)=0.
 \end{equation}

 Since $\Bar{\beta}_n(z)$ and $\Bar{m}_n(z)$ are uniformly bounded, for any subsequence in $n$, there is a further convergent sub-subsequence. We denote the limit of such sub-subsequence by $\beta(z)$ and $m(z)\in\C^+$ respectively. Hence, by \eqref{eq:fixed_point_appro_20} and \eqref{eq:fixed_point_appro_1}, one can conclude 
 \[ \lim_{n,\,d\rightarrow\infty}\left(\beta(z)+\tr\left(\beta(z)\Phi_n+z\Id\right)^{-1}\Phi_n\right)=0.\]
 Because of the convergence of the empirical eigenvalue distribution of $\Phi_n$, we obtain the fixed point equation \eqref{eq:fixed_point2} for $\beta(z)$. Analogously, we can also obtain \eqref{eq:fixed_point1} for $m(z)$ and $\beta(z)$. The existence and the uniqueness of the solutions to \eqref{eq:fixed_point1} and \eqref{eq:fixed_point2} are proved in \cite[Theorem 2.1]{bai2010limiting} and \cite[Section 3.4]{wang2014limiting}, which implies the  convergence of $\Bar{m}_n(z)$ and $\Bar{\beta}_n(z)$ to $m(z)$ and $\beta(z)$  governed by the self-consistent equations \eqref{eq:fixed_point1} and \eqref{eq:fixed_point2} as $n\rightarrow\infty$,  respectively. 
 
 Then, by virtue of condition \eqref{condition_1} in Theorem \ref{thm:iff_conditions}, we know $m_n(z)-\Bar{m}_n(z)\overset{\text{a.s.}}{\longrightarrow }0 $ and $\beta_n(z)-\Bar{\beta}_n(z)\overset{\text{a.s.}}{\longrightarrow }0 $. Therefore, the empirical  Stieltjes  transform $m_n(z)$ converges to $m(z)$ almost surely for each $z\in\C^+$. 
Recall that the Stieltjes transform of $\mu$ is $m(z)$. By the standard Stieltjes continuity theorem (see for example, \cite[Theorem B.9]{bai2010spectral}), this finally concludes the weak convergence of empirical eigenvalue distribution of $A_n$ to $\mu$.
 
 \bigskip
 
 Now we show $\mu=\mu_s\boxtimes \mu_\Phi$.
 The fixed point equations \eqref{eq:fixed_point1} and \eqref{eq:fixed_point2} induce
\begin{equation}\label{eq:fixed_point22}
    \beta^2(z)+1+zm(z)=0,
\end{equation}
since $\beta(z)\in \C^+$ for any $z\in\C^+$. Together with \eqref{eq:fixed_point1}, we attain the same self-consistent equations for the convergence of the empirical spectral distribution of the Wigner-type matrix studied in \cite[Theorem 1.1]{bai2010limiting}. 

Define $W_n$, the $n$-by-$n$ Wigner matrix, as a Hermitian matrix with independent entries \[\{W_n[i,j]: \E[W_n[i,j]]=0,\text{ }\E[W_n[i,j]^2]=1,\text{ } 1\le i\le j\le n\}.\]  The Wigner-type matrix studied in \cite[Definition 1.2]{bai2010limiting} is indeed $\frac{1}{\sqrt{n}}\Phi_n^{1/2}W_n\Phi_n^{1/2}$. Hence, such Wigner-type matrix $\frac{1}{\sqrt{n}}\Phi_n^{1/2}W_n\Phi_n^{1/2}$ has the same limiting spectral distribution as $A_n$ defined in Theorem \ref{thm:iff_conditions}. Both limits are determined by self-consistent equations \eqref{eq:fixed_point1} and \eqref{eq:fixed_point22}.

On the other hand, based on \cite[Theorem 5.4.5]{anderson2010introduction}, $\frac{1}{\sqrt{n}}W_n$ and $\Phi_n$ are almost surely asymptotically free, i.e. the empirical distribution of $\{\frac{1}{\sqrt{n}}W_n,\Phi_n\}$ converges almost surely to the law of $\{\bf s,d\}$, where $\textbf{s}$ and $\textbf{d}$ are two free non-commutative random variables ($\textbf{s}$ is a semicircle element and $\textbf{d}$ has the law $\mu_\Phi$). Thus, the limiting spectral distribution $\mu$ of $\frac{1}{\sqrt{n}}\Phi_n^{1/2}W_n\Phi_n^{1/2}$ is  the free multiplicative convolution between $\mu_s$ and $\mu_{\Phi}$. This implies  $\mu=\mu_s\boxtimes \mu_\Phi$ in our setting.
\end{proof}

\section{Proof of Theorem \ref{thm:law_DNN} and Theorem \ref{thm:law_NTK}}\label{sec:1hlNN}

To prove Theorem \ref{thm:law_DNN}, we first establish the following proposition to analyze the difference between Stieltjes transform of \eqref{eq:center_RF} and its expectation. This will assist us to verify condition \eqref{condition_1} in Theorem \ref{thm:iff_conditions}. The proof is based on  \cite[Lemma E.6]{fan2020spectra}.
\begin{prop}\label{prop:trRD}
Let $D\in\R^{n\times n}$ be any deterministic symmetric matrix with a uniformly bounded spectral norm. Following the notions in Theorem \ref{thm:law_DNN}, assume $\|X\|\le C$ for some constant $C$ and Assumption \ref{assump:sigma} holds. Let $R(z)$ be the resolvent 
\[\left(\frac{1}{\sqrt{d_1n}}\left(Y^\top Y-\E[Y^\top Y]\right)-z\Id\right)^{-1},\]
for any fixed $z\in\C^+$. Then, there exist some constants $s,n_0>0$ such that for all $n>n_0$ and any $t>0$,
\[\P\left(\left|\tr R(z)D-\E[\tr R(z) D]\right|>t\right)\le 2e^{-cnt^2}.\]
\end{prop}
\begin{proof}
 Define function $F:\R^{d_1\times d_0}\rightarrow \R$ by $F(W):=\tr R(z)D$. Fix any $W,\Delta\in \R^{d_1 \times d_0}$ where $\|\Delta\|_F=1$, and let $W_t=W+t \Delta$. We want to verify $F(W)$ is a Lipschitz function in $W$ with respect to the Frobenius norm. First, recall \[R(z)^{-1}=\frac{1}{\sqrt{d_1n}}\sigma(WX)^\top\sigma(WX)-\sqrt{\frac{d_1}{n}}\Phi-z\Id,\] where the last two terms are deterministic with respect to $W$. Hence,
\begin{align*}
\vec(\Delta)^\top (\nabla F(W))&=\frac{d}{dt}\Big|_{t=0} F(W_t)\\
&=- \tr R(z)\left(\frac{d}{dt}\Big|_{t=0}R(z)^{-1}\right)R(z)D\\
&=-\frac{1}{\sqrt{d_1n}}\tr R(z)\left(\frac{d}{dt}\Big|_{t=0}\sigma(W_tX)^\top\sigma(W_tX)
\right)R(z)D\\
&=-\frac{2}{\sqrt{d_1n}}\tr R(z)\left(\sigma(WX)^\top\cdot \frac{d}{dt}\Big|_{t=0}\sigma(W_tX)
\right)R(z)D\\
&=-\frac{2}{\sqrt{d_1n}}\tr R(z)\left(\sigma(WX)^\top\cdot \left(\sigma'(WX) \odot
(\Delta X)\right)
\right)R(z)D,
\end{align*}
where $\odot$ is the Hadamard product, and $\sigma'$ is applied entrywise. Here we utilize the formula \[\partial R(z)=-R(z)(\partial (R(z)^{-1}))R(z)\] and $R(z)=R(z)^\top$. Lemma~\ref{lem:propC2} in Appendix~\ref{appendix:lemmas} implies that $\|R(z)\|\leq \frac{1}{|\Im z|}.$ Therefore, based on the assumption of $D$, we have
\[\Big|\vec(\Delta)^\top (\nabla F(W))\Big|
\leq \frac{C}{\sqrt{d_1n}}  \|R(z)\sigma(WX)^\top\| \cdot
\|\sigma'(WX) \odot (\Delta X)\|,\] for some constant $C>0$.
For the first term in the product on the right-hand side,
\begin{align*}
&\left(\frac{1}{\sqrt{d_1n}}\|R(z)\sigma(WX)^\top\|\right)^2\\
=&\frac{1}{\sqrt{d_1n}}\left\|R(z)\left(\frac{1}{\sqrt{d_1n}} \sigma(WX)^\top \sigma(WX)\right) R(z)^*\right\|\\
\leq & \frac{1}{\sqrt{d_1n}}\left(\|R(z)R(z)^{-1}R(z)^*\|
+\left\|R(z)\left(\sqrt{\frac{d_1}{n}}\Phi+z\Id\right)R(z)^*\right\|\right)\\
 \leq &\frac{1}{\sqrt{d_1n}}\left(\|R(z)\|+\|R(z)\|^2\left(\sqrt{\frac{d_1}{n}}\|\Phi\|+|z|\right)\right)\leq \frac{C}{n}.
\end{align*}
For the second term,
\[\|\sigma'(WX) \odot (\Delta X)\| \leq \|\sigma'(WX) \odot (\Delta
X)\|_F \leq \lambda_\sigma \|\Delta X\|_F \leq \lambda_\sigma \|\Delta\|_F \cdot
\|X\| \leq C.\]
Thus, $|\vec(\Delta)^\top (\nabla F(W))| \leq C/\sqrt{n}$. This holds for
every $\Delta$ such that $\|\Delta\|_F=1$, so
$F(W)$ is $C/\sqrt{n}$-Lipschitz in $W$ with respect to the Frobenius norm.
Then the result follows from the Gaussian concentration inequality for Lipschitz functions.
\end{proof}

Next, we investigate the approximation of $\Phi=\E_{\w}[\sigma(\w^\top X)^\top\sigma(\w^\top X)]$ via the Hermite polynomials $\{h_k\}_{k\ge 0}$.  The orthogonality of Hermite polynomials allows us to write $\Phi$ as a series of kernel matrices. Then we only need to estimate each kernel matrix in this series. The proof is directly based on \cite[Lemma 2]{ghorbani2019limitations}. The only difference is that we consider the deterministic input data $X$ with the $(\varepsilon_n, B)$-orthonormal property, while in Lemma 2 of \cite{ghorbani2019limitations}, the matrix $X$ is formed by independent Gaussian vectors.  

\begin{lemma}\label{lemma:operator_Phi}
Recall the definition of $\Phi_0$ in \eqref{def:Phi0}. If  $X$ is $(\varepsilon_n,B)$-orthonormal and Assumption \ref{assump:sigma} holds, then we have the spectral norm bound
 \[\|\Phi-\Phi_0\|\le C_B\varepsilon_n^2\sqrt{n},\] 
 where $C_B$ is a constant depending on $B$. Suppose that $\varepsilon_n^2\sqrt{n}\to0$ as $n\to\infty$, then  $\|\Phi\|\leq C$ uniformly for some constant $C$ independent of $n$. 
\end{lemma}
 \begin{proof}
By Assumption \ref{assump:sigma}, we know that
\[ \xi_0(\sigma)=0, \quad \sum_{k=1}^{\infty}\zeta_k^2(\sigma)=\E[\sigma(\xi)^2]=1.\]
For any fixed $t$, $\sigma(tx)\in  L^2(\R,\Gamma)$. This is because $\sigma(x)\in L^2(\R,\Gamma)$ is a Lipschitz function and by triangle inequality $ |\sigma(tx)-\sigma(x)|\leq \lambda_{\sigma} |tx-x|$, we have, for $\xi\sim\N(0,1)$,
\begin{align}
    \E (\sigma(t\xi)^2) &\leq \E (|\sigma(\xi)|+\lambda_{\sigma}|t\xi-\xi|)^2<\infty.
\end{align}

 For $1\le\alpha\le n$, let $\sigma_\a(x):=\sigma(\|\x_\a\|x)$ and the Hermite expansion of $\sigma_a$ can be written as
 \[\sigma_\a(x)=\sum_{k=0}^{\infty} \zeta_k(\sigma_\a) h_k(x),\]
 where the coefficient $\zeta_k(\sigma_\a)=\E[\sigma_\a(\xi)h_k(\xi)]$. Let unit vectors be $\mathbf u_\a=\x_\a/\|\x_\a\|$, for $1\le \a\le n$. So for $1\le\a,\b\le n$, the $(\a,\b)$ entry of $\Phi$ is  
 \[\Phi_{\a\b}=\E[\sigma(\w^\top \x_{\alpha})\sigma(\w^\top \x_{\beta})]=\E[\sigma_\a(\xi_\a)\sigma_\b(\xi_\b)],\]
 where $(\xi_\a,\xi_\b)=(\w^\top \mathbf u_\a,\w^\top  \mathbf u_\b)$ is a  Gaussian random vector with mean zero and covariance 
 \begin{equation}\label{eq:covariance}
     \begin{pmatrix}
 1 & \mathbf u_\a^\top \mathbf u_\b\\
 \mathbf u_\a^\top \mathbf u_\b & 1
 \end{pmatrix}.
 \end{equation}
 By the orthogonality of Hermite polynomials with respect to $\Gamma$ and  Lemma \ref{lem:NM20D2}, we can obtain \[\E[h_j(\xi_\a)h_k(\xi_\b)]=\E [h_j(\w^\top \mathbf u_{\a})h_k(\w^\top \mathbf u_{\beta})]=\delta_{j,k}(\mathbf u_\a^\top \mathbf u_\b)^k,\]
 which leads to
 \begin{equation}\label{eq:Phiab}
     \Phi_{\a\b}= \sum_{k=0}^{\infty} \zeta_k(\sigma_\a)\zeta_k(\sigma_\b) (\mathbf u_\a^\top \mathbf u_\b)^k.
 \end{equation} 
 
For any $k\in\NN$, let $T_k$ be an $n$-by-$n$ matrix with $(\a,\b)$-th entry
 \begin{align}\label{eq:Talphabeta}
  (T_k)_{\a\b}:= \zeta_k(\sigma_\a)\zeta_k(\sigma_\b) (\mathbf u_\a^\top \mathbf u_\b)^k.   
 \end{align}
 Specifically, for any $k\in\NN$, we have \[T_k=D_k f_k(X^\top X)D_k,\] where $D_k$ is the diagonal matrix $\diag(\zeta_k(\sigma_\a)/\|\x_\a\|^k)_{\alpha\in [n]}$. 
 
At first, we consider twice differentiable $\sigma$ in Assumption \ref{assump:sigma}. Similar with \cite[Equation (26)]{ghorbani2019limitations}, for any $\varepsilon>0$ and $|t-1|\le \varepsilon$, we take the Taylor approximation of $\sigma(tx)$ at point $x$, then there exists $\eta$ between $tx$ and $x$ such that 
\[\sigma(tx)-\sigma(x)=\sigma'(x)x(t-1)+\frac{1}{2}\sigma''(\eta)x^2(t-1)^2.\]
Replacing $x$ by $\xi$ and taking expectation, since $\sigma''$ is uniformly bounded, we can get 
 \begin{equation}\label{eq:zeta_0_appr_0}
     \left|\E\left[\sigma(t\xi)-\sigma(\xi)\right]-\E[\sigma'(\xi)\xi](t-1)\right|\le C|t-1|^2\le C\varepsilon_n^2,
 \end{equation} 
 For $k\ge 1$, the Lipschitz condition for $\sigma$ yields
\begin{align}\label{eq:zeta_k_appr}
    |\zeta_k(\sigma_\a)-\zeta_k(\sigma) |\le C\left|\|\x_\a\|-1\right|\cdot \E[|\xi|\cdot|h_k(\xi)|]\le C\varepsilon_n,
\end{align}where constant $C$ does not depend on $k$. As for piece-wise linear $\sigma$, it is not hard to see
\begin{equation}\label{eq:zeta_0_appr_00}
     \E\left[\sigma(t\xi)-\sigma(\xi)\right]=\E[\sigma'(\xi)\xi](t-1).
\end{equation}

Now, we begin to approximate $T_k$ separately based on \eqref{eq:zeta_0_appr_0}, \eqref{eq:zeta_k_appr} and \eqref{eq:zeta_0_appr_00}. Denote $\diag(A)$ the diagonal submatrix of a matrix $A$.
 
\textbf{(1) Approximation for $\sum_{k\ge 4}(T_k-\diag(T_k))$}. At first, we estimate the $L^2$ norm with respect to $\Gamma$ of the function $\sigma_\a$. Recall that
$  \|\sigma_{\alpha} \|_{L^2}=\mathbb E [\sigma_{\alpha}(\xi)^2]^{1/2}.$
Because $\|\sigma\|_{L^2}=1$ and $\sigma$ is a Lipschitz function, we have  
\begin{align}
    \sup_{1\le \a\le n}\|\sigma-\sigma_\a\|_{L^2}=\E[(\sigma(\xi)-\sigma_\a(\xi))^2]^{1/2}\le~& C|\|\x_\a\|-1|,\label{eq:L_2bound}\\
    \sup_{1\le \a\le n}\|\sigma_\a\|_{L^2}\le~& 1+C\varepsilon_n.\label{eq:L_2bound_1}
\end{align}
Hence, $\|\sigma_\a\|_{L^2}$ is uniformly bounded with some constant for all large $n$. Next, we estimate the off-diagonal entries of $T_k$ when $k\ge 4$. From \eqref{eq:Talphabeta}, we obtain that
\begin{align}
    \left\|\sum_{k\ge 4}(T_k-\diag(T_k))\right\|\le ~& \left\|\sum_{k\ge 4}(T_k-\diag(T_k))\right\|_F\le\sum_{k\ge 4} \left\|T_k-\diag(T_k)\right\|_F\nonumber\\
    \le ~& \sum_{k\ge 4} \left(\sup_{\a\neq\b}|\mathbf u_\a^\top \mathbf u_\b|^{k}\right)\left[\sum_{\a,\b=1}^n\zeta_k(\sigma_\a)^2\zeta_k(\sigma_\b)^2 \right]^{\frac{1}{2}}\nonumber\\
    \le ~& \left(\sup_{\a\neq\b}|\mathbf u_\a^\top \mathbf u_\b|^{4}\right) \sum_{\a=1}^n\sum_{k=0}^\infty\zeta_k(\sigma_\a)^2\nonumber\\
    \le ~& n\cdot \left(\sup_{\a\neq\b}\frac{|\x_\a^\top \x_\b|^{4}}{\|\x_\a\|^4\|\x_\b\|^4}\right)\sup_{1\le \a\le n}\|\sigma_\a\|_{L^2}^2\le Cn\cdot\varepsilon_n^4, \label{eq:off_diagonal_bounds}
\end{align}when $n$ is sufficiently large.

\textbf{(2) Approximation for $T_0$}.  Recall $\E[\sigma(\xi)]=0$ and by Gaussian integration by part,
\[\E[\sigma'(\xi)\xi]=\E[\xi \int_{0}^{\xi} \sigma'(x)x dx]=\E [\xi^2\sigma(\xi)]-\E[\xi \int_{0}^{\xi} \sigma(x)dx]=\E [\xi^2\sigma(\xi)]-\E[\sigma(\xi)].\]
Then, we have
\begin{align*}
    \E[\sigma'(\xi)\xi]=\E[(\xi^2-1)\sigma(\xi)]=\E[\sqrt{2}h_2(\xi)\sigma(\xi)]=\sqrt{2}\zeta_2(\sigma). 
\end{align*}
If $\sigma$ is twice differentiable, then $\E[\sigma''(\xi)]=\sqrt{2}\zeta_2(\sigma) $ as well.

Thus,  taking $t=\|\x_{\alpha}\|$ in \eqref{eq:zeta_0_appr_0} and \eqref{eq:zeta_0_appr_00} implies that for any $1\le\a\le n,$
 \begin{equation}\label{eq:zeta_0_appr}
     \left|\zeta_0(\sigma_\a)-\sqrt{2}\zeta_2(\sigma)(\|\x_\a\|-1)\right|\le C\varepsilon_n^2.
 \end{equation} 
 Define $\vnu^\top:=(\zeta_0(\sigma_1),\ldots,\zeta_0(\sigma_n))$, then $T_0=\vnu\vnu^\top$. Recall the definition of $\vmu$ in \eqref{def:Phi0}. Then, \eqref{eq:zeta_0_appr} ensures that
 \[\|\vmu-\vnu\|\le C\sqrt{n}\varepsilon_n^2.\]
 Applying the $(\varepsilon_n,B)$-orthonormal property of $\x_\a$ yields
\begin{equation}\label{eq:bound_mu}
    \|\vmu\|^2=2\zeta_2(\sigma)^2\sum_{\a=1}^n (\|\x_\a\|-1)^2\le 2\zeta_2(\sigma)^2\sum_{\a=1}^n (\|\x_\a\|^2-1)^2\le 2B^2\zeta_2(\sigma)^2.
\end{equation}
Hence the difference between $T_0$ and $\vmu\vmu^\top$ is controlled by
\begin{equation}\label{eq:T0}
    \|T_0-\vmu\vmu^\top \|\le \|\vmu-\vnu\|\left(2\|\vmu\|+\|\vnu-\vmu\|\right)\le C\sqrt{n}\varepsilon_n^2.
\end{equation}

\textbf{(3) Approximation for $T_k$ for $k=1,2,3$}. For $0\le k\le 3$, Assumption~\ref{assump:asymptotics} and \eqref{eq:zeta_k_appr} show that
\begin{align}
    \left|\zeta_k(\sigma_\a)/\|\x_\a\|^k-\zeta_k(\sigma)\right|\le~ & \frac{1}{\|\x_\a\|^k}\left[\left|\zeta_k(\sigma_\a)-\zeta_k(\sigma)\right|+|\zeta_k(\sigma)|\cdot|\|\x_\a\|^k-1|\right]\nonumber\\
    \le~ & \frac{C\varepsilon_n+C_1|\|\x_\a\|-1|}{(1-\varepsilon_n)^k}\le C_2\varepsilon_n,\label{eq:zeta_k_appro}
\end{align} 
when $n$ is sufficiently large.  Notice that $T_k=D_kf_k(X^\top X)D_k$, where $D_k$ is the diagonal matrix. Hence, by \eqref{eq:zeta_k_appro}, \[\|D_k-\zeta_k(\sigma)\Id\|\le C_2\varepsilon_n.\] 
And for $k=1,2,3,$ by the triangle inequality,
\begin{align*}
    &\|T_k-\zeta_k(\sigma)^2f_k(X^\top X)\|=\|D_kf_k(X^\top X)D_k-\zeta_k(\sigma)^2f_k(X^\top X) \|\\
    \le~ & \|D_k-\zeta_k(\sigma)\Id\|\cdot \|f_k(X^\top X)\|(|\zeta_k(\sigma)|+\|D_k-\zeta_k(\sigma)\Id\|)
    \le  C\varepsilon_n\|f_k(X^\top X)\|.
\end{align*}  When $k=1$, $f_1(X^\top X)=X^\top X$ and $\|X^\top X\|\le \|X\|^2\le B^2$. When $k=2$, \[f_2(X^\top X)=(X^\top X)\odot (X^\top X).\]
From Lemma \ref{lem:Hadamardinequality} in Appendix~\ref{appendix:lemmas}, we have that
\begin{equation}\label{eq:bound_norm_f_k}
    \|f_2(X^\top X)\|\le \max_{1\le \a,\b\le n}|\x_\a^\top\x_{\b}|\cdot\|X\|^2\le B^2(1+\varepsilon_n).
\end{equation}
So the left-hand side of \eqref{eq:bound_norm_f_k} is bounded. Analogously, we can verify $\|f_3(X^\top X)\|$ is also bounded. Therefore, we have
\begin{equation}\label{eq:Tk}
    \|T_k-\zeta_k(\sigma)^2f_k(X^\top X)\|\le C\varepsilon_n,
\end{equation}for some constant $C$ and $k=1,2,3$ when $n$ is sufficiently large.

\textbf{(4) Approximation for $\sum_{k\ge 4 }\diag (T_k)$}.
Since $\mathbf u_\a^\top \mathbf u_\a=1$, we know
\[\sum_{k\ge 4 }\diag (T_k)=\diag \left(\sum_{k\ge 4}\zeta_k(\sigma_\a)^2\right)_{\a\in [n]}=\diag \left(\|\sigma_\a\|_{L^2}^2-\sum_{k=0}^4\zeta_k(\sigma_\a)^2\right)_{\a \in [n]}.\]
First, by \eqref{eq:L_2bound} and \eqref{eq:L_2bound_1}, we can claim that
\begin{align*}
    |\|\sigma_\a\|_{L^2}^2-1|= |\|\sigma_\a\|_{L^2}^2-\|\sigma\|_{L^2}^2|\le C\|\sigma_\a-\sigma\|_{L^2}\le C\varepsilon_n.
\end{align*}
Second, in terms of \eqref{eq:zeta_k_appro}, we obtain
\[|\zeta_k(\sigma_\a)^2-\zeta_k(\sigma)^2|\le C|\zeta_k(\sigma_\a)-\zeta_k(\sigma)|\le C\varepsilon_n,\] for $k=1,2$ and $3$. Combining these together, we conclude that
\begin{align}
    &\left\|\sum_{k\ge 4 }\diag (T_k)-(1-\zeta_1(\sigma)^2-\zeta_2(\sigma)^2-\zeta_3(\sigma)^2)\Id\right\|\nonumber\\
    \le \quad &\max_{1\le \a\le n}\left|(\|\sigma_\a\|_{L^2}^2-1)-\sum_{k=0}^4(\zeta_k(\sigma_\a)^2-\zeta_k(\sigma)^2)\right|\le C\varepsilon_n.\label{eq:diag_control}
\end{align}
Recall 
\[ \Phi_0= \vmu\vmu^\top+\sum_{k=1}^3 \zeta_k(\sigma)^2f_k(X^\top X)+(1-\zeta_1(\sigma)^2-\zeta_2(\sigma)^2-\zeta_3(\sigma)^2)\Id.\]
In terms of approximations \eqref{eq:off_diagonal_bounds}, \eqref{eq:T0}, \eqref{eq:Tk} and \eqref{eq:diag_control}, we can finally manifest
\begin{align}\label{eq:differencePhi}
    \|\Phi-\Phi_0\|\le C\left(\varepsilon_n+\sqrt{n}\varepsilon_n^2+n\varepsilon_n^4\right)\le C \sqrt{n}\varepsilon_n^2,
\end{align} for some constant $C>0$ as $\sqrt{n}\varepsilon_n^2\to 0$.  The spectral norm bound of $\Phi$ is directly deduced by the spectral norm bound of $\Phi_0$ based on \eqref{eq:bound_mu} and \eqref{eq:bound_norm_f_k}, together with \eqref{eq:differencePhi}.
 \end{proof}

 \begin{remark}[Optimality of $\eps_n$]\label{remark:optimal}
  For general deterministic data $X$, our pairwise orthogonality assumption with rate $n\eps_n^4=o(1)$ is optimal for the approximation of $\Phi$ by $\Phi_0$ in the spectral norm.  If we relax the decay rate of $\eps_n$ in Assumption~\ref{assump:asymptotics}, the above approximation may require including terms of higher-degree $f_k(X^\top X)$ for $k\ge4$ in $\Phi_0$, which will lead to the invalidation of some of our following results and simplifications. Subsequent to the initial completion of our paper, this  weaker regime has  been considered in our follow-up work \cite{wang2022overparameterized}. 
 \end{remark}

Next, we continue to provide an additional estimate for $\Phi$, but in the Frobenius norm to further simplify the limiting spectral distribution of $\Phi$.
 \begin{lemma}\label{lemma:spectrum_Phi}
 If Assumptions \ref{assump:sigma} and \ref{assump:asymptotics} hold, then $\Phi$ has the same limiting spectrum as $b_\sigma^2 X^\top X+(1-b_\sigma^2)\Id$ when $n\to \infty$, i.e.
 \[\limspec \Phi=\limspec\left(b_\sigma^2 X^\top X+(1-b_\sigma^2)\Id\right)=b_\sigma^2 \mu_0+(1-b_\sigma^2).\]
 \end{lemma}
 \begin{proof}
 By the definition of $b_\sigma$, we know that $b_\sigma=\zeta_1(\sigma)$. As a direct deduction of Lemma \ref{lemma:operator_Phi}, the limiting spectrum of $\Phi$ is identical to the limiting spectrum of $\Phi_0$. To prove this lemma, it suffices to check the Frobenius norm of the difference between $\Phi_0$ and $\zeta_1(\sigma)^2 X^\top X+(1-\zeta_1(\sigma)^2)\Id$. Notice that
 \begin{align*}
     &\Phi_0-\zeta_1(\sigma)^2 X^\top X-(1-\zeta_1(\sigma)^2)\Id\\
     =~&\vmu\vmu^\top +\zeta_2(\sigma)^2f_2(X^\top X)+\zeta_3(\sigma)^2f_3(X^\top X)-(\zeta_2(\sigma)^2+\zeta_3(\sigma)^2)\Id.
 \end{align*}
 By the definition of vector $\vmu$ and the assumption of $X$, we have
 \begin{equation}
     \|\vmu\vmu^\top\|_F=\|\vmu\|^2= 2\zeta_2^2(\sigma)\sum_{\a=1}^n(\|\x_\a\|-1)^2\le 2\zeta_2^2(\sigma) B^2.
 \end{equation}
 For $k=2,3$, the Frobenius norm can be controlled by
 \begin{align*}
     \|f_k(X^\top X)-\Id\|_F^2=~&\sum_{\a,\b=1}^n\left((\x_\a^\top\x_\b)^k-\delta_{\a\b}\right)^2\\
     \le~& n(n-1)\varepsilon_n^{2k}+\sum_{\a=1}^n (\|\x_\a\|^{2k}-1)^2\le n^2\varepsilon_n^{2k}+C n\varepsilon_n^2.
 \end{align*}
 Hence, as $n\to \infty$, we have
 \[\frac{1}{n}\|\vmu\vmu^\top\|_F^2,\quad\frac{1}{n}\|f_k(X^\top X)-\Id\|_F^2\to 0,~\text{ for } ~k=2,3,\]
as $n\varepsilon_n^4\to 0$. Then we conclude that
 \begin{equation*}
     \frac{1}{n}\|\Phi_0-\zeta_1(\sigma)^2 X^\top X-(1-\zeta_1(\sigma)^2)\Id\|_F^2\le C(n\varepsilon_n^4+\varepsilon_n^2)\to 0.
 \end{equation*}
Hence, $\limspec \Phi$ is the same as $\limspec (\zeta_1(\sigma)^2 X^\top X+(1-\zeta_1(\sigma)^2)\Id)$ when $n\to \infty,$ due to Lemma \ref{lem:ESDforbenius} in Appendix~\ref{appendix:lemmas}. 
 \end{proof} 
 
Moreover, the proof of Lemma~\ref{lemma:spectrum_Phi} can be modified to prove \eqref{eq:K_equi_psi}, so we omit its proof. Now, based on Corollary \ref{cor:L_2_converge}, Proposition \ref{prop:trRD}, Lemma \ref{lemma:operator_Phi}, and Lemma \ref{lemma:spectrum_Phi}, applying Theorem \ref{thm:iff_conditions} for general sample covariance matrices, we can finish the proof of Theorem \ref{thm:law_DNN}.
 \begin{proof}[Proof of Theorem \ref{thm:law_DNN}]
Based on Corollary \ref{cor:L_2_converge} and Proposition \ref{prop:trRD}, we can verify the conditions \eqref{condition_1} and \eqref{condition_2} in Theorem \ref{thm:iff_conditions}. By Lemma \ref{lemma:operator_Phi} and Lemma \ref{lemma:spectrum_Phi}, we know that the limiting eigenvalue distributions of $\Phi$ and $(1-b_\sigma^2)\Id+b_\sigma^2X^\top X$ are identical and $\|\Phi\|$ is uniformly bounded. So the limiting eigenvalue distribution of $\Phi$ denoted by $\mu_\Phi$ is just $(1-b_\sigma^2)+b_\sigma^2\mu_0$. Hence, the first conclusion of Theorem \ref{thm:law_DNN} follows from Theorem \ref{thm:iff_conditions}. 
 
 For the second part of this theorem, we consider the difference
\begin{align*}
     &\frac{1}{n}\left\|\frac{1}{\sqrt{d_1n}}\left(Y^\top Y-\E[Y^\top Y]\right)-\frac{1}{\sqrt{d_1n}}\left(Y^\top Y-d_1\Phi_0\right)\right\|^2_F\\
     \le~ & \frac{d_1}{n^2}\|\Phi-\Phi_0\|_F^2\le \frac{d_1}{n}\|\Phi-\Phi_0\|^2\le d_1\varepsilon_n^4\to 0,
\end{align*}where we employ Lemma \ref{lemma:operator_Phi} and the assumption $d_1\varepsilon_n^4=o(1)$. Thus, because of Lemma \ref{lem:ESDforbenius}, $\frac{1}{\sqrt{d_1n}}\left(Y^\top Y-d_1\Phi_0\right)$ has the same limiting eigenvalue distribution as \eqref{eq:center_RF}, $\mu_s\boxtimes ((1-b_\sigma^2)+b_\sigma^2\mu_0)$. This finishes the proof of Theorem \ref{thm:law_DNN}.
\end{proof}

 \bigskip
Next, we move to study the empirical NTK and its corresponding limiting eigenvalue distribution. Similarly, we first verify that such NTK concentrates around its expectation and then simplify this expectation by some deterministic matrix only depending on the input data matrix $X$ and nonlinear activation $\sigma$.  The following lemma  can be obtained from \eqref{eq:Lupperboundd} in Theorem \ref{thm:NTK_concentration}.
 \begin{lemma}\label{lem:NTKexpectation}
Suppose that Assumption \ref{assump:W} holds, $\sup_{x\in\R}|\sigma'(x)|\le \lambda_\sigma$ and $\|X\|\le B$. Then if $d_1=\omega(\log n)$, we have
 \begin{align}\label{eq:NTK_first_concentration_inequality}
     \frac{1}{d_1}\left\|(S^\top S)\odot(X^\top X)-\E[(S^\top S)\odot(X^\top X)]\right\|\to 0,
     \end{align}
 almost surely as $n,d_0,d_1\to\infty$. Moreover, if $d_1/n\to\infty$ as $n\to\infty$, then almost surely
 \begin{equation}\label{eq:NTK_second_concentration}
     \frac{1}{\sqrt{nd_1}}\left\|(S^\top S)\odot(X^\top X)-\E[(S^\top S)\odot(X^\top X)]\right\|\to 0.
 \end{equation}
 \end{lemma}

 \begin{lemma}\label{lemma:Psi_0} 
 Suppose $X$ is $(\varepsilon_n, B)$-orthonormal. 
Under Assumption~\ref{assump:sigma}, we have
\begin{equation}
    \|\Psi-\Psi_0\|\le C_B\varepsilon_n^4n,
\end{equation}
where $\Psi$ and $\Psi_0$ are defined in \eqref{eq:defPsi} and \eqref{eq:defPsi_0}, respectively, and $C_B$ is a constant depending on $B$.
 \end{lemma}
 \begin{proof}
 We can directly apply methods in the proof of Lemma \ref{lemma:operator_Phi}. 
 Notice that \eqref{eq:Hentrywise} and \eqref{eq:defcolumnS} imply
 $$\E[S^\top S]=d_1\E[\sigma'(\w^\top X)^\top \sigma'(\w^\top X)],$$
 for any standard Gaussian random vector $\w\sim \N(0,\Id)$. Recall that \eqref{eq:def_eta_k} defines the $k$-th coefficient of Hermite expansion of $\sigma'(x)$ by $\eta_k(\sigma)$ for any $k\in\NN.$
 Then, Assumption~\ref{assump:sigma} indicates $b_\sigma=\eta_0(\sigma)$ and $a_\sigma=\sum_{k=0}^\infty \eta_k^2(\sigma)$.  For $1\le\alpha\le n$, we introduce $\phi_\a(x):=\sigma'(\|\x_\a\|x)$ and the Hermite expansion of this function as
 \[\phi_\a(x)=\sum_{k=0}^{\infty} \zeta_k(\phi_\a) h_k(x),\]
 where the coefficient $\zeta_k(\sigma_\a)=\E[\phi_\a(\xi)h_k(\xi)]$. Let $\mathbf u_\a=\x_\a/\|\x_\a\|$, for $1\le \a\le n$. So for $1\le\a,\b\le n$, the $(\a,\b)$-entry of $\Psi$ is  \[\Psi_{\a\b}=\E[\phi_\a(\xi_\a)\phi_\b(\xi_\b)]\cdot (\x_a^\top\x_\b),\]
 where $(\xi_\a,\xi_{\b})=(\w^\top \mathbf u_\a,\w^\top \mathbf u_\b)$ is a Gaussian random vector with mean zero and covariance \eqref{eq:covariance}.
 Following the derivation of formula \eqref{eq:Phiab}, we obtain
 \begin{equation}
     \Psi_{\a\b}= \sum_{k=0}^{\infty} \frac{\zeta_k(\phi_\a)\zeta_k(\phi_\b)}{\|\x_\a\|^k\|\x_\b\|^k} (\x_\a^\top \x_\b)^{k+1}.
 \end{equation}
For any $k\in\NN$, let $T_k\in\R^{n\times n}$ be an $n$-by-$n$ matrix with $(\a,\b)$ entry
 \[(T_k)_{\a\b}:= \frac{\zeta_k(\phi_\a)\zeta_k(\phi_\b)}{\|\x_\a\|^k\|\x_\b\|^k} (\x_\a^\top \x_\b)^{k+1}.\]
We can write $T_k=D_kf_{k+1}(X^\top X)D_k$  for any $k\in\NN$, where $D_k$ is $\diag(\zeta_k(\phi_\a)/\|\x_\a\|^k)$.  Then, adopting the proof of \eqref{eq:Tk}, we can similarly conclude that 
 \begin{equation*}
    \|T_k-\eta_k^2(\sigma)f_{k+1}(X^\top X)\|\le C\varepsilon_n,
\end{equation*}for some constant $C$ and $k=0,1,2$, when $n$ is sufficiently large. Likewise, \eqref{eq:off_diagonal_bounds} indicates \[ \left\|\sum_{k\ge 3}(T_k-\diag(T_k))\right\|\le C\varepsilon_n^4 n,\] and a similar proof of \eqref{eq:diag_control} implies that 
\[ \left\|\sum_{k\ge 3 }\diag (T_k)-\left(a_\sigma-\sum_{k=0}^2\eta_k^2(\sigma)\right)\Id\right\|\le C\varepsilon_n.\]
Based on these approximations, we can conclude the result of this lemma.
 \end{proof}

\begin{proof}[Proof of Theorem \ref{thm:law_NTK}]
The first part of the statement is a straight consequence of \eqref{eq:NTK_second_concentration} and Theorem \ref{thm:law_DNN}. Denote by $A:=\sqrt{\frac{d_1}{n}}\left(H-\E[H]\right)$ and $B:=\sqrt{\frac{d_1}{n}}\left(\frac{1}{d_1}Y^\top Y-\Phi\right)$. Observe that \[B-A=\frac{1}{\sqrt{nd_1}}\left[(S^\top S)\odot(X^\top X)-\E[(S^\top S)\odot(X^\top X)]\right].\] Hence, \eqref{eq:NTK_second_concentration} indicates  $\|B-A\|\to 0$ as $n\to\infty$. This convergence implies that limiting laws of $A$ and $B$ are identical because of Lemma \ref{lem:BS10A45}.

The second part is because of Lemma \ref{lemma:operator_Phi} and Lemma \ref{lemma:Psi_0}. From \eqref{eq:Hmatrixform} and \eqref{eq:defPsi},  $\E[H]=\Phi+\Psi.$
Then almost surely,
\begin{align*}
   & \left\|\sqrt{\frac{d_1}{n}}\left(H-\E[H]\right)-\sqrt{\frac{d_1}{n}}\left(H-\Phi_0-\Psi_0\right)\right\|=\sqrt{\frac{d_1}{n}}\left\|\Phi_0+\Psi_0-\E[H]\right\|\\
   \le ~& \sqrt{\frac{d_1}{n}}\left(\|\Phi-\Phi_0\|+\|\Psi-\Psi_0\|\right)\le \sqrt{\frac{d_1}{n}}\left(\sqrt{n}\varepsilon_n^2+n\varepsilon_n^4\right)\to 0,
\end{align*}as $\varepsilon_n^4d_1\to 0$ by the assumption of Theorem \ref{thm:law_NTK}. Therefore, the limiting eigenvalue distribution of \eqref{eq:NTK_2} is the same as \eqref{eq:NTK_1}. 
\end{proof}

\section{Proof of the concentration for extreme eigenvalues}\label{sec:nonasymptotic}
In this section, we obtain the estimates of the extreme eigenvalues for the CK and NTK we studied in Section \ref{sec:1hlNN}. The limiting spectral distribution of $\frac{1}{\sqrt{d_1n}} (Y^\top Y-\E [Y^\top Y])$ tells us the bulk behavior of the spectrum. An estimation of the extreme eigenvalues will show that the eigenvalues are confined in a finite interval with high probability.
We first provide a non-asymptotic bound on the concentration of $\frac{1}{d_1}Y^\top Y$ under the spectral norm. The proof is based on the Hanson-Wright inequality we proved in Section~\ref{sec:nonlinearHW} and an $\varepsilon$-net argument.  
 
\begin{proof}[Proof of Theorem \ref{thm:spectralnorm}]
 Recall notations in Section \ref{sec:intro}. Define 
\begin{align*}
    M:=~&\frac{1}{\sqrt{d_1n}}Y^{\top}Y=\frac{1}{\sqrt{d_1n}}\sum_{i=1}^{d_1} \y_i \y_i^\top, \notag \\
    M-\mathbb E M =~&\frac{1}{\sqrt{d_1n}}\sum_{i=1}^{d_1} (\y_i \y_i^\top-\mathbb E [\y_i \y_i^\top])=\frac{1}{\sqrt{d_1n}}\sum_{i=1}^{d_1} (\y_i \y_i^\top-\Phi),
\end{align*}
where $\y_i^\top=\sigma ( \w_i^{\top} X)$.
 
For any fixed $\z\in \mathbb S^{n-1}$, we have
\begin{align}
    \z^{\top} (M-\mathbb E M)\z&=\frac{1}{\sqrt{d_1n}}\sum_{i=1}^{d_1}
[\langle \z, \y_i \rangle ^2-\z^\top \Phi \z ]\nonumber\\
&=\frac{1}{\sqrt{d_1n}}\sum_{i=1}^{d_1} [\y_i^{\top} (\z\z^{\top})\y_i-\Tr(\Phi \z\z^{\top})]\nonumber\\
&=(\y_1,\dots, \y_{d_1})^{\top}A_{\z} (\y_1,\dots, \y_{d_1})-\Tr(A_{\z}\tilde \Phi),\label{eq:zAz}
\end{align}
where 
\begin{align*}
    A_{\z}=\frac{1}{\sqrt{d_1n}} \begin{bmatrix}
    \z\z^{\top} & &\\
     &  \ddots &\\
     &  & \z\z^{\top}
    \end{bmatrix} \in \mathbb R^{nd_1\times nd_1}, \quad  \tilde{\Phi}=\begin{bmatrix}
    \Phi & & \\
    &  \ddots & \\
    & & \Phi
    \end{bmatrix} \in \R^{nd_1\times nd_1},
\end{align*}
and column vector $(\y_1,\dots, \y_{d_1})\in \R^{nd_1}$ is the concatenation of column vectors $\y_1,\dots,\y_{d_1}$. Then 
\begin{align*}
    (\y_1,\dots, \y_{d_1})^{\top}=\sigma((\w_1,\dots, \w_{d_1})^{\top}\tilde X) 
\end{align*}
with block matrix
\begin{align*}
    \tilde{X}=\begin{bmatrix}
    X & &\\
     &  \ddots &\\
     &  & X
    \end{bmatrix}.
\end{align*}
Notice that 
\begin{align*}
    \|A_{\z}\|=\frac{1}{\sqrt{d_1n}}, \quad \|A_{\z}\|_F=\frac{1}{\sqrt n}, \quad \|\tilde{X}\|=\|X\|.
\end{align*}
Denote $\tilde{\y}=(\y_1,\dots, \y_{d_1})$. With \eqref{eq:Eybound}, we obtain
\begin{align*}
    \|\E\tilde{\y}\|^2&=d_1\|\E\y\|^2\leq d_1 \left( 2\lambda_{\sigma}^2\sum_{i=1}^n (\|\x_i\|^2-1 )^2+2n(\E\sigma(\xi))^2\right)\\
    &=  d_1 \left( 2\lambda_{\sigma}^2\sum_{i=1}^n (\|\x_i\|^2-1 )^2\right)\leq 2d_1\lambda_{\sigma}^2 B^2,
\end{align*}
where the last line is from the assumptions on $X$ and $\sigma$. When $B\not=0$, applying \eqref{eq:inparticular} to \eqref{eq:zAz}  implies
\begin{align*}
     &\P \left( |(\y_1,\dots, \y_{d_1})^{\top}A_{\z} (\y_1,\dots, \y_{d_1})-\Tr(A_{\z}\tilde \Phi)|\geq t\right)\\
    \leq ~& 2\exp \left(-\frac{1}{C} \min \left\{ \frac{t^2n}{8 \lambda_{\sigma}^4 \|X\|^4}, \frac{t\sqrt{d_1n}}{\lambda_{\sigma}^2 \|X\|^2}\right\}\right) + 2\exp \left( -\frac{t^2d_1 n}{32\lambda_{\sigma}^2\|X\|^2 \|\E \tilde\y\|^2  }\right)\\
    \leq~ &2\exp \left(-\frac{1}{C} \min \left\{ \frac{t^2n}{8 \lambda_{\sigma}^4 \|X\|^4}, \frac{t\sqrt{d_1n}}{\lambda_{\sigma}^2 \|X\|^2}\right\}\right) + 2\exp \left( -\frac{t^2 n}{64\lambda_{\sigma}^4 B^2\|X\|^2  }\right).
\end{align*}
Let $\mathcal N$ be a $1/2$-net on $\mathbb S^{n-1}$ with $|\mathcal N|\leq 5^n$ (see e.g. \cite[Corollary 4.2.13]{vershynin2018high}),  then
\begin{align*}
    \|M-\mathbb EM\| \leq 2\sup_{\z\in \mathcal N}|\z^{\top} (M-\mathbb E M)\z|.
\end{align*}
Taking a union bound over $\mathcal N$ yields
\begin{align*}
   \mathbb P (\|M-\mathbb EM\|\geq 2t)\leq~ & 2\exp \left( n\log 5-\frac{1}{C} \min \left\{ \frac{t^2n}{16 \lambda_{\sigma}^4 \|X\|^4}, \frac{t\sqrt{d_1n}}{2\lambda_{\sigma}^2 \|X\|^2}\right\}\right)\\
   &+2\exp \left( n\log 5-\frac{t^2 n}{64\lambda_{\sigma}^4 B^2\|X\|^2  }\right).
\end{align*}
We then can set
\[t=\left(8\sqrt{C}+ 8C\sqrt{\frac{n}{d_1}}\right)\lambda_{\sigma}^2 \|X\|^2+16 B\lambda_{\sigma}^2 \|X\|,
\]
to conclude
\begin{align*}
 \mathbb P \left(\|M-\mathbb EM\|\geq \left(16\sqrt{C}+ 16C\sqrt{\frac{n}{d_1}}\right)\lambda_{\sigma}^2 \|X\|^2+32 B\lambda_{\sigma}^2 \|X\|\right)\leq  4e^{-2n}.
\end{align*}
Since
\begin{align*}
     \left\| \frac{1}{d_1}Y^{\top}Y-\Phi \right\|=\sqrt{\frac{n}{d_1}}\|M-\E M\|,
\end{align*}
the upper bound in \eqref{eq:d_1Y} is then verified. When $B=0$, we can apply \eqref{eq:zero-meanbound} and follow the same steps to get the desired bound.
\end{proof}

By the concentration inequality in Theorem \ref{thm:spectralnorm}, we can get a lower bound on the smallest eigenvalue of the conjugate kernel $\frac{1}{d_1} Y^\top Y$ as follows.
\begin{lemma} \label{lem:lambdamin}
 Assume $X$ satisfies
$ \sum_{i=1}^n (\|\x_i\|^2-1)^2\leq B^2$ for a constant $B>0$,
 and $\sigma$ is $\lambda_{\sigma}$-Lipschitz with  $\E \sigma(\xi)=0$.  Then with probability at least $1-4e^{-2n}$,
\begin{align}\label{eq:lowerY}
    \lambda_{\min}\left(\frac{1}{d_1}Y^\top Y\right)\geq  \lambda_{\min}(\Phi)-C\left(\sqrt{\frac{n}{d_1}}+\frac{n}{d_1}\right)\lambda_{\sigma}^2\|X\|^2-32 B\lambda_{\sigma}^2 \|X\|\sqrt{\frac{n}{d_1}} .
\end{align}
 
\end{lemma}

\begin{proof}
By Weyl's inequality \cite[Corollary A.6]{anderson2010introduction}, we have 
\begin{align*}
   \left| \lambda_{\min}\left(\frac{1}{d_1}Y^\top Y\right)-\lambda_{\min}(\Phi)\right|\leq \left \|\frac{1}{d_1}Y^\top Y-d_1\Phi \right\|. 
\end{align*}
Then \eqref{eq:lowerY} follows from  \eqref{eq:d_1Y}.
\end{proof}

The lower bound in \eqref{eq:lowerY} relies on $\lambda_{\min}(\Phi)$.
Under certain assumptions on $X$ and $\sigma$, we can guarantee that $\lambda_{\min}(\Phi)$ is bounded below by an absolute constant.

\begin{lemma}\label{lem:sup}
Assume $\sigma$ is not a linear function and $\sigma(x)$ is Lipschitz. Then 
\begin{align}\label{eq:suphermite}
    \sup \{k\in\mathbb{N}: \zeta_k(\sigma)^2>0\}=\infty.
\end{align}
\end{lemma}
\begin{proof}
Suppose that $ \sup \{k\in\mathbb{N}: \zeta_k(\sigma)^2>0\}$ is finite. Then $\sigma$ is a polynomial of degree at least $2$ from our assumption, which is a contradiction to the fact that $\sigma$ is Lipschitz. Hence, \eqref{eq:suphermite} holds.
\end{proof}

\begin{lemma}\label{thm:minlowerbound}
Assume Assumption \ref{assump:sigma} holds, $\sigma$ is not a linear function, and $X$ satisfies $(\varepsilon_n, B)$-orthonormal property.
Then,  
\begin{align}\label{eq:lambdaminlowerbound}
    \lambda_{\min}(\Phi)\geq 1-\zeta_1(\sigma)^2-\zeta_2(\sigma)^2-\zeta_3(\sigma)^2-C_B\eps_n^2\sqrt{n}.
\end{align}
\begin{remark}
This bound will not hold when $\sigma$ is a linear function. Suppose $\sigma$ is a linear function, under Assumption \ref{assump:sigma},  we must have $\sigma(x)=x$ and $\Phi=X^\top X$. Then we will not have a lower bound on $\lambda_{\min}(\Phi)$ based on the Hermite coefficients of $\sigma$. 
\end{remark}
 
\end{lemma}

\begin{proof}[Proof of Lemma \ref{thm:minlowerbound}]
From Lemma \ref{lemma:operator_Phi}, under our assumptions, we know that
\[
    \|\Phi-\Phi_0\|\leq C_B\varepsilon_n^2 \sqrt{n}.
\]
where  $\Phi_0$ is given by \eqref{def:Phi0}. Thus,
$
    \lambda_{\min} (\Phi)\geq \lambda_{\min}(\Phi_0)-C_B\varepsilon_n^2 \sqrt{n}, 
$
 
and,  from Weyl's inequality \cite[Theorem A.5]{anderson2010introduction}, we have
 \begin{align*}
     \lambda_{\min}(\Phi_0)\geq \sum_{k=1}^3\zeta_k(\sigma)^2 \lambda_{\min}(f_k(X^\top X))+(1-\zeta_1(\sigma)^2-\zeta_2(\sigma)^2-\zeta_3(\sigma)^2).
 \end{align*}
 Note that $
     f_k(X^\top X)=K_k^\top K_k,$ 
 where $K_k\in \mathbb R^{d_0^k\times n}$, and each column of $K_k$ is given by the $k$-th Kronecker product  $\x_i\otimes \cdots \otimes \x_i$.  Hence, $f_k(X^\top X)$ is positive semi-definite.
 
Therefore,
\begin{align*}
    \lambda_{\min}(\Phi_0)\geq (1-\zeta_1(\sigma)^2-\zeta_2(\sigma)^2-\zeta_3(\sigma)^2).
\end{align*}
Since $\sigma$ is not linear but Lipschitz, \eqref{eq:suphermite} holds for $\sigma$. 
Therefore, $(1-\zeta_1(\sigma)^2-\zeta_2(\sigma)^2-\zeta_3(\sigma)^2)=\sum_{k=4}^{\infty}\zeta_k(\sigma)^2>0$, and then \eqref{eq:lambdaminlowerbound} holds.
 
\end{proof}
Theorem \ref{thm:lowerKn} then follows directly from Lemma  \ref{lem:lambdamin} and Lemma \ref{thm:minlowerbound}.

\bigskip 

Next, we move on to  non-asymptotic estimations for NTK. Recall that the empirical NTK matrix $H$ is given by \eqref{eq:Hmatrixform} and the $\alpha$-th column of $S$ is defined by $\diag (\sigma'(W\x_\alpha))\aa$, for $1\le \alpha\le n$, in \eqref{eq:defcolumnS}. 
 
The $i$-th row of $S$ is given by
$
   \z_i^\top:= \sigma'(\w_i^{\top} X)a_i,
$
and $\E [\z_i]=0$, where $a_i$ is the $i$-th entry of $\aa$. Define 
$
    D_{\alpha}=\diag (\sigma'(\w_{\alpha}^{\top} X)a_{\alpha}),
$ for $1\le \alpha\le d_1$. We can rewrite $(S^\top S)\odot(X^\top X)$ as 
 \[(S^\top S)\odot(X^\top X)=\sum_{\a=1}^{d_1}a_\a^2 D_\a X^\top XD_\a.\]
Let us define $L$ and further expand it as follows:
\begin{align}
   L&:=\frac{1}{d_1} (S^\top S-\E[S^\top S])\odot (X^\top X) \label{eq:defLL}\\
   &=\frac{1}{d_1}\sum_{i=1}^{d_1} (\z_i \z_i^\top -\E [\z_i \z_i^\top])\odot (X^\top X)\notag\\
    &=\frac{1}{d_1}\sum_{i=1}^{d_1} \left( D_i (X^\top X) D_i-\E[ D_i (X^\top X) D_i]\right)=\frac{1}{d_1}\sum_{i=1}^{d_1} Z_i.\label{eq:defZi}
\end{align}
Here $Z_i$ is a centered random matrix, and we can apply matrix Bernstein's inequality to show the concentration of $L$. Since $Z_i$ does not have an almost sure bound on the spectral norm, we will use the following sub-exponential version of the matrix Bernstein inequality from \cite{tropp2012user}.

\begin{lemma}[\cite{tropp2012user}, Theorem 6.2]
Let $Z_k$ be independent Hermitian matrices of size $n\times n$. Assume 
\begin{align*}
    \E Z_i=0,  \quad \left\|\E[Z_i^p]\right\|\leq \frac{1}{2} p! R^{p-2} a^2,
\end{align*}
for any integer $p\geq 2$. Then  for all $t\geq 0$,
\begin{align}\label{eq:Tropp}
    \mathbb P \left(~ \left\|\sum_{i=1}^{d_1} Z_i \right\|\geq t  \right) \leq n \exp \left(-\frac{t^2}{2d_1a^2 +2Rt} \right).
\end{align}
\end{lemma}

\begin{proof}[Proof of Theorem \ref{thm:NTK_concentration}]
From \eqref{eq:defZi}, $\E Z_i=0$, and 
\begin{align*}
   \|Z_i\|&\leq \|D_i\|^2 \|XX^\top\| +\E  \|D_i\|^2 \|XX^\top\|\leq   C_1(a_i^2+1),
\end{align*}
where $C_1=\lambda_{\sigma}^2\|X\|^2$ and where $a_i\sim\N(0,1)$ is the $i$-th entry of the second layer weight $\aa$. Then 
\begin{align*}
\|\E[Z_i^p]\|\leq \E \|Z_i\|^p &\leq C_1^{2p}   \E(a_i^2+1)^{p}\leq  C_1^{2p} \sum_{k=1}^p \binom{p}{k}(2k-1)!!\\
&=  C_1^{2p} p! \sum_{k=1}^p \frac{(2k-1)!!}{k! (p-k)!}\leq C_1^{2p} p! \sum_{k=1}^p 2^k\leq 2(2C_1^2)^p p!.
\end{align*}
So we can take $R=2C_1^2, a^2=8C_1^4$  in \eqref{eq:Tropp} and obtain
\begin{align*}
    \mathbb P \left( \left\|\sum_{i=1}^{d_1} Z_i \right\|\geq t  \right) \leq n \exp \left(-\frac{t^2}{16d_1C_1^4 +4C_1^2t} \right).
\end{align*}
Hence, $L$ defined in \eqref{eq:defLL} has a probability bound:
\begin{align*}
   \mathbb P\left( \|L\|\geq t \right) =\mathbb P \left( \frac{1}{d_1}\left\|\sum_{i=1}^{d_1} Z_i \right\|\geq t  \right) \leq n \exp \left(-\frac{t^2d_1}{16C_1^4 +4C_1^2t} \right).
\end{align*}
Take $t=10C_1^2\sqrt{\log n/d_1}$. Under the assumption that $d_1\geq \log n$, we conclude that, with high probability at least $1-n^{-7/3}$,
\begin{align}\label{eq:boundL}
    \|L\|\leq 10C_1^2\sqrt{\frac{\log n}{d_1}}.
\end{align}
Thus, as a corollary, the  two statements in Lemma \ref{lem:NTKexpectation} follow from \eqref{eq:boundL}. Meanwhile, since \[\|H-\E H\| \leq \left\| \frac{1}{d_1} Y^\top Y-\Phi\right\| +\|L\|,\]
the bound in \eqref{eq:YEY2} follows from Theorem \ref{thm:spectralnorm} and \eqref{eq:boundL}.
\end{proof}

We now proceed to provide a lower bound of $\lambda_{\min}(H)$ from Theorem \ref{thm:NTK_concentration}.

\begin{proof}[Proof of Theorem \ref{thm:NTK_Hlowerbound}]
Note that from  \eqref{eq:Hmatrixform}, \eqref{eq:defPsi} and \eqref{eq:defLL}, we have
\begin{align*}
    \lambda_{\min} (H)&\geq \frac{1}{d_1} \lambda_{\min}((S^\top S)\odot(X^\top X)) \\
   &\geq \frac{1}{d_1}\lambda_{\min} ((\E S^\top S)\odot(X^\top X))- \|L\|=\lambda_{\min}(\Psi)-\|L\|.
   \end{align*}
  Then with Lemma \ref{lemma:Psi_0}, we can get
  \begin{align*} 
 \lambda_{\min} (H)&\geq      \lambda_{\min}(\Psi_0)-C\varepsilon_n^4 n-\|L\|\geq \left(a_\sigma-\sum_{k=0}^2\eta_{k}^2(\sigma)\right)-C\varepsilon_n^4 n-\|L\|.
  \end{align*}
  Therefore, from Theorem \ref{thm:NTK_concentration}, with probability at least $1-n^{-7/3}$,
  \begin{align*}
      \lambda_{\min} (H)&\geq a_\sigma-\sum_{k=0}^2\eta_{k}^2(\sigma)-C\varepsilon_n^4 n- 10\lambda_{\sigma}^4 \|X\|^4\sqrt{\frac{\log n}{d_1}}\\
      &\geq a_\sigma-\sum_{k=0}^2\eta_{k}^2(\sigma)-C\varepsilon_n^4 n- 10\lambda_{\sigma}^4 B^4\sqrt{\frac{\log n}{d_1}}.
  \end{align*}

Since $\sigma$ is Lipschitz and non-linear, we know $\sigma'(x)$ is not a linear function (including the constant function) and $|\sigma'(x)|$ is bounded. Suppose that $\sigma'(x)$ has finite many non-zero Hermite coefficients, $\sigma(x)$ is a polynomial, then we get a contradiction. Hence, the Hermite coefficients of $\sigma'$ satisfy
\begin{align}
    \sup \{k\in\mathbb{N} :\eta_k^2(\sigma)>0\}=\infty ~\text{ and }~ a_\sigma-\sum_{k=0}^2\eta_{k}^2(\sigma)=\sum_{k=3}^{\infty}\eta_{k}^2(\sigma)>0.
\end{align}
This finishes the proof.
\end{proof}

\section{Proof of Theorem \ref{thm:train_diff} and Theorem \ref{thm:limit_error}}\label{sec:generalizationerror}

By definitions, the random matrix $K_n(X,X)$ is $\frac{1}{d_1}Y^\top Y$ and the kernel matrix $K(X,X)=\Phi$ is defined in \eqref{def:phi}. These two matrices have been already analyzed in Theorem \ref{thm:spectralnorm} and Theorem \ref{thm:lowerKn}, so we will apply these results to estimate how great the difference between training errors of random feature regression and its corresponding kernel regression.

\begin{proof}[Proof of Theorem \ref{thm:train_diff}]
Denote $K_\lambda:=(K+\lambda\Id)$. From the definitions of training errors in \eqref{eq:Etrain_Kn} and \eqref{eq:Etrain_K}, we have
\begin{align}
     &\left| E_{\train}^{(RF,\lambda)}-E_{\train}^{(K,\lambda)}\right| 
     =\frac{1}{n}\left| \|\hat{f}_{\lambda}^{(RF)}(X)-\v y\|^2-\|\hat{f}_{\lambda}^{(K)}(X)-\v y\|^2 \right|\notag \\
     = ~&\frac{\lambda^2}{n} \left|\Tr[(K(X,X)+\lambda\Id)^{-2} \v y\v y^\top]-\Tr[(K_n(X,X)+\lambda\Id)^{-2} \v y\v y^\top] \right| \notag \\
     = ~&\frac{\lambda^2}{n} \left|\v y^\top\left[(K(X,X)+\lambda\Id)^{-2}-(K_n(X,X)+\lambda\Id)^{-2}\right] \v y\right| \notag \\
     \leq~ &\frac{\lambda^2}{n} \|(K(X,X)+\lambda\Id)^{-2}-(K_n(X,X)+\lambda\Id)^{-2}\|\cdot  \|\v y\|^2 \notag\\
     \leq~ &\frac{\lambda^2\|\v y\|^2}{n\lambda_{\min}^{2}(K(X,X)) \lambda_{\min}^{2}(K_n(X,X))}    \| (K_\lambda^2-(K_n(X,X)+\lambda\Id)^2 \|  .\label{eq:719}
\end{align}
Here, in \eqref{eq:719}, we employ the identity
\begin{equation}\label{eq:a-1_b-1}
    A^{-1}-B^{-1}=B^{-1}(B-A)A^{-1},
\end{equation}
for $A=(K(X,X)+\lambda\Id)^{-2}$ and $B=(K_n(X,X)+\lambda\Id)^{-2}$, and the fact that
$\|(K(X,X)+\lambda\Id)^{-1}\|\le \lambda_{\min}^{-1}(K(X,X))$ and $(K_n(X,X)+\lambda\Id)^{-1}\|\le \lambda_{\min}^{-1}(K_n(X,X))$. Next, before providing uniform upper bounds for $\lambda_{\min}^{-2}(K(X,X))$ and $\lambda_{\min}^{-2}(K_n(X,X))$ in \eqref{eq:719}, we can first get a bound for the last term of \eqref{eq:719} as follows: 
\begin{align}
  ~&\| (K(X,X)+\lambda\Id)^2-(K_n(X,X)+\lambda\Id)^2 \| \notag \\
  =~&\| K^2(X,X)-K_n^2(X,X) +2\lambda (K(X,X)-K_n(X,X))\| \notag \\
 \leq~ & \| K^2(X,X)-K_n^2(X,X)\| +2\lambda \| (K(X,X)-K_n(X,X))\| \notag\\
 \leq~ &\Big(\|K_n(X,X)-K(X,X)\| +2\|K(X,X)\|+ 2\lambda\Big)\cdot\|K(X,X)-K_n(X,X)\|\notag\\
 \leq ~& C\left(\sqrt{\frac{n}{d_1}}+C\right)\sqrt{\frac{n}{d_1}}\label{eq:KKKK}.
\end{align}
for some constant $C>0$, with  probability at least $1-4 e^{-2n}$, where the last bound in \eqref{eq:KKKK} is due to Theorem~\ref{thm:spectralnorm} and Lemma \ref{lem:Kupperbound} in Appendix~\ref{appendix:lemmas}. Additionally, combining Theorem~\ref{thm:spectralnorm} and Theorem \ref{thm:lowerKn}, we can easily get 
\begin{equation}\label{eq:lambda_min}
    \|(K_n(X,X)+\lambda\Id)^{-1}\|\le\lambda_{\min}^{-1}(K_n(X,X))\le C
\end{equation} 
for all large $n$ and some universal constant $C$, under the same event that \eqref{eq:KKKK} holds. Theorem \ref{thm:minlowerbound} also shows $\lambda_{\min}^{-1}(K(X,X))\le C$ for all large $n$. Hence, with the upper bounds for $\lambda_{\min}^{-2}(K(X,X))$ and $\lambda_{\min}^{-2}(K_n(X,X))$, \eqref{eq:ETrain} follows from the bounds of \eqref{eq:719} and \eqref{eq:KKKK}.
\end{proof}

For ease of notation, we denote $K:=K(X,X)$ and $K_n:=K_n(X,X)$. Hence, from \eqref{eq:def_testerror}, we can further decompose the test errors for $K$ and $K_n$ into
\begin{align}
    \mathcal L(\hat f^{(K)}_\lambda)
    = ~&\E_\x [|f^*(\x)|^2]
    +\Tr\left[(K+\lambda\Id)^{-1}\v y\v y^\top (K+\lambda\Id)^{-1}\E_\x[K(\x,X)^\top K(\x,X)]\right]\label{eq:test_K}\\
    ~&~-2\Tr\left[(K+\lambda\Id)^{-1}\v y\E_\x[f^*(\x)K(\x,X)]\right],\notag\\
   \mathcal L(\hat f^{(RF)}_\lambda)
    =~  &\E_\x [|f^*(\x)|^2]
    +\Tr\left[(K_n+\lambda\Id)^{-1}\v y\v y^\top (K_n+\lambda\Id)^{-1}\E_\x[K_n(\x,X)^\top K_n(\x,X)]\right]\label{eq:test_RF}\\
    ~&~-2\Tr\left[(K_n+\lambda\Id)^{-1}\v y\E_\x[f^*(\x)K_n(\x,X)]\right].\notag
\end{align}
Let us denote
\begin{align*}
        E_1:=&\Tr\left[(K_n+\lambda\Id)^{-1}\v y\v y^\top (K_n+\lambda\Id)^{-1}\E_\x[K_n(\x,X)^\top K_n(\x,X)]\right],\\
        \bar E_1:=&\Tr\left[(K+\lambda\Id)^{-1}\v y\v y^\top (K+\lambda\Id)^{-1}\E_\x[K(\x,X)^\top K(\x,X)]\right],\\
       E_2:=&\Tr\left[(K_n+\lambda\Id)^{-1}\v y\bbeta^{*\top}\E_\x[\x K_n(\x,X)]\right],\\
        \bar E_2:=&\Tr\left[(K+\lambda\Id)^{-1}\v y\bbeta^{*\top}\E_\x[\x K(\x,X)]\right].
\end{align*}
As we can see, to compare the test errors between random feature and kernel regression models, we need to control $|E_1-\bar E_1|$ and $|E_2-\bar E_2|$. Firstly, it is necessary to study the concentrations of
\[\E_\x[K(\x,X)^\top K(\x,X)-K_n(\x,X)^\top K_n(\x,X)]\]
and 
\[\E_\x\left[f^*(\x)\left(K(\x,X)-K_n(\x,X)\right)\right].\]
\begin{lemma}\label{lemma:K_n}
Under Assumption \ref{assump:sigma} for $\sigma$ and Assumption \ref{assump:testdata} for $\x$ and $X$, with probability at least $1-4e^{-2n}$, we have
 \begin{align}\label{eq:K_n}
     \left\| K_n(\x,X)-K(\x,X) \right\|\leq C\sqrt{\frac{n}{d_1}},
 \end{align}
 where $C>0$ is a universal constant. Here, we only consider the randomness of the weight matrix in $K_n(\x,X)$ defined by \eqref{eq:K_n_def} and \eqref{eq:K_n_x_def}. 
\end{lemma}
\begin{proof}
We consider $\tilde X=[\x_1,\ldots,\x_n,\x] $, its corresponding kernels $K_{n}(\tilde X, \tilde X)$ and $K(\tilde X, \tilde X)\in\R^{(n+1)\times (n+1)}$. Under Assumption \ref{assump:testdata}, we can directly apply Theorem \ref{thm:spectralnorm} to get the concentration of $K_{n}(\tilde X, \tilde X)$ around $K(\tilde X, \tilde X)$, namely,
\begin{equation}\label{eq:K_n_tilde_X}
    \left\| K_n(\tilde X, \tilde X)-K(\tilde X, \tilde X) \right\|\leq C\sqrt{\frac{n}{d_1}},
\end{equation}
 with probability at least $1-4e^{-2n}$. Meanwhile, we can write $K_{n}(\tilde X, \tilde X)$ and $K(\tilde X, \tilde X)$ as block matrices:
 \[K_{n}(\tilde X, \tilde X)=\begin{pmatrix}
  K_{n}( X, X) & K_{n}(X, \x)\\
  K_{n}(\x, X) & K_n(\x,\x)
 \end{pmatrix} ~\text{ and }~
 K(\tilde X, \tilde X)=\begin{pmatrix}
  K( X, X) & K(X, \x)\\
  K(\x, X) & K(\x,\x)
 \end{pmatrix}.\]
 Since the $\ell_2$-norm of any row is bounded above by the spectral norm of its entire matrix, we complete the proof of  \eqref{eq:K_n}. 
\end{proof}

\begin{lemma}\label{lemma:quadratic_ybeta}
Assume that training labels satisfy Assumption~\ref{assump:target} and  $\|X\|\le B$, then for any deterministic $A\in\R^{n\times n}$, we have
\begin{align*}
    \Var\left(\v y^\top A\v y\right),\Var\left(\bbeta^{*\top} A\v y\right)\le c\|A\|_F^2,
\end{align*}where constant $c$ only depends on $\sigma_{\bbeta},$ $\sigma_{\veps}$ and $B$. Moreover, \[\E[ \v y^\top A\v y]=\sigma^2_{\bbeta}\Tr AX^\top X+\sigma^2_{\veps}\Tr A, \quad  \E[\bbeta^{*\top} A\v y]=\sigma^2_{\bbeta}\Tr AX^\top.\]
\end{lemma}
\begin{proof}We follow the idea in Lemma C.8 of \cite{mei2019generalization} to investigate the variance of the quadratic form for the Gaussian random vector by
\begin{equation}\label{eq:var}
    \Var ( \v g^{\top} A  \v g)=\|A\|_{F}^2+\Tr(A^2)\le 2\|A\|_F^2,
\end{equation}
for any deterministic square matrix $A$ and standard normal random vector $\v g$. Notice that the quadratic form 
\begin{equation}
    \v y^\top A\v y=\v g^\top \begin{pmatrix}
     \sigma_{\bbeta}^{2} XAX^\top & \sigma_{\veps}\sigma_{\bbeta} XA\\
     \sigma_{\veps}\sigma_{\bbeta} AX^\top &  \sigma_{\veps}^{2}A
    \end{pmatrix}\v g,
\end{equation}
where $\v g$ is a standard Gaussian random vector in $\R^{d_0+n}$. Similarly, the second quadratic form can be written as
\begin{equation*}
    \bbeta^{*\top} A\v y=\v g^\top \begin{pmatrix}
      \sigma_{\bbeta}^2 AX^\top & \sigma_{\veps}\sigma_{\bbeta} A\\
     \mathbf{0} &\mathbf{0}
    \end{pmatrix}\v g.
\end{equation*} 
Let \[\tilde A_1:=\begin{pmatrix}
     \sigma_{\bbeta}^{2} XAX^\top & \sigma_{\veps}\sigma_{\bbeta} XA\\
     \sigma_{\veps}\sigma_{\bbeta} AX^\top &  \sigma_{\veps}^{2}A
    \end{pmatrix},\quad\tilde A_2:=\begin{pmatrix}
      \sigma_{\bbeta}^2 AX^\top & \sigma_{\veps}\sigma_{\bbeta} A\\
     \mathbf{0} &\mathbf{0}
    \end{pmatrix}.\]
By \eqref{eq:var}, we know $\Var\left(\v y^\top A\v y\right)\le 2\|\tilde A_1\|_F^2$ and $\Var\left(\bbeta^{*\top} A\v y\right)\le 2\|\tilde A_2\|_F^2$. Since $$\|\tilde A_1\|^2_F=\sigma_{\bbeta}^4\|XAX^\top\|^2_F+\sigma^2_{\veps}\sigma^2_{\bbeta} \|XA\|^2_F+\sigma^2_{\veps}\sigma^2_{\bbeta} \|AX^\top\|_F^2 + \sigma_{\veps}^{4}\|A\|_F^2\le c\|A\|_F^2$$ and similarly $\|\tilde A_2\|_F\le c\|A\|_F^2$ for a constant $c$, we can complete the proof. 
\end{proof}
As a remark, in Lemma \ref{lemma:quadratic_ybeta}, for simplicity, we only provide a variance control for the quadratic forms to obtain convergence in probability in the following proofs of Theorems \ref{thm:test_diff} and \ref{thm:limit_error}. However, we can actually apply Hanson-Wright inequalities in Section \ref{sec:nonlinearHW} to get more precise probability bounds and consider non-Gaussian distributions for $\bbeta^*$ and $\veps $.

\begin{proof}[Proof of Theorem \ref{thm:test_diff}]
Based on the preceding expansions of $\mathcal L(\hat{f}_{\lambda}^{(RF)}(\x))$ and $\mathcal L (\hat{f}_{\lambda}^{(K)}(\x))$ in \eqref{eq:test_K} and \eqref{eq:test_RF}, we need to control the right-hand side of
\begin{align*}
   & \left| \mathcal L(\hat{f}_{\lambda}^{(RF)}(\x))-\mathcal L (\hat{f}_{\lambda}^{(K)}(\x))\right|
\le  \left| E_1-\bar E_1\right|+2\left|\bar E_2 - E_2\right|.
\end{align*}
In the subsequent procedure, we first take the concentrations of $E_1$ and $E_2$ with respect to normal random vectors $\bbeta^*$ and $\veps$, respectively. Then, we apply Theorem~\ref{thm:spectralnorm} and Lemma~\ref{lemma:K_n} to complete the proof of \eqref{eq:test_diff}. For simplicity, we start with the second term 
\begin{align}
\left|\bar E_2 - E_2\right|\le~ &  \left|\bbeta^{*\top}\E_\x[\x (K_n(\x,X)-K(\x,X))](K_n+\lambda\Id)^{-1}\v y\right|\nonumber\\
    ~&+\left|\bbeta^{*\top}\E_\x[\x K(\x,X)]\left((K_n+\lambda\Id)^{-1}-(K+\lambda\Id)^{-1}\right)\v y\right|\nonumber\\
    \le ~& |I_1- \bar I_1|+|I_2- \bar I_2|+| \bar I_1|+| \bar I_2|,\label{eq:E_2-E_2}
\end{align}
where $I_1$ and $I_2$ are quadratic forms defined below
\begin{align*}
    I_1:= ~& \bbeta^{*\top}\E_\x[\x (K_n(\x,X)-K(\x,X))](K_n+\lambda\Id)^{-1}\v y,\\
    I_2:= ~& \bbeta^{*\top}\E_\x[\x K(\x,X)]\left((K_n+\lambda\Id)^{-1}-(K+\lambda\Id)^{-1}\right)\v y,
\end{align*}
and their expectations with respect to random vectors $\bbeta^*$ and $\veps$ are denoted by
\begin{align*} 
    \bar I_1:=\E_{\veps,\, \bbeta^*}[I_1]=~&\sigma_{\bbeta}^{2}\Tr\left(\E_\x[\x (K_n(\x,X)-K(\x,X))](K_n+\lambda\Id)^{-1} X^\top\right),\\
    \bar I_2:=\E_{\veps,\,\bbeta^*}[I_2]=~&\sigma_{\bbeta}^{2}\Tr\left(\left((K_n+\lambda\Id)^{-1}-(K+\lambda\Id)^{-1}\right)X^\top \E_\x[\x K(\x,X)] \right).
\end{align*}

We first consider the randomness of the weight matrix in $K_n$ and define the event $\cE$ where both \eqref{eq:lambda_min} and \eqref{eq:K_n_tilde_X} hold. Then, Theorem \ref{thm:lowerKn} and the proof of Lemma \ref{lemma:K_n}  indicate that event $\cE$ occurs with probability at least $1-4e^{-2n}$ for all large $n$. Notice that $\cE$ does not rely on the randomness of test data $\x$.

We now consider $A=\E_\x[\x (K_n(\x,X)-K(\x,X))](K_n+\lambda\Id)^{-1}$ in Lemma \ref{lemma:quadratic_ybeta}. Conditioning on event $\cE$, we have
\begin{align}
   \|A\|_F^2\le ~& \E_\x\left[ \left\|\x (K_n(\x,X)-K(\x,X))^\top\right\|_F^2\right]\cdot\left\|(K_n+\lambda\Id)^{-1} X^\top\right\|^2\\
    \le~ & \|X\|^2\left\|(K_n+\lambda\Id)^{-1}\right\|^2\cdot\E_\x\left[\|\x\|^2\|K_n(\x,X)-K(\x,X)\|^2\right]\le C\frac{n}{d_1},\label{eq:bound_A}
\end{align}
for some constant $C$, where we utilize the assumption $\E[\|\x\|^2]=1$. Hence, based on Lemma \ref{lemma:quadratic_ybeta}, we know $\Var_{\veps,\,\bbeta^*}(I_1)\le cn/d_1$, for some constant $c$. By Chebyshev’s inequality and event $\cE$, 
\begin{equation}\label{eq:I_1-I_1}
    \P\left(|I_1-\bar I_1|\ge (n/d_1)^{\frac{1-\varepsilon}{2}}\right)\le  c\left(\frac{n}{d_1}\right)^{\varepsilon}+4e^{-2n},
\end{equation}
 for any $\varepsilon\in (0,1/2)$. Hence, $\left(d_1/n\right)^{\frac{1}{2}-\varepsilon}\cdot|I_1-\bar I_1|=o(1)$
with probability $1-o(1)$, when $n/d_1\to 0 $ and $n\to \infty$.  

Likewise, when $A=\E_\x[\x K(\x,X)]\left((K_n+\lambda\Id)^{-1}-(K+\lambda\Id)^{-1}\right)$, we can apply \eqref{eq:a-1_b-1} and 
\begin{equation}\label{eq:KxX_bound}
    \|K(\x,X)\|\le \|K(\tilde X, \tilde X)\|\le C\lambda_\sigma^2 B^2,
\end{equation} due to Lemma \ref{lem:Kupperbound} in Appendix~\ref{appendix:lemmas}, to obtain $\|A\|_F^2\le Cn/d_1$ conditionally on event $\cE$. Then, similarly, Lemma \ref{lemma:quadratic_ybeta} shows $\Var_{\veps,\,\bbeta^*}(I_2)\le cn/d_1$. Therefore, \eqref{eq:I_1-I_1} also holds for $|I_2-\Bar{I}_2|$.

Moreover, conditioning on the event $\cE$,
\begin{align}
    |\bar I_1|=~& \sigma_{\bbeta}^{2}\left|\E_\x\left[ (K_n(\x,X)-K(\x,X))(K_n+\lambda\Id)^{-1} X^\top \x\right]\right|\\ 
    \le~ &\sigma_{\bbeta}^{2}\E_\x\left[ \|\x\|\cdot\left\|K_n(\x,X)-K(\x,X)\right\|\cdot\|X\|\cdot\left\|(K_n+\lambda\Id)^{-1}\right\| \right],\\
    \le~ &\sigma_{\bbeta}^{2}\E_\x\left[ \|\x\|^2\right]^{\frac{1}{2}} \E_\x\left[\left\|K_n(\x,X)-K(\x,X)\right\|^2\right]^{\frac{1}{2}} \|X\| \left\|(K_n+\lambda\Id)^{-1}\right\| \le C\sqrt{\frac{n}{d_1}},\label{eq:I1_bar}
\end{align}
for some constant $C$. In the same way, with \eqref{eq:KxX_bound}, $|\bar I_2|\le C\sqrt{\frac{n}{d_1}}$ on the event $\cE$. Therefore, from \eqref{eq:E_2-E_2}, we can conclude $\left|\bar E_2 - E_2\right|=o\left(\left(n/d_1\right)^{1/2-\varepsilon}\right)$ for any $\varepsilon\in (0,1/2)$, with probability $1-o(1)$, when $n/d_1\to 0 $ and $n\to \infty$.

Analogously, the first term $\left|\bar E_1 - E_1\right|$ is controlled by the following four quadratic forms
\begin{equation*}
    \left|\bar E_1 - E_1\right|\le \sum_{i=1}^4\left|\v y^\top A_i\v y\right|,
\end{equation*}
where we define by $J_i:=\v y^\top A_i\v y$ for $1\le i\le 4$ and
\begin{align*}
    A_1 :=~&  (K_n+\lambda\Id)^{-1}\E_\x[K_n(\x,X)^\top \left(K_n(\x,X)-K(\x,X)\right)](K_n+\lambda\Id)^{-1},\\
    A_2 :=~& (K_n+\lambda\Id)^{-1}\E_\x[ \left(K_n(\x,X)-K(\x,X)\right)^\top K(\x,X)](K_n+\lambda\Id)^{-1},\\
    A_3 :=~& \left((K_n+\lambda\Id)^{-1}-(K+\lambda\Id)^{-1}\right)\E_\x[  K(\x,X)^\top K(\x,X)](K_n+\lambda\Id)^{-1},\\
    A_4 := ~& (K+\lambda\Id)^{-1} \E_\x[  K(\x,X)^\top K(\x,X)]\left((K_n+\lambda\Id)^{-1}-(K+\lambda\Id)^{-1}\right).
\end{align*}
Similarly with \eqref{eq:bound_A} and \eqref{eq:I1_bar}, it is not hard to verify $\|A_i\|_F\le C\sqrt{n/d_1}$ and $|\E_{\veps,\,\bbeta^*}[J_i]|\le C\sqrt{n/d_1}$ conditioning on the event $\cE$. Then, like \eqref{eq:I_1-I_1}, we can invoke Lemma \ref{lemma:quadratic_ybeta} for each $A_i$ to apply Chebyshev’s inequality and conclude $\left|\bar E_1- E_1\right|=o\left(\left(n/d_1\right)^{1/2-\varepsilon}\right)$ with probability $1-o(1)$ when $d_1/n\to \infty$, for any $\varepsilon\in (0,1/2)$.
\end{proof}

\begin{lemma}\label{lemma:KxXKxX}
With Assumptions \ref{assump:sigma} and  \ref{assump:testdata}, for $(\varepsilon_n, B)$-orthonormal $X$, we have that
\begin{align}
    \left\|\E_\x[K(\x,X)^\top K(\x,X)]-\frac{b_\sigma^4}{d_0}X^\top X\right\|\le &\left\|\E_\x[K(\x,X)^\top K(\x,X)]-\frac{b_\sigma^4}{d_0}X^\top X\right\|_F
    \le  C\sqrt{n}\eps_n^2,\label{eq:KxXKxX}\\
    \left\|\E_\x[\x K(\x,X)]-\frac{b_\sigma^2}{d_0} X\right\|\le &\left\|\E_\x[\x K(\x,X)]-\frac{b_\sigma^2}{d_0} X\right\|_F\le C\sqrt{n}\eps_n^2,\label{eq:KxX}
\end{align}
for some constant $C>0$.
\end{lemma}
\begin{proof} 
By Lemma \ref{lem:D3fan}, we have an entrywise approximation
\begin{equation*}
    |K(\x,\x_i)-b_\sigma^2\x^\top\x_i|\le C\lambda_\sigma \eps_n^2,
\end{equation*}for any $1\le i\le n$. Hence,
$
    \|K(\x, X)-b_\sigma^2\x^\top X\|\le C\lambda_\sigma \sqrt{n}\eps_n^2.
$
Assumption \ref{assump:testdata} of $\x$ implies that $\frac{b_\sigma^4}{d_0}X^\top X=b_\sigma^4\E_\x[X^\top \x \x^\top X]$. Then, we can verify \eqref{eq:KxXKxX} based on the following approximation
\begin{align*}
     &\left\|\E_\x[K(\x,X)^\top K(\x,X)]-\frac{b_\sigma^4}{d_0}X^\top X\right\|_F\le \E_\x\left[\left\|K(\x,X)^\top K(\x,X)-b_\sigma^4X^\top \x \x^\top X\right\|_F\right]\\
     \le ~&\E_\x\left[\left\|K(\x,X)^\top \left(K(\x,X)-b_\sigma^2 \x^\top X\right)\right\|_F+b_\sigma^2\left\|\left(K(\x,X)^\top-b_\sigma^2X^\top \x\right)\x^\top X\right\|_F\right]\\
    \le ~& \E_\x\left[\|K(\x, X)-b_\sigma^2\x^\top X\|\left(\|K(\x,X)\| +\|b_\sigma^2\x^\top X\|\right)\right]\le  C\sqrt{n}\eps_n^2,
\end{align*}for some universal constant $C$. The same argument can also be employed to prove \eqref{eq:KxX}, so details will be omitted here.
\end{proof}

\begin{proof}[Proof of Theorem \ref{thm:limit_error}]
From \eqref{eq:ETrain} and \eqref{eq:test_diff}, we can easily conclude that 
\begin{align}
    E_{\train}^{(RF,\lambda)}-E_{\train}^{(K,\lambda)}&\overset{\P}{\to} 0,\label{eq:RF-K_train}\\
    \mathcal L(\hat{f}_{\lambda}^{(RF)}(\x))-\mathcal L (\hat{f}_{\lambda}^{(K)}(\x))&\overset{\P}{\to} 0,\label{eq:RF-K_test}
\end{align}
as $n\to\infty$ and $n/d_1\to0$. 
Therefore, to study the training error $E_{\train}^{(RF,\lambda)}$ and the test error $\mathcal L(\hat{f}_{\lambda}^{(RF)}(\x))$ of random feature regression,
it suffices to analyze the asymptotic behaviors of $E_{\train}^{(K,\lambda)}$ and $\mathcal L (\hat{f}_{\lambda}^{(K)}(\x))$ for the kernel regression, respectively. In the rest of the proof, we will first analyze the test error $\mathcal L (\hat{f}_{\lambda}^{(K)}(\x))$ and then compute the training error $E_{\train}^{(K,\lambda)}$ under the ultra-wide regime.

Recall that $K_\lambda=(K+\lambda\Id)$ and the test error is given by
\begin{align}
    \mathcal L(\hat f^{(K)}_\lambda)=~ &\frac{1}{d_0}\|\bbeta^*\|^2+L_1-2L_2,\label{eq:expansion_test}
\end{align}
where 
$L_1:=\v y^\top K_\lambda^{-1}\E_\x[K(\x,X)^\top K(\x,X)]K_\lambda^{-1}\v y$, $L_2:=\bbeta^{*\top}\E_\x[\x K(\x,X)]K_\lambda^{-1}\v y$.
The spectral norm of $K_\lambda $ is bounded from above and the smallest eigenvalue is bounded from below by some positive constants. 

We first focus on the last two terms $L_1$ and $L_2$ in the test error. Let us define
\[\widetilde L_1:=\frac{b_\sigma^4}{d_0}\v y^\top K_\lambda^{-1}X^\top X K_\lambda^{-1}\v y\quad\text{and}\quad \widetilde L_2:=\frac{b_\sigma^2}{d_0}\bbeta^{*\top}X K_\lambda^{-1}\v y.\]
Then, we obtain two quadratic forms
\begin{align*}
    L_1-\widetilde L_1= & \v y^\top K_\lambda^{-1}\left(\E_\x[K(\x,X)^\top K(\x,X)]-\frac{b_\sigma^4}{d_0}X^\top X\right)K_\lambda^{-1}\v y=:\v y^\top A_1 \v y,\\
    L_2-\widetilde L_2= & \bbeta^{*\top}\left(\E_\x[\x K(\x,X)]-\frac{b_\sigma^2}{d_0} X\right)K_\lambda^{-1}\v y=:\bbeta^{*\top} A_2\v y,
\end{align*}
where $\|A_1\|_F$ and $\|A_2\|_F$ are at most $ C\sqrt{n}\eps_n^2$ for some constant $C>0$, due to Lemma \ref{lemma:KxXKxX}. Hence, applying Lemma \ref{lemma:quadratic_ybeta} for these two quadratic forms, we have $\Var(L_i-\widetilde L_i)\le cn\eps_n^4\to 0$ as $n\to\infty$. Additionally, Lemma \ref{lemma:quadratic_ybeta} and the proof of Lemma~\ref{lemma:KxXKxX} verify that $\E[\v y^\top A_1 \v y]$ and $\E[\bbeta^{*\top} A_2\v y]$ are vanishing as $n\to\infty$. Therefore, $L_i-\widetilde L_i$ converges to zero in probability  for $i=1,2$. So we can move to analyze $\widetilde L_1$ and $\widetilde L_2$ instead. Copying the above procedure, we can separately compute the variances of $\widetilde L_1$ and $\widetilde L_2$ with respect to $\bbeta^*$ and $\veps$, and then apply Lemma \ref{lemma:quadratic_ybeta}. Then, $|\widetilde L_1-\bar L_1|$ and $|\widetilde L_2-\bar L_2|$ will converge to zero in probability as $n,d_0\to \infty$, where 
\begin{align*}
    \bar L_1:= ~&\E_{\veps,\,\bbeta^*}[\widetilde L_1]=\frac{b_\sigma^4\sigma_{\bbeta}^2 n}{d_0}\tr K_\lambda^{-1}X^\top X K_\lambda^{-1}X^\top X+\frac{b_\sigma^4\sigma_{\veps}^2 n}{d_0}\tr K_\lambda^{-1}X^\top X K_\lambda^{-1},\\
    \bar L_2:= ~&\E_{\veps,\,\bbeta^*}[\widetilde L_2]=\frac{b_\sigma^2\sigma_{\bbeta}^2n}{d_0}\tr  K_\lambda^{-1}X^\top X.
\end{align*}
To obtain the last approximation, we define $\bar K(X,X):=b_\sigma^2X^\top X+(1-b_\sigma^2)\Id$ and
\begin{equation}\label{eq:def_Klambda}
    \bar K_\lambda:=b_\sigma^2X^\top X+(1+\lambda-b_\sigma^2)\Id.
\end{equation}
We aim to replace $K_\lambda$ by $\bar K_\lambda$ in $\bar L_1$ and $\bar L_2$. Recalling the identity \eqref{eq:a-1_b-1}, we have
\begin{align*}
    K_\lambda^{-1}-\bar K_\lambda^{-1}=\bar K_\lambda^{-1}\left(K(X,X)-\bar K(X,X) \right)K_\lambda^{-1}.
\end{align*}
Since $\sigma$ is not a linear function, $1-b_\sigma^2>0$. Then, with \eqref{eq:lambda_min}, the proof of Lemma \ref{lemma:spectrum_Phi} indicates
\begin{align}\label{eq:frob_K}
     \left\|K_\lambda^{-1}-\bar K_\lambda^{-1}\right\|_F\le C\sqrt{n^2\eps_n^4+n\eps_n^2},
\end{align}
where we apply the fact that $\lambda_{\min}(\bar K(X,X))\ge 1-b_\sigma^2>0$. Let us denote
\begin{align}
    L_1^0:=~&\frac{b_\sigma^4\sigma_{\bbeta}^2 n}{d_0}\tr \bar K_\lambda^{-1}X^\top X \bar K_\lambda^{-1}X^\top X+\frac{b_\sigma^4\sigma_{\veps}^2 n}{d_0}\tr \bar  K_\lambda^{-1}X^\top X \bar K_\lambda^{-1},\label{eq:defL1}\\
    L_2^0:=~& \frac{b_\sigma^2\sigma_{\bbeta}^2n}{d_0}\tr  \bar K_\lambda^{-1}X^\top X.\label{eq:defL2}
\end{align}
Notice that for any matrices $A,B\in\R^{n\times n}$,
$\|AB\|_F\le \|A\|\|B\|_F,  |\Tr(AB)|\le \|A\|_F\|B\|_F.$ Then, with the help of \eqref{eq:frob_K} and  uniform bounds of the spectral norms of $X^\top X$, $K_\lambda^{-1}$ and $\bar K_\lambda^{-1}$, we obtain that
\begin{align*}
    |\bar L_1-L^0_1|
    \le~ & \frac{b_\sigma^4\sigma_{\bbeta}^2 }{d_0}\left|\Tr K_\lambda^{-1}X^\top X  (K_\lambda^{-1}-\bar K_\lambda^{-1})X^\top X\right|+\frac{b_\sigma^4\sigma_{\bbeta}^2 }{d_0}\left|\Tr (K_\lambda^{-1}-\bar K_\lambda^{-1})X^\top X \bar K_\lambda^{-1}X^\top X\right|\\
    ~&+\frac{b_\sigma^4\sigma_{\veps}^2 }{d_0}\left|\Tr (K_\lambda^{-1}-\bar  K_\lambda^{-1})X^\top X \bar K_\lambda^{-1}\right|+\frac{b_\sigma^4\sigma_{\veps}^2 }{d_0}\left|\Tr K_\lambda^{-1} X^\top X(K_\lambda^{-1}-\bar  K_\lambda^{-1})\right| \\
    \le~ & \frac{C\sqrt{n}}{d_0}\left\|K_\lambda^{-1}-\bar K_\lambda^{-1}\right\|_F\le C\frac{n}{d_0}\sqrt{n\eps_n^4+\eps_n^2}\to 0,
\end{align*}
as $n\to \infty $, $n/d_0\to \gamma$ and $n\eps_n^4\to 0$. Combining all the approximations, we conclude that $L_i$ and $L_i^0$ have identical limits in probability for $i=1,2$. On the other hand, based on the assumption of $X$ and definitions in \eqref{eq:def_Klambda}, \eqref{eq:defL1} and \eqref{eq:defL2}, it is not hard to check that
\begin{align*}
    \lim_{n\to\infty} L_1^0=~& b_\sigma^4\sigma_{\bbeta}^2\gamma \int_{\mathbb R} \frac{x^2}{(b_\sigma^2x+1+\lambda-b_\sigma^2)^2} d\mu_0(x)+b_\sigma^4\sigma_{\veps}^2\gamma \int_{\R}\frac{x}{(b_\sigma^2x+1+\lambda-b_\sigma^2)^2} d\mu_0(x), \\
    \lim_{n\to\infty} L_2^0=~&b_\sigma^2\sigma_{\bbeta}^2\gamma \int_{\R}\frac{x}{b_\sigma^2x+1+\lambda-b_\sigma^2} d\mu_0(x).
\end{align*}
Therefore, $L_1$ and $L_2$ converge in probability to the above limits, respectively, as $n\to \infty$. In the end, we apply the concentration of the quadratic form $\bbeta^{*\top}\bbeta^*$ in \eqref{eq:expansion_test} to get $\frac{1}{d_0}\|\bbeta^*\|^2\xrightarrow[]{\P}\sigma^2_{\bbeta}$. Then, by \eqref{eq:RF-K_test}, we can get the limit in \eqref{eq:test_limit} for the test error $\mathcal L(\hat f^{(RF)}_\lambda)$. As a byproduct, we can even use $L_1^0$ and $L_2^0$ to form an $n$-dependent deterministic equivalent of $\mathcal L(\hat f^{(RF)}_\lambda)$ as well.

Thanks to Lemma \ref{lemma:quadratic_ybeta}, the training error, 
$E_{\train}^{(K,\lambda)}=\frac{\lambda^2}{n}\v y^\top K_\lambda^{-2} \v y$, analogously, concentrates around its expectation with respect to $\bbeta^*$ and $\veps$, which is $\sigma_{\bbeta}^2\lambda^2\tr K_\lambda^{-2}X^\top X +\sigma_{\veps}^2\lambda^2\tr K_\lambda^{-2}$. Moreover, because of \eqref{eq:frob_K}, we can further substitute $ K_\lambda^{-2}$ by $\bar K_\lambda^{-2}$ defined in \eqref{eq:def_Klambda}. Hence, we know that, asymptotically,
\[\left|E_{\train}^{(K,\lambda)}-\sigma_{\bbeta}^2\lambda^2\tr \bar K_\lambda^{-2}X^\top X-\sigma_{\veps}^2\lambda^2\tr \bar K_\lambda^{-2}\right|\xrightarrow[]{\P} 0,\]
where as $n,d_0\to \infty $,
\begin{align}
    \lim_{n\to \infty} \sigma_{\bbeta}^2\lambda^2\tr \bar K_\lambda^{-2}X^\top X=~&\sigma_{\bbeta}^2\lambda^2\int_{\R} \frac{x}{(b_\sigma^2x+1+\lambda-b_\sigma^2)^2}d\mu_0(x),\\
    \lim_{n\to \infty} \sigma_{\veps}^2\lambda^2\tr \bar K_\lambda^{-2}=~&\sigma_{\veps}^2\lambda^2\int_{\R} \frac{1}{(b_\sigma^2x+1+\lambda-b_\sigma^2)^2}d\mu_0(x).
\end{align}
The last two limits are due to $\mu_0=\limspec X^\top X$ as $n,d_0\to \infty $. Therefore, by \eqref{eq:RF-K_train}, we obtain our final result \eqref{eq:train_limit} in Theorem \ref{thm:limit_error}.
\end{proof}

\subsection*{Acknowledgements} 
Z.W. is partially supported by NSF DMS-2055340 and NSF DMS-2154099. Y.Z. is partially supported by  NSF-Simons Research Collaborations on the Mathematical and Scientific
Foundations of Deep Learning.
This material is based upon work supported by the National Science Foundation under Grant No. DMS-1928930 while Y.Z. was in residence at the Mathematical Sciences Research Institute in Berkeley, California, during the Fall 2021 semester for the program “Universality and Integrability in Random Matrix Theory and Interacting Particle Systems”. Z.W. would like to thank Denny Wu for his valuable suggestions and comments. Both authors would like to thank Lucas Benigni, Ioana Dumitriu, and  Kameron Decker Harris for their helpful discussion.

\appendix
\addcontentsline{toc}{section}{Appendices}

\section{Auxiliary lemmas}\label{appendix:lemmas}

\begin{lemma}[Equation (3.7.9) in \cite{johnson1990matrix}] \label{lem:Hadamardinequality}
Let $A, B$ be two $n\times n$ matrices, $A$ be positive semidefinite, and $A\odot B$ be the Hadamard product between $A$ and $B$. Then,
\begin{equation}
    \|A\odot B\|\le \max_{i,j}|A_{ij}|\cdot\|B\|.
\end{equation}
\end{lemma}

\begin{lemma}[Sherman–Morrison formula, \cite{bartlett1951inverse}] \label{lem:SMformula}
Suppose $A\in \R^{n\times n}$ is an invertible square matrix and $\mathbf u,\mathbf v\in \R^n$ are column vectors. Then
\begin{align}
    (A+\mathbf u \mathbf v^\top)^{-1}=A^{-1}-\frac{A^{-1}\mathbf u\mathbf v^\top A^{-1}}{1+\mathbf v^\top A^{-1}\mathbf u}.
\end{align}
\end{lemma}

\begin{lemma}[Theorem A.45 in \cite{bai2010spectral}]\label{lem:BS10A45}
Let $A, B$ be two $n\times n$ Hermitian matrices. Then $A$ and $B$ have the same limiting spectral distribution if $\|A-B\|\to 0$ as $n\to\infty$.
\end{lemma}

\begin{lemma}[Theorem B.11 in \cite{bai2010spectral}]\label{lem:BS10B11}
Let $z=x+iv\in \C, v>0$ and $s(z)$ be the Stieltjes transform of a probability measure. Then
$
    |\Re s(z)|\leq v^{-1/2}\sqrt{\Im s(z)}
$.
\end{lemma}

\begin{lemma}[Lemma D.2 in \cite{nguyen2020global}] \label{lem:NM20D2} Let $\x, \y\in \R^d$ such that $\|\x\|=\| \y\|=1$ and $\w\sim \N(0, I_{d})$. Let $h_j$ be the $j$-th normalized Hermite polynomial given in  \eqref{eq:hermitepolynomial}. Then 
\begin{align}
    \E_{\w}[h_j(\langle \w,\x\rangle)h_k(\langle \w,\y\rangle)  ]= \delta_{jk} \langle \x, \y \rangle ^k.
\end{align}
\end{lemma}

\begin{lemma}[Proposition C.2 in \cite{fan2020spectra}]\label{lem:propC2}
Suppose $M=U+iV\in \C^{n\times n}$, $U,V$ are real symmetric, and $V$ is invertible with $\sigma_{\min}(V)\geq c_0>0$. Then $M$ is invertible with $\sigma_{\min}(M)\geq c_0$.
\end{lemma}

\begin{lemma}[Proposition C.3 in \cite{fan2020spectra}]\label{lem:ESDforbenius}  Let $M,\widetilde{M}$  be two sequences of $n\times n$ Hermitian matrices satisfying 
\[\frac{1}{n} \|M-\widetilde{M}\|_F^2\to 0
\] as $n\to\infty$. Suppose that, as $n\to\infty$, $\limspec M=\nu$ for a probability distribution $\nu$ on $\R$, then $\limspec \widetilde{M}=\nu$.
\end{lemma}

\begin{lemma}\label{lem:D3fan}
Recall the definition of $\Phi$ in \eqref{def:phi}. Under Assumption~\ref{assump:sigma}, if $X$ is $(\varepsilon,B)$-orthonormal with sufficiently small $\varepsilon$, then for a universal constant $C>0$ and any $\alpha\neq\beta\in [n]$, we have
\begin{align}
    | \Phi_{\alpha\beta}-b_{\sigma}^2 \x_{\alpha}^\top \x_{\beta}|\leq ~&C \varepsilon^2,\\
    |\E_{\w}[\sigma(\w^\top\x_{\alpha})]|\leq ~&C \varepsilon.
\end{align}
\end{lemma}
\begin{proof}
 When $\sigma$ is twice differentiable in Assumption~\ref{assump:sigma}, this result follows from Lemma D.3 in \cite{fan2020spectra}. When $\sigma$ is a piece-wise linear function defined in case \ref{assump:sigma_relu} of Assumption~\ref{assump:sigma}, the second inequality follows from \eqref{eq:zeta_0_appr_00} with $t=\|\x_{\a}\|$. For the first inequality, the Hermite expansion of $\Phi_{\alpha\beta}$ is given by \eqref{eq:Phiab} with coefficients $\zeta_k(\sigma_\a)=\E[\sigma(\|\x_{\a}\|\xi)h_k(\xi)]$ for $k\in\mathbb{N}$. Observe that the piece-wise linear function in case \ref{assump:sigma_relu} of Assumption~\ref{assump:sigma} satisfies
 \begin{align}
     \zeta_k(\sigma_\a)=~&\|\x_{\a}\|\zeta_k(\sigma), ~\text{ for }~k\ge 1,\\
     \zeta_0(\sigma_\a)=~& b(1-\|\x_{\a}\|),
 \end{align}
 because of condition \eqref{eq:conditionsigma} for $\sigma$. Recall $\mathbf u_\a=\x_\a/\|\x_\a\|$ and $\zeta_1(\sigma)=b_\sigma$. Then, analogously to the derivation of \eqref{eq:off_diagonal_bounds}, there exists some constant $C>0$ such that
 \begin{align}
     |\Phi_{\alpha\beta}-b_{\sigma}^2 \x_{\alpha}^\top \x_{\beta}|=~ & \left|\sum_{k\neq 1} \zeta_k(\sigma_\a)\zeta_k(\sigma_\b) (\mathbf u_\a^\top \mathbf u_\b)^k\right|\\
     \le ~& b^2(1-\|\x_{\a}\|)(1-\|\x_{\b}\|) + \frac{|\x_{\a}^\top\x_{\b}|^2}{\|\x_{\a}\|\|\x_{\b}\|}\|\sigma\|_{L^2}^2 \le C \varepsilon^2,
 \end{align}
for $\varepsilon\in(0,1)$ and $(\varepsilon,B)$-orthonormal $X$. This completes the proof of this lemma.
\end{proof}
With the above lemma, the proof of Lemma D.4 in \cite{fan2020spectra} directly yields the following lemma. 
\begin{lemma}\label{lem:Kupperbound} 
Under the same assumptions as Lemma~\ref{lem:D3fan}, there exists a constant $C>0$ such that 
$\|K(X,X)\|\leq C B^2$. Additionally, with Assumption~\ref{assump:testdata}, we have $\|K(\tilde X,\tilde X)\|\leq C  B^2$.
\end{lemma}

\section{Additional simulations}\label{appendix:simulation}
Figures \ref{fig:2} and \ref{fig:4} provide additional simulations for the eigenvalue distribution described in Theorem \ref{thm:law_DNN} with different activation functions and scaling. Here, we compute the empirical eigenvalue distributions of centered CK matrices in histograms and the limiting spectra in terms of self-consistent equations. All the input data $X$'s are standard random Gaussian matrices. Interestingly, in Figure~\ref{fig:4}, we observe an outlier that emerges outside the bulk distribution for the piece-wise linear activation function defined in case \ref{assump:sigma_relu} of Assumption~\ref{assump:sigma}. The analysis of the emergence of the outlier, in this case, would be interesting for future work.
\begin{figure}[!ht]
\includegraphics[width=0.32\textwidth]{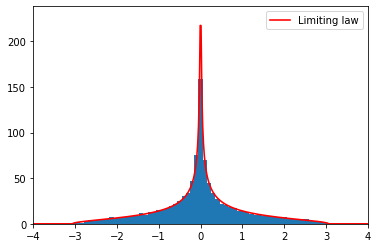}
\includegraphics[width=0.32\textwidth]{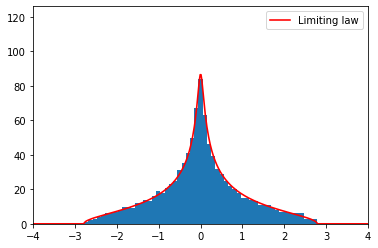}
\includegraphics[width=0.32\textwidth]{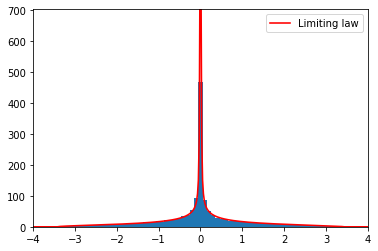}\\
\includegraphics[width=0.32\textwidth]{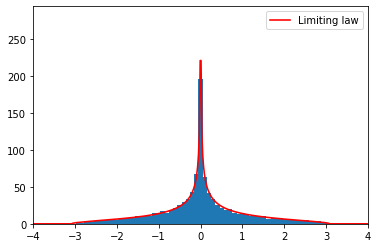}
\includegraphics[width=0.32\textwidth]{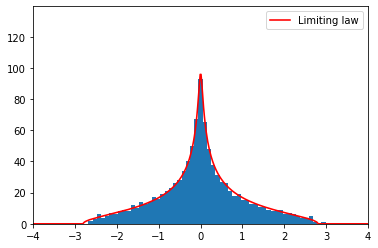}
\includegraphics[width=0.32\textwidth]{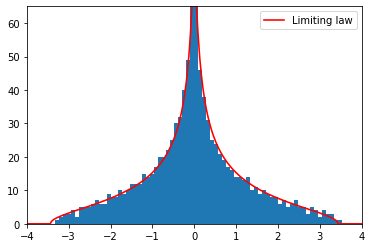}
\caption{\small Simulations for empirical eigenvalue distributions of \eqref{eq:center_RF2} and theoretical predication (red curves) of the limiting law $\mu$ with activation functions $\sigma(x)\propto$ Sigmoid function (first row) and $\sigma(x)=x$ linear function (second row) satisfying Assumption \ref{assump:sigma}: $n=10^3$, $d_0=10^3$ and $d_1=10^5$ (left); $n=10^3$, $d_0=1.5\times 10^3$ and $d_1=10^5$ (middle); $n=1.5\times 10^3$, $d_0=10^3$ and $d_1=10^5$ (right).}\label{fig:2}
\end{figure}

\newpage
\begin{figure}[!ht]
\includegraphics[width=0.32\textwidth]{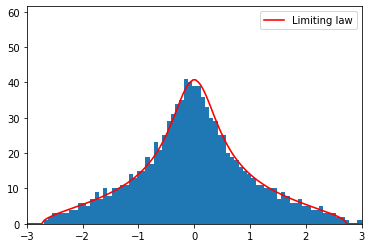}
\includegraphics[width=0.32\textwidth]{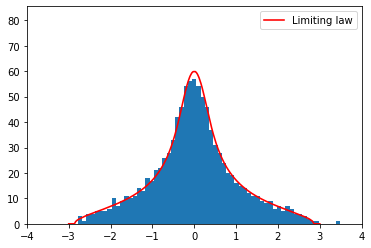}
\includegraphics[width=0.32\textwidth]{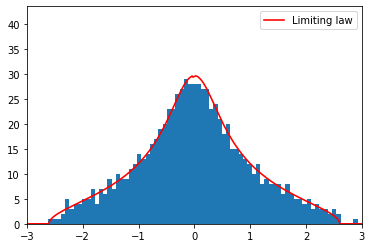}
\caption{\small Simulations for empirical eigenvalue distributions of \eqref{eq:center_RF2} and theoretical predication (red curves) of the limiting law $\mu$ where activation function $\sigma(x)\propto$ ReLU function satisfies case \ref{assump:sigma_relu} of Assumption~\ref{assump:sigma}: $n=10^3$, $d_0=10^3$ and $d_1=10^5$ (left); $n=10^3$, $d_0=800$ and $d_1=10^5$ (middle); $n=800$, $d_0=10^3$ and $d_1=10^5$ (right).}\label{fig:4}
\end{figure}

\bibliographystyle{alpha}
\bibliography{ref.bib}

\end{document}